\documentclass[oneside,a4paper,12pt]{book}

\usepackage{setspace}
\usepackage[left=40mm,top=30mm,right=25mm,bottom=25mm]{geometry}

\usepackage{fancyhdr}
\usepackage[pagebackref,colorlinks]{hyperref}

\usepackage{amsthm}

\theoremstyle{plain}

\renewenvironment{proof}[1][Proof.]{\begin{trivlist}
\item[\hskip \labelsep {\bfseries #1}]}{$\hfill \Box$
\end{trivlist}}

\newtheorem{thm}[subsection]{Theorem}
\newtheorem{lemma}[subsection]{Lemma}
\newtheorem{cor}[subsection]{Corollary}
\newtheorem{prop}[subsection]{Proposition}

\newtheorem{predefinitions}[subsection]{Definitions}
    {\begin{predefinitions}\upshape}{\end{predefinitions}}
\newtheorem{predefinition}[subsection]{Definition}
  \newenvironment{defin}%
    {\begin{predefinition}\upshape}{\end{predefinition}}
\newtheorem{preexample}[subsection]{Example}
  \newenvironment{example}%
    {\begin{preexample}\upshape}{\end{preexample}}
\newtheorem{preexamples}[subsection]{Examples}
  \newenvironment{examples}%
    {\begin{preexamples}\upshape}{\end{preexamples}}
\newtheorem{prenote}[subsection]{Note}
    {\begin{prenote}\upshape}{\end{prenote}}
\newtheorem{prequestion}[subsection]{Question}
  \newenvironment{ques}%
    {\begin{prequestion}\upshape}{\end{prequestion}}

\newtheorem{preremark}[subsection]{Remark}
  \newenvironment{remark}%
    {\begin{preremark}\upshape}{\end{preremark}}

\usepackage[all]{xy}
\usepackage{amsfonts}
\usepackage[mathcal]{eucal}
\usepackage{eufrak}
\usepackage{amssymb}
\usepackage{amsmath}
\usepackage{mathrsfs} 
\usepackage{verbatim} 
\usepackage{graphicx}
\usepackage{makeidx} 
\usepackage{url} 
\usepackage{mdwlist} 

\newcommand{\lmps}{\longmapsto}
\newcommand{\mps}{\mapsto}
\newcommand{\ra}{\rightarrow}
\newcommand{\Ra}{\Rightarrow}
\newcommand{\nono}{\nonumber}
\newcommand{\Lra}{\Leftrightarrow}
\newcommand{\mi}{\setminus}
\newcommand{\lra}{\longrightarrow}
\newcommand{\ov}{\overline}

\newcommand{\mal}{\mathcal{L}}
\newcommand{\mar}{\mathcal{R}}

\newcommand{\pr}{\mathcal{P}\mathrm{r}}
\newcommand{\Proj}{\mathcal{P}\mathrm{roj}}
\newcommand{\Mod}{\textrm{-}\mathcal{M}\mathfrak{od}}

\newcommand{\I}{\mathcal{I}}
\newcommand{\G}{\mathcal{G}}

\newcommand{\R}{\mathbb R}
\newcommand{\Z}{\mathbb Z}
\newcommand{\C}{\mathbb C}
\renewcommand{\H}{\mathbb H}
\newcommand{\+}{\oplus}

\newcommand{\conggr}{\cong_{\mathrm{gr}}}
\newcommand{\gr}{\mathrm{gr}}
\newcommand{\Pgr}{\mathcal P \mathrm{gr}}
\newcommand{\hygrmod}{\textrm{-}\mathfrak{gr}\textrm{-}\mathcal{M}\mathfrak{od}}

\mathchardef\mhyphen="2D

\newcommand{\va}{\varphi}
\newcommand{\Ga}{\Gamma}
\newcommand{\ga}{\gamma}
\newcommand{\al}{\alpha}

\newcommand{\de}{\delta}
\newcommand{\la}{\lambda}

\newcommand{\si}{\sigma}
\newcommand{\ep}{\varepsilon}
\newcommand{\GL}{\operatorname{GL}}
\newcommand{\Det}{\operatorname{Det}}

\newcommand\CK[1][1]{\operatorname{CK}_{#1}}

\newcommand\ZK[1][1]{\operatorname{ZK}_{#1}} 
\newcommand\SK[1][1]{\operatorname{SK}_{#1}} 
\newcommand{\Nrd}[1][{}]{{\operatorname{Nrd}_{#1}}} 
\newcommand{\Trd}[1][{}]{{\operatorname{Trd}_{#1}}}
\newcommand{\coker}[1][{}]{{\operatorname{coker}_{#1}}}
\newcommand{\St}{\operatorname{St}}

\newcommand{\ind}{\operatorname{ind}}

\newcommand{\End}{\operatorname{End}}
\newcommand{\im}{\operatorname{im}}
\newcommand{\Hom}{\operatorname{Hom}}
\newcommand{\Max}{\operatorname{Max}}
\newcommand{\id}{\operatorname{id}}
\newcommand{\Ann}{\operatorname{Ann}}
\newcommand{\op}{\operatorname{op}}
\newcommand{\HOM}{\operatorname{HOM}}
\newcommand{\END}{\operatorname{END}}
\newcommand{\Supp}{\operatorname{Supp}}
\newcommand{\chr}{\operatorname{char}}
\renewcommand{\Im}{\operatorname{Im}}

\newcommand{\rank}{\operatorname{rank}}
\newcommand{\Obj}{\operatorname{Obj}}
\newcommand{\Tr}{\operatorname{Tr}}

\newcommand{\m}{\mathrm{m}}

\makeindex

\begin{document}




\begin{titlepage}
\vspace*{2cm}
\centering{\Huge \textit{K}-Theory of Azumaya Algebras}\\
\vspace{6cm}

\centering{\it {\large A thesis for the degree of Doctor of Philosophy\\
\vspace*{3pt}
 submitted to\\
 \vspace*{3pt}
 the School of Mathematics and Physics\\
 \vspace*{3pt}
 of \\
 \vspace*{3pt}
 Queen's University Belfast.}}\\

\vspace{4cm}
\Large{Judith Ruth Millar MSci}\\
\large{September 2010} \thispagestyle{empty}
\end{titlepage}


\chapter*{Acknowledgements} \setcounter{page}{0} \markboth{}{}
\thispagestyle{empty}

I would like to express my gratitude to my supervisor Dr.\ Roozbeh
Hazrat for his advice, encouragement, patience and the many hours of
his time over the last four years. \\

\noindent I would also like to thank the members of the Department
of Pure Mathematics, and especially the postgraduate students I have
met here at Queen's.\\

\noindent I acknowledge the financial assistance I received from the
Department of Education and Learning. \\

\noindent Finally, I would like to thank my family for their support
throughout my university life.


\tableofcontents \markboth{Contents}{}


\chapter*{Notation\markboth{Notation}{}}



\thispagestyle{plain} \addcontentsline{toc}{chapter}{Notation}

Throughout all rings are assumed to be associative and any ring $R$
has a multiplicative identity element $1_R$. A subring $S$ of $R$
contains the identity element of $R$. We assume that a ring
homomorphism $R \ra R'$ takes the identity of $R$ to the identity of
$R'$.

Modules over a ring $R$ are assumed to be left $R$-modules unless
otherwise stated. For an $R$-module $M$, we assume that $1_R m = m$
for every $m \in M$. For a homomorphism between $R$-modules, we will
use the terms $R$-module homomorphism and $R$-linear homomorphism
interchangeably.

For a field $F$, a division algebra $D$ over $F$ is defined to be a
division ring with centre $F$ such that $[D:F] < \infty$.

\subsection*{Some symbols used}

\hspace{2.7ex} For a ring $R$,
\begin{basedescript}{\desclabelstyle{\pushlabel}\desclabelwidth{2.7cm}}
\item[$R^{*}$] the group of units of $R$; that is, the elements of
$R$ which have a multiplicative inverse;

\item[$R \Mod$] the category of $R$-modules with $R$-module
homomorphisms between them;

\item[$\pr (R)$] the subcategory of $R \Mod$ consisting of
finitely generated projective $R$-modules.
%
\end{basedescript}

%

For a multiplicative group $G$ and $x, y \in G$,
\begin{basedescript}{\desclabelstyle{\pushlabel}\desclabelwidth{2.7cm}}
\item[\mbox{$[\, x,y \,]$}] the commutator $xyx^{-1}y^{-1}$;

\item[\mbox{$[\hspace{.2ex} G,G \hspace{.2ex} ]$}$ =G'$] the
\emph{commutator subgroup of $G$}\index{commutator subgroup}; that
is, the (normal) subgroup of $G$ generated by the commutators;

\item[$G/G'$] the \emph{abelianisation\index{abelianisation}
of $G$}.
\end{basedescript}


\setcounter{chapter}{-1}
\chapter{Introduction}

\thispagestyle{plain}



For a division algebra $D$ finite dimensional over its centre $F$,
consider the multiplicative group $D^* = D \mi 0$. The structure of
subgroups of $D^*$ is not known in general. In 1953, Herstein
\cite{hersteinsubgroups} proved that if $D$ is of characteristic $p
\neq 0$, then every finite subgroup of $D^*$ is cyclic. This is an
easy result in the setting of fields (see
\cite[Thm.~V.5.3]{hungerford}), so the finite subgroups of such
division algebras behave in a similar way to those of fields. We
give an example here to show that this result doesn't necessarily
hold for division algebras of characteristic zero. Hamilton's
quaternions $\mathbb H$ form a division algebra of characteristic
zero, but the subgroup $\{\pm 1, \pm i , \pm j , \pm k \}$ is a
finite subgroup of $\mathbb H^*$ which is not abelian, so not
cyclic.

In 1955, Amitsur classified all finite subgroups of $D^*$ in his
influential paper \cite{amitsursubgroups}. Since then, subgroups of
the group $D^*$ have been studied by a number of people (see for
example \cite{faudree,greenfield,lam1,lichtman}). Herstein
\cite{hersteinconjugates} showed that a non-central element in a
division algebra has infinitely many conjugates. Since normal
subgroups are invariant under conjugation, this shows that
non-central normal subgroups are ``big'' in $D^*$. Note that if $D$
is a non-commutative division algebra, then non-trivial non-central
normal subgroups exist in $D^*$. For example $D'$, the subgroup of
$D^*$ generated by the multiplicative commutators, is non-trivial
and is a non-central normal subgroup of $D^*$.


So as for normal subgroups, we could ask: how large are maximal
subgroups of $D^*$? But it remains as an open question whether
maximal subgroups even exist in $D^*$. The existence of maximal
subgroups of $D^*$ is connected with the non-triviality of $K$-group
$\CK[1](D)$. For a division algebra $D$ with centre $F$, we note
that $\CK[1](D) \cong D^*/(F^* D')$. The group $\CK[1](D)$ is
related to algebraic $K$-theory; more specifically, to the functor
$K_1$. Before discussing the group, we indicate how algebraic
$K$-theory has developed.

Algebraic $K$-theory defines a sequence of functors $K_i$ from the
category of rings to the category of abelian groups. For the lower
$K$-groups, the functor $K_0$ was introduced in the mid-1950s by
Grothendieck, and the functor $K_1$ was developed in the 1960s by
Bass. Many attempts were made to extend these functors to cover all
$K_i$ for $i \geq 0$. Milnor defined the functor $K_2$ in the 1960s,
but it was not clear how to construct the higher $K$-functors. The
functor $K_2$ is defined in such a way that there is an exact
sequence linking it with $K_0$ and $K_1$ (see
\cite[Thm.~4.3.1]{rosenberg}).  The ``correct'' definition of the
higher $K$-functors was required not only to provide such an exact
sequence connecting the functors, but also to cover the given
definitions of $K_0$ and $K_1$. Then in 1974, Quillen gave two
different constructions of the higher $K$-functors,
which are equivalent for rings and which satisfy these expected
properties (see \cite[Ch.~5]{rosenberg}).

It is straightforward to describe the lower $K$-groups concretely.
The group $K_0$ can be considered as the group completion of the
monoid of isomorphism classes of finitely generated projective
modules, and $K_1$ is the abelianisation of the infinite general
linear group (see Chapter~\ref{chktheory} for the details). The
higher $K$-groups are considerably more difficult to compute. They
are, however, functorial in construction.

Returning to the setting of division algebras, consider a central
simple algebra $A$. Then by the Artin-Wedderburn Theorem, $A$ is
isomorphic to a matrix $M_n (D)$ over a division algebra $D$. Let
$F$ be the centre of $D$. Since each $K_i$, $i \geq 0$, is a functor
from the category of rings to the category of abelian groups, the
inclusion map $F \ra A$ induces a map $K_i (F) \ra K_i (A)$. Let
$\ZK[i] (A)$ denote the kernel of this map and $\CK[i] (A)$ denote
the cokernel. This gives an exact sequence
\begin{equation} \label{exactseqeqn}
1 \ra \ZK[i] (A) \ra K_i(F) \ra K_i(A) \ra \CK[i](A) \ra 1.
\end{equation}
So the group $\CK[1] (A)$ is defined to be $\coker \big(K_1 (F) \ra
K_1 (A) \big)$. Then it can be shown that $\CK[1](A)$ is isomorphic
to $D^*/ F^{*n} D'$, and it is a bounded torsion abelian group (see
Section~\ref{sectionlowerkgroupsofcsalgebras} for the details).

For a division algebra $D$, if $\CK[1] \hspace{-.5ex} \big(M_n
(D)\big)$ is not the trivial group for some $n \in \mathbb N$, then
$D^*$ has a normal maximal subgroup (see
\cite[Section~2]{hazwadmax}). It has been conjectured that if
$\CK[1](D)$ is trivial, then $D$ is a quaternion algebra (see
\cite[p.~408]{hazratmirzaii}). One of the most significant results
in this direction proves that if $D$ is a division algebra with
centre $F$ such that $D^*$ has no maximal subgroups, then $D$ and
$F$ satisfy a number of conditions (see \cite[Thm.~1]{hazwadmax}).
This result ensures that certain division algebras, for example
division algebras of degree $2^n$ or $3^n$ for $n \geq 1$, have
maximal subgroups.

We note that $\CK[1]$ can be considered as a functor from the
category of central simple algebras over a fixed field to the
category of abelian groups (see
Section~\ref{sectionlowerkgroupsofcsalgebras}). In fact, this can be
generalised to cover commutative rings. Central simple algebras over
fields are generalised by Azumaya algebras over commutative rings.
Azumaya algebras were originally defined as ``proper maximally
central algebras'' by Azumaya in his 1951 paper \cite{az}. We
outline in Section~\ref{sectiondevelopmentofazumayaalgebras} how the
definition has developed since then. An Azumaya algebra $A$ over a
commutative ring $R$ can be defined as an $R$-algebra $A$ such that
$A$ is finitely generated as an $R$-module and $A/ \m A$ is a
central simple $R/ \m $-algebra for all $\m \in \Max(R)$ (see
Theorem~\ref{azumayadefinthm} for some equivalent definitions).

Then $\CK[1]$ can also be considered as a functor from the category
of Azumaya algebras over a fixed commutative ring to the category of
abelian groups (see page~\pageref{ckifunctorazalgtoabgroups}).
Related to this, various $\CK[1]$-like functors on the categories of
central simple algebras and Azumaya algebras have been investigated
in \cite{sk12001, hazratreducedkthofazumaya, lewistignol}. In these
papers various abstract functors have been defined, which have
similar properties to the functor $\CK[1]$. For example, in
\cite{hazratreducedkthofazumaya}, the functor defined there is used
to show that the $K$-theory of an Azumaya algebra over a local ring
is almost the same as the $K$-theory of the base ring.

Along the same lines, in Chapter~\ref{chapterktheoryofazalg} we
define an abstract functor, called a $\mathcal D$-functor, which
also has similar properties to $\CK[1]$. We show that the range of a
$\mathcal D$-functor is a bounded torsion abelian group, and that
$\CK[i]$ and $\ZK[i]$, for $i \geq 0$, are $\mathcal D$-functors. By
combining these results with \eqref{exactseqeqn}, we show in
Theorem~\ref{azumayafreethm} that if $A$ is an Azumaya algebra free
over its centre $R$ of rank $n$, then
$$
K_i(A) \otimes \mathbb Z[1/n] \cong K_i(R) \otimes \mathbb Z[1/n]
$$
for any $i \geq 0$. This allows us to extend the results of
\cite{hazratreducedkthofazumaya} to cover Azumaya algebras over
semi-local rings (see Corollary~\ref{semilocal}).


Thus far, we have been considering division algebras. Recently
Tignol and Wadsworth \cite{tigwad, tigwadtotallyramified} have
studied division algebras equipped with a valuation. Valuations are
more common on fields than on division algebras. However they noted
that a number of division algebras are equipped with a valuation,
and the valuation structure on the division algebra contains a
significant amount of information about the division algebra.

A division algebra $D$ equipped with a valuation gives rise to an
associated graded division algebra $\gr (D)$. These graded division
algebras have been studied in
\cite{boulag,hazwadsk1,hwalg,hwcor,tigwad}. In these papers, as they
are considering graded division algebras associated to division
algebras with valuations, their grade groups are totally ordered
abelian groups. It was noted in \cite{hwcor} that it is relatively
easier to work with graded division algebras, and that not much
information is lost in passing between the graded and non-graded
settings.

We show in Theorem~\ref{gcsaazumayaalgebra} that a graded central
simple algebra (so, in particular, a graded division algebra) with
an abelian grade group is an Azumaya algebra, and therefore the
results of Chapter~\ref{chapterktheoryofazalg} also hold in this
setting. But in the graded setting, we can also consider graded
finitely generated projective modules over a given graded ring. We
define the graded $K$-theory of a graded ring $R$ to be $K_i^{\gr}
(R) = K_i (\Pgr (R))$, where $\mathcal P \mathrm{gr}(R)$ is the
category of graded finitely generated projective $R$-modules.

However, considering graded $K$-theory of graded Azumaya algebras,
Example~\ref{eggrktheory} gives a graded Azumaya algebra such that
its graded $K$-theory is not isomorphic to the graded $K$-theory of
its centre. In this example, we take the real quaternion algebra
$\mathbb H$. Then we show $\mathbb H$ is a graded Azumaya algebra
over $\mathbb R$, with $K_0^\gr (\mathbb H) \otimes \mathbb Z[1/n]
\cong \mathbb Z \otimes \mathbb Z [1/n]$ and $K_0^\gr \big(
Z(\mathbb H) \big) \otimes \mathbb Z[1/n] \cong ( \mathbb Z \oplus
\mathbb Z \oplus \mathbb Z\oplus \mathbb Z) \otimes \mathbb Z
[1/n]$, so they are not isomorphic. Thus the results of
Chapter~\ref{chapterktheoryofazalg} do not follow immediately in the
setting of graded $K$-theory.

But for a graded Azumaya algebra subject to certain conditions, we
show in Theorem~\ref{grazumayafreethm} that its graded $K$-theory is
almost the same as the graded $K$-theory of its centre. More
precisely, for a commutative graded ring $R$, we let
$\Ga^*_{M_{n}(R)}$ denote the elements $(d) \in \Ga^n$ such that
$\GL_{n}(R)[d] \neq \emptyset$, where $\GL_{n}(R)[d]$ are invertible
$n \times n$ matrices with ``shifting'' (see
page~\pageref{pagerefmatrixwithshifting}). We show that if $A$ is a
graded Azumaya algebra which is graded free over its centre $R$ of
rank $n$, such that $A$ has a homogeneous basis with degrees
$(\de_1, \ldots , \de_n)$ in $\Ga^*_{M_n(R)}$, then for any $i \geq
0$,
$$
K_i^{\gr}(A) \otimes \mathbb Z[1/n] \cong K_i^{\gr}(R) \otimes
\mathbb Z[1/n].
$$

Another $K$-group which has been studied in the setting of division
algebras is the reduced Whitehead group $\SK[1]$ (see for example
\cite{platonov}). For a division algebra $D$, the group $\SK[1](D)$
is defined to be $D^{(1)} / D'$ where $D^{(1)}$ is the kernel of the
reduced norm and $D'$ is the group generated by the multiplicative
commutators of $D$. For a graded division algebra $D$, it has been
shown that $\SK[1] (QD) \cong \SK[1] (D)$, where $QD$ is the
quotient division ring of $D$ (see \cite[Thm.~5.7]{hazwadsk1}).
Related to this, we study additive commutators in the setting of
graded division algebras in
Chapter~\ref{chapteradditivecommutators}. We show that for a graded
division algebra over its centre $F$, which is Noetherian as a ring,
then
$$
\frac{D}{[D,D]} \otimes_F QF \cong \frac{QD}{[QD,QD]}.
$$


\subsubsection*{Summary of the Thesis}

In Chapter~\ref{chazumayaalgebras}, we combine various results from
the literature to show some of the definitions of an Azumaya
algebra, their basic properties and the equivalence of some of these
definitions. We also outline how the definition has progressed since
the work of Azumaya. We note here that Grothendieck \cite[\S 5]{gro}
also defines an Azumaya algebra on a scheme $X$ with structure sheaf
$\mathcal O _X$, but we do not consider that point of view.

In Chapter~\ref{chktheory}, we begin by recalling the definitions of
the lower $K$-groups $K_0$, $K_1$ and $K_2$. We then look at the
lower $K$-groups of central simple algebras, including the functors
$\CK[0]$ and $\CK[1]$, and some of their properties. We finish the
chapter by recalling some properties of the higher $K$-groups.

In Chapter~\ref{chapterktheoryofazalg}, we define an abstract
functor, called a $\mathcal D$-functor, defined on the category of
Azumaya algebras over a fixed commutative ring.
This allows us to show that the $K$-theory of an Azumaya algebra
free over its centre is almost the same as the $K$-theory of its
centre. We also note that  Corti\~nas and Weibel
\cite{cortinasweibel} have shown a similar result for the Hochschild
homology of an Azumaya algebra, which we mention in this chapter.

Chapter~\ref{chgraded} introduces graded objects. Often in the
literature the grade groups are abelian and totally ordered, so
torsion-free. We begin this section by adopting in the graded
setting some theorems that we require from the non-graded setting.
Some of these results hold for grade groups which are neither
abelian nor totally ordered. Though in some cases we require
additional conditions on the grade group. We show that for a graded
division ring $D$ graded by an arbitrary group, a graded module over
$D$ is graded free and has a uniquely defined dimension. For a
graded field $R$ and a graded central simple $R$-algebra $A$ graded
by an abelian group, we show that $A$ is a graded Azumaya algebra
over $R$. We also prove a number of results for graded matrix rings
graded by arbitrary groups.

We begin Chapter~\ref{chgradedktheoryofazumayaalgebras} by defining
the group $K_0$ in the setting of graded rings. We show what this
group looks like for a trivially graded field and for a strongly
graded ring. For a specific example of a graded Azumaya algebra, we
show that its graded $K$-theory is not the same as its usual
$K$-theory (see Example~\ref{eggrktheorynotsameasktheory}). Then in
a similar way to Chapter~\ref{chapterktheoryofazalg}, we define an
abstract functor called a graded $\mathcal D$-functor. This allows
us to prove that the graded $K$-theory of a graded Azumaya algebra
(subject to some conditions) is almost the same as the graded
$K$-theory of its centre.

In Chapter~\ref{chapteradditivecommutators}, we study additive
commutators in the setting of graded division algebras. We observe
in Section~\ref{sectiongradedsplitting} that the reduced trace holds
in this setting. We then recall the definition of the quotient
division algebra, and show in Corollary~\ref{d/d,dcongqd/qd,qd} how
the subgroup generated by homogeneous additive commutators in a
graded division algebra relates to that of the quotient division
algebra.



\chapter{Azumaya Algebras} \label{chazumayaalgebras}



The concept of an Azumaya algebra over a commutative ring
generalises the concept of a central simple algebra over a field.
The term Azumaya algebra originates from the work done by Azumaya in
his 1951 paper ``On maximally central algebras'' \cite{az}. The
definition has developed since then, and we will outline in
Section~\ref{sectiondevelopmentofazumayaalgebras} how it has
progressed.  In Theorem~\ref{azumayadefinthm}, we state a number of
equivalent reformulations of this definition.


This chapter is organised as follows. We begin this chapter by
recalling the various definitions of the term ``faithfully
projective'', which are required for the definition of an Azumaya
algebra (see Definition~\ref{azumayadefinition}). In
Sections~\ref{sectionseparable} and~\ref{sectionseparabletwo} we
discuss separable algebras, which can also be used to define Azumaya
algebras. The definition of an Azumaya algebra in stated in
Section~\ref{sectionazumayaalgebras}, along with some examples and
properties, and in Section~\ref{sectionfurthercharacterisations} we
show some additional characterisations of Azumaya algebras. We
conclude this chapter by summarising some of the key progressions in
the development of the theory of Azumaya algebras.


\section{Faithfully projective modules}
\label{sectionfaithfullyprojective}

Let $R$ be a (possibly non-commutative) ring. Consider a covariant
additive functor $T$ from the category of (left or right)
$R$-modules to some category of modules. We say that $T$ is an
\emph{exact functor}\index{functor!exact}\index{exact functor} if,
whenever $L \ra M \ra N$ is an exact sequence of $R$-modules, $T(L)
\ra T(M) \ra T(N)$ is exact. Further $T$ is defined to be a
\emph{faithfully exact functor}\index{functor!faithfully
exact}\index{faithfully exact functor} if the sequence $T(L) \ra
T(M) \ra T(N)$ is exact if and only if the sequence $L \ra M \ra N$
is exact.

We recall that an $R$-module $M$ is called
\emph{faithful}\index{module!faithful}\index{faithful module} if
$rM=0$ implies $r=0$ or, equivalently, if its annihilator
$\mathrm{Ann}(M) = \{ x \in R : xm=0$ for all $m \in M \}$ is zero.
An $R$-module $M$ is called a \emph{flat
module}\index{module!flat}\index{flat module} if the functor $-
\otimes_R M$ is an exact functor from the category of right
$R$-modules to the category of abelian groups. An $R$-module $P$ is
called a {\em projective
module}\index{module!projective}\index{projective module} if the
functor $\Hom_R(P, -)$ is an exact functor from the category of left
$R$-modules to the category of abelian groups. This is equivalent to
saying that $P$ is a direct summand of a free $R$-module. If $P$ is
a projective $R$-module which is finitely generated by $n$ elements,
then $P$ is a direct summand of $R^n$. See Magurn
\cite[Ch.~2]{magurn} for results involving projective modules.


The following results on faithfully exact functors are from Ishikawa
\cite[p.~30--33]{ish}.

\begin{thm}\label{faithfullyexactthm}\index{functor!faithfully exact}
Let $T$ be an exact functor from the category of left (resp.\ right)
$R$-modules to some category of modules. Then the following are
equivalent:
\begin{enumerate}
\item $T$ is a faithfully exact functor.

\item $T(A) \neq 0$ for every non zero  left (resp.\ right) $R$-module $A$.

\item $T(\phi) \neq 0$ for every non zero $R$-linear homomorphism
$\phi$.

\item $T(R/I) \neq 0$ for every proper left (resp.\ right) ideal
$I$ of $R$.

\item $T(R/\m) \neq 0$ for every maximal left (resp.\ right)
ideal $\m$ of $R$.
\end{enumerate}
\end{thm}
We will consider the left version of this theorem in the proof
below. The right version follows analogously.

\begin{proof}
(1) $\Ra$ (2): Let $T(A) = 0$ for an $R$-module $A$. Then since
$T(0)$ is the zero module, $T(0) \ra T(A) \ra T(0)$ is exact. By
(1), this implies that $0 \ra A \ra 0$ is exact, proving $A=0$.

\vspace*{3pt}

(2) $\Ra$ (3): Let $\phi :X \ra Y$ be an $R$-linear homomorphism.
Then we have exact sequences $X \stackrel{\phi'}{\ra} \Im(\phi) \ra
0$ and $0 \ra \Im(\phi) \stackrel{i}{\ra} Y$, where $i$ is the
inclusion map. Since $T$ is an exact functor, we get the following
commutative diagram with its row and column exact:
\begin{displaymath}
\xymatrix{& 0 \ar[d] &    \\
T(X) \ar[r]^{T(\phi')\;\;\;} \ar[dr]_{T(\phi)} & T\big(\Im
(\phi)\big) \ar[d]^{T(i)}  \ar[r] & 0 \\
& T(Y)&  }
\end{displaymath}
If $T(\phi) = 0$, then $T(i) \circ T(\phi') = 0$. This implies that
$\Im \big(T(\phi')\big) \subseteq \ker \big(T(i)\big) =0$, so
$T(\phi')=0$ and, since $T(\phi')$  is surjective, $T\big(\Im
(\phi)\big)=0$. By condition (2), $\Im(\phi) =0$, so $\phi=0$.

\vspace*{3pt}

(3) $\Ra$ (1): Let $T(A) \stackrel{T(f)}{\lra} T(B)
\stackrel{T(g)}{\lra} T(C)$ be exact. Since $T(g \circ f) = T(g)
\circ T(f)=0$, by condition (3), $g \circ f=0$ and $\Im(f) \subseteq
\ker (g)$. We have exact sequences $0 \ra \ker (g) \stackrel{j}{\ra}
B \stackrel{g}{\ra} C$, $A \stackrel{f'}{\ra} \Im (f) \ra 0$ and $0
\ra \Im (f) \stackrel{i}{\ra} \ker(g) \stackrel{p}{\ra} \ker(g)/\Im
(f) \ra 0$, where $i$ and $j$ are inclusion maps. Since $T$ is an
exact functor, we obtain the following commutative diagram with
exact rows and columns:
\begin{displaymath}
\xymatrix{& 0 \ar[d] & &   \\
0  & T\big(\Im (f)\big) \ar[l] \ar[d]_{T(i)}  & T(A) \ar[d]^{T(f)}
\ar[l]_{\;\;\;\;\;\;\;T(f')}  &\\
0 \ar[r] & T\big(\ker (g)\big) \ar[r]_{\;\;\;\;\; T(j)}
\ar[d]_{T(p)} & T(B) \ar[r]_{\; T(g)} & T(C) \\
& T\big(\ker(g)/ \Im (f)\big) \ar[d] &&\\
&0&&}
\end{displaymath}
For $x \in T\big( \ker(g)\big)$, $\big(T(g) \circ T(j)\big) (x) =0$,
so $T(j)(x) \in \ker \big(T(g) \big)= \Im \big(T(f) \big)$. So there
is an element $y \in T(A)$ such that $T(f)(y)= T(j)(x)$. Hence
$T(j)(x)= T(f)(y) = T(j) \circ T(i) \circ T(f') (y)$, so that $x =
\big( T(i) \circ T(f') \big) (y)$, since $T(j)$ is injective. This
shows that $T(i)$ is surjective, and hence $T\big( \Im(f)\big) \cong
T\big( \ker(g)\big)$, which means that $T(p)=0$. By condition (3),
this implies $p=0$, so $\Im(f) = \ker (g)$, proving $A
\stackrel{f}{\ra} B \stackrel{g}{\ra} C$ is exact.

\vspace*{3pt}

(2) $\Ra$ (4) and (4) $\Ra$ (5) are trivial.

\vspace*{3pt}

(5) $\Ra$ (2): Let $T(A)=0$. Let $a \in A$ and let $Ra$ be the left
$R$-submodule of $A$ generated by $a$. Since $0 \ra Ra \ra A$ is
exact and $T$ is exact, $T(Ra)=0$. Let $\mal (a)= \{ r \in R: ra =0
\}$, which is a left ideal of $R$. If $\mal (a) \neq R$, then there
is a maximal ideal $\m$ of $R$ containing $\mal (a)$. We have an
exact sequence $R/ \mal (a) \ra R/\m \ra 0$. Since $R \ra Ra$ is
surjective, $Ra \cong R/ \mal (a)$ by the First Isomorphism Theorem,
and we have an exact sequence $0 = T(Ra) \cong T \big(R/\mal
(a)\big) \ra T(R/\m) \ra 0$. This implies $T(R/\m) = 0$,
contradicting (5). So $\mal (a) =R$, which implies $a=0$ and
therefore $A=0$.
\end{proof}

For fixed left $R$-modules $P$ and $M$, the functors $T(-) =
\Hom_R(P,-)$ and $U(-)= - \otimes_R M$ are covariant functors
defined on the category of left $R$-modules and right $R$-modules,
respectively.

\begin{defin} \label{faithfullyprojdefin}
\index{module!faithfully projective}\index{faithfully projective
module}\index{module!faithfully flat}\index{faithfully flat module}
An $R$-module $P$ is said to be \emph{faithfully projective} if
$T(-) = \Hom_R(P,-)$ is a faithfully exact functor and an $R$-module
$M$ is said to be \emph{faithfully flat} if $U(-)= - \otimes_R M$ is
a faithfully exact functor.
\end{defin}

By applying Theorem~\ref{faithfullyexactthm} to the functors $U$ and
$T$ respectively, we get the following theorems.

\begin{thm} \label{faithfullyflatthm}
Let $M$ be a flat left $R$-module. Then the following are
equivalent:
\begin{enumerate}
\item $M$ is faithfully flat.

\item $A \otimes_R M \neq 0$ for every non zero right $R$-module $A$.

\item $\phi \otimes_R \id_M \neq 0$ for every non zero right $R$-linear homomorphism
$\phi$.

\item $IM \neq M$ for every proper right ideal $I$ of $R$.

\item $\m M \neq M$ for every maximal right ideal $\m$ of $R$.
\end{enumerate}
\end{thm}

\begin{proof}
Follows immediately from Theorem~\ref{faithfullyexactthm}. Note that
in part (4), $R/I \otimes_R M \cong M/IM$, and $M/IM = 0$ if and
only if $M=IM$.
\end{proof}

\begin{thm} \label{faithfullyprojthm}
Let $P$ be a projective left $R$-module. Then the following are
equivalent:
\begin{enumerate}
\item $P$ is faithfully projective.

\item $\Hom_R(P, A) \neq 0$ for every non zero left $R$-module $A$.

\item $\Hom_R(P, \phi) \neq 0$ for every non zero left $R$-linear homomorphism
$\phi$.

\item $\Hom_R(P,R/I) \neq 0$ for every proper left ideal $I$ of $R$.

\item $\Hom_R(P,R/\m) \neq 0$ for every maximal left ideal $\m$ of $R$.
\end{enumerate}
\end{thm}

\begin{proof}
Follows immediately from Theorem~\ref{faithfullyexactthm}. In part
(3), if $\phi \in \Hom_R(X,Y)$, then $T(\phi)= \Hom_R( P, \phi):
\Hom_R(P,X) \ra \Hom_R(P,Y);$ $\psi \mps \phi \circ \psi$.
\end{proof}

\begin{prop} \label{faithfullprojandflatprop}
If an $R$-module $P$ is faithfully projective, then $P$ is
projective and faithfully flat. Further, when the ring $R$ is
commutative the converse holds.
\end{prop}

\begin{proof}
If an $R$-module $P$ is faithfully projective, it is projective, and
therefore also flat (see \cite[Prop.~4.3]{lam}). Using
Theorem~\ref{faithfullyflatthm}(2), we assume $A \otimes_R P =0$ and
need to prove that $A=0$. Then by \cite[\S II.4.1,
Prop.~1]{bouralgebra},
$$
\Hom_R (P , \Hom_{\mathbb Z} (A,A)) \cong \Hom_{\mathbb Z} (A
\otimes_R P, A) = \Hom_{\mathbb Z}(0, A) =0.
$$
Since $P$ is faithfully projective, by
Theorem~\ref{faithfullyprojthm}(2), $\Hom_{\mathbb Z} (A,A)=0$, so
$A=0$.

Conversely, let $R$ be commutative and $P$ be faithfully flat and
projective. By Theorem~\ref{faithfullyflatthm}(5), for any maximal
ideal $\m$ of $R$, we have $P/ \m P \neq 0$. Since $R \ra R/\m$ is a
surjective ring homomorphism, $R/\m$-linear maps can be considered
as $R$-linear maps and we have $\Hom_R(P/\m P, R/ \m) =\Hom_{R/\m}
(P/ \m P , R/\m)$. Since $R/\m$ is a field and $P/\m P \neq 0$, its
dual module $\Hom_{R/\m} (P/\m P, R/\m)$ is also non-zero, using
$$
\dim_{R/\m} (\Hom_{R/\m} (P/\m P, R/\m)) \geq \dim_{R/\m}(P/\m P),
$$
from \cite[p.~204, Remarks]{hungerford}. We have an exact sequence
$$
0 \lra \Hom_R(P/\m P, R/ \m) \lra \Hom_R(P, R/ \m),
$$
so $\Hom_R(P, R/ \m) \neq 0$, since $\Hom_R(P/\m P, R/ \m) \neq 0$.
By Theorem~\ref{faithfullyprojthm}(5), $P$ is faithfully projective.
\end{proof}

In the following proposition, we show that the definition of a
faithfully projective $R$-module
(Definition~\ref{faithfullyprojdefin}) can be expressed in a number
of different ways, which are equivalent to the definition given
above provided $R$ is a commutative ring. The second definition is
from \cite[p.~39]{bass} or \cite[p.~186]{farbdennis}, and the third
from \cite[p.~52]{knus}.

\begin{prop} \label{faithfullyprojprop}
Let $R$ be a commutative ring, and let $P$ be an $R$-module. Then
the following are equivalent:
\begin{enumerate}
\item $P$ is faithfully projective;

\item $P$ is finitely generated, projective and faithful as an
$R$-module;

\item $P$ is projective over $R$ and $P \otimes_R N = 0$ implies
$N=0$ for any left $R$-module $N$.
\end{enumerate}
\end{prop}

\begin{proof}
(1) $\Lra$ (3): This follows from
Proposition~\ref{faithfullprojandflatprop} and
Theorem~\ref{faithfullyflatthm}(2).

\vspace*{3pt}

(1) $\Lra$ (2): See Bass \cite[Cor.~II.5.10]{bass}.
\end{proof}

We show below how faithfully flat modules are related to modules
which are faithful and flat.

\begin{prop}
Let $R$ be a ring. A faithfully flat left $R$-module $M$ is both
faithful and flat.
\end{prop}

\begin{proof}
See Lam \cite[Prop.~4.73]{lam}, with minor alterations for left
modules.
\end{proof}

In general the converse does not hold. For example, if $R=
\mathbb{Z}$, then the module $\mathbb{Q}$ is faithful and flat,
since $\mathbb Q = (\mathbb Z \mi 0)^{-1} \mathbb Z$ is a
localisation of $\mathbb Z$, and we know localisations are flat (see
\cite[Prop.~6.56]{magurn}). But by Theorem~\ref{faithfullyflatthm},
$\mathbb Q$ not faithfully flat over $\mathbb Z$, since for the
ideal $2 \mathbb Z$ of $\mathbb Z$, $\mathbb{Q} \otimes_{\mathbb{Z}}
(\mathbb{Z}/2\mathbb{Z}) = 0$. Even a faithful and projective
$R$-module $M$ is not necessarily faithfully flat over $R$. For
example, let $R$ be the direct product $\mathbb Z \times \mathbb Z
\times \cdots$, and let $M$ be the ideal $\mathbb Z \oplus \mathbb Z
\oplus \cdots$ in $R$. Then $M$ is faithful as a left $R$-module,
and it is projective (see \cite[Eg.~2.12C]{lam}). But we have $M^2
=M$, so for any maximal ideal $\m$ of $R$ containing $M$, we have $M
\m = M$. So by Theorem~\ref{faithfullyflatthm}, $M$ is not
faithfully flat.


\section{Separable algebras over commutative rings}
\label{sectionseparable}

In this section, we let $R$ denote a commutative ring. Let $A$ be an
$R$-algebra, and let $A^e = A \otimes_R A^{\mathrm{op}}$ be the
enveloping algebra\index{enveloping
algebra}\index{algebra!enveloping} of $A$, where $A^{\mathrm{op}}$
denotes the opposite algebra of $A$. Then the $R$-algebra $A^e$ has
a left action on $A$ induced by:
$$
(a \otimes b)x   := axb \;\;\;\;\;\;\;\;\;\; \textrm{ for $a, x \in
A$, $b \in A^{\op}$,}
$$
which is denoted by $(a \otimes b) \ast x$. Any $A$-bimodule $M$ can
also be viewed as a left $A^e$-module. We set
$$
M^A = \{ m \in M : ma =am \textrm{ for all } a \in A \}.
$$
There is an $A^e$-linear map $\mu :A^e \ra A;$ $a \otimes b \mapsto
(a \otimes b)\ast 1 = ab$, extended linearly, and we let $J$ denote
the kernel of $\mu$.

\begin{defin} \label{separabledefin}
An $R$-algebra $A$ is said to be \emph{separable}\index{separable
algebra}\index{algebra!separable} over $R$ if $A$ is projective as a
left $A^e$-module.
\end{defin}

The following two theorems show some equivalent characterisations of
separability.

\begin{thm} \label{separablelemma}
Let $A$ be an $R$-algebra. The following are equivalent:
\begin{enumerate}
\item $A$ is separable.

\item The exact sequence of left $A^e$-modules
$$ 0 \lra J \lra A^e \stackrel{\mu}{\lra} A \lra 0 $$ splits.

\item The functor $(-)^A: A^e$-$\mathcal{M}\mathfrak{od} \ra R$-$
\mathcal{M}\mathfrak{od}$ is exact.

\item There is an element $e \in A^e$ such that $e *1 = 1$ and $Je
=0$.

\item There is an element $e \in A^e$ such that $e \ast 1 = 1$
and $(a \otimes 1) e = (1 \otimes a)e$ for all $a \in A$.
\end{enumerate}
\end{thm}

Such an element $e$ as in Theorem~\ref{separablelemma} is an
idempotent, called a \emph{separability
idempotent}\index{separability idempotent} for $A$, since $e^2 - e =
(e - 1 \otimes 1)e \in J e = 0$.

\begin{proof}

 \cite[Lemma~III.5.1.2]{knus}, \cite[Prop.~II.1.1]{demeyer}

\vspace*{2pt}

(1) $\Leftrightarrow$ (2): The forward direction follows immediately
from the definition of a projective module. For the converse, using
known results involving projective modules (see
\cite[Cor.~2.16]{magurn}), (2) implies $A^e \cong J \oplus A$, so
$A$ is projective.

\vspace*{3pt}

(1) $\Leftrightarrow$ (3): For all $A$-bimodules $M$, the natural
map
\begin{eqnarray} \rho_M :\Hom_{A^e} (A, M) & \lra & M^A \nono \\
f & \lmps & f(1) \nono
\end{eqnarray}
is an isomorphism of $R$-modules, with the inverse being
\begin{eqnarray} \rho^{-1}_M :  M^A & \lra & \Hom_{A^e} (A, M) \nono \\
x & \longmapsto & \mar_x :A \ra M \nono \\
&& \;\;\;\;\;\;\;\; a \mapsto ax. \nono
\end{eqnarray}
Since $A$ is separable if and only if $\Hom_{A^e}(A , -)$ is an
exact functor, this proves the equivalence of (1) and (3).

\vspace*{3pt}

(2) $\Rightarrow$ (4): Let $\ga : A \ra A^e$ be an $A^e$-module
homomorphism such that $\mu \circ \ga = \id_A$. Let $e = \ga (1)$,
so that $1 = \mu(e)=e*1$. To show that $J e =0$, let $a \in J$. Then
as $\ga$ is $A^e$-linear, $ae = a \ga(1) = \ga (a * 1)=0$, since we
have $a *1 = \mu (a) = 0$, proving $Je =0$, as required for (4).

\vspace*{3pt}

(4) $\Ra$ (5): From (4), we have an element $e \in A^e$ such that
$e*1 = 1$ and $Je =0$. Let $a \in A$ be arbitrary. Then $\mu(1
\otimes a - a \otimes 1 )= 0$, so $1 \otimes a - a \otimes 1 \in J$.
Hence $(1 \otimes a -a \otimes 1)e = 0$; that is, $(1 \otimes a) e =
(a \otimes 1)e$, proving (5).

\vspace*{3pt}

(5) $\Rightarrow$ (2): If $e$ is an element of $A^e$ satisfying the
conditions in (5), we can define a map $\ga$ by $\ga : A  \ra A^e;$
$a  \mapsto  (1 \otimes a)e$. Using the assumption that $(a \otimes
1) e = (1 \otimes a)e$ for all $a \in A$, we can show that $\ga$ is
an $A^e$-module homomorphism. It is a right inverse of $\mu$ since,
for $a \in A$, writing $e = \sum x_i \otimes y_i$ gives
\begin{align*}
\mu \circ \ga (a) &= \mu \big( (1 \otimes a) e \big) \\
&= \mu \left( \sum ( x_i \otimes y_i a) \right) \\
&= \left( \sum x_i y_i \right) a \\
&= 1.a \; = \;  \id_A (a),
\end{align*}
completing the proof.
\end{proof}

\begin{thm}\label{maxseparable}
Let $A$ be an $R$-algebra which is finitely generated as an
$R$-module. The following are equivalent:
\begin{enumerate}
\item $A$ is separable over $R$.

\item $A_{\m}$ is separable over $R_{\m}$ for all  $\m \in \Max(R)$.

\item $A/ \m \! A$ is separable over $R/ \m $ for all  $\m \in \Max(R)$.
\end{enumerate}

\end{thm}

\begin{proof}
See \cite[Lemma~III.5.1.10]{knus}.
\end{proof}

For a free $R$-module $F$, we know that $F$ is isomorphic to a
direct sum of copies of $R$ as a left $R$-module; that is, $F \cong
\bigoplus_{i \in I} R_i$ for $R_i =R$ (see
\cite[Thm.~IV.2.1]{hungerford}). Let $f_i \in \Hom_R ( F , R)$ be
the projection of $R_i$ onto $R$, and let $e_i$ be the element of
$F$ with $1$ in the $i$-th position and zeros elsewhere. Then
clearly the following results hold:
\begin{enumerate}
\item for every $x \in F$, $f_i (x) = 0$ for all but a finite subset
of $i \in I$;

\item for every $x \in F$, $\sum_i f_i (x) e_i = x$.
\end{enumerate}
The following lemma shows that we have similar results when we
consider projective modules, rather than free modules. Moreover,
such properties are sufficient to characterise a projective module.

\begin{lemma}[Dual Basis Lemma] \label{dualbasislemma}
Let $M$ be an $R$-module. Then $M$ is projective if and only if
there exists $\{ m _i \}_{i \in I} \subseteq M$ and $\{ f_i \}_{i
\in I} \subseteq \Hom_R (M,R)$, for some indexing set $I$, such that
\begin{enumerate}
\item for every $m \in M$, $f_i (m) =0$ for all but a finite subset
of $i \in I$; and

\item for every $m \in M$, $\sum_{i \in I} f_i (m) m_i = m.$
\end{enumerate}
Moreover, $I$ can be chosen to be a finite set if and only if $M$ is
finitely generated.
\end{lemma}
The collection $\{ f_i , m_i \}$ is called a {\em dual basis} for
$M$.

\begin{proof}
See \cite[Lemma~I.1.3]{demeyer}.
\end{proof}

The  following proposition, from Villamayor, Zelinsky
\cite[Prop.~1.1]{villamayorzelinsky} (see also
\cite[Prop.~II.2.1]{demeyer}), shows that an algebra which is
separable and projective is finitely generated. This is a somewhat
surprising result, as the requirement of being separable and
projective does not immediately appear to imply a finitely generated
condition. The proof of the proposition uses the Dual Basis Lemma.

\begin{prop} \label{separableprojectivefg}
Let $A$ be a separable $R$-algebra which is projective as an
$R$-module. Then $A$ is finitely generated as an $R$-module.

\end{prop}

\begin{proof}

Since $A$ is projective as an $R$-module, $A^{\op}$ is also
projective as an $R$-module. Let $\{ f_i , a_i\}$ be a dual basis
for $A^{\op}$ over $R$, where $a_i \in A^{\op}$ and $f_i \in \Hom_R
( A^{\op} , R)$. Then for every $b \in A^{\op}$, $b = \sum_{i \in I}
f_i (b) a_i$ and $f_i(b) =0$ for all but finitely many $i \in I$
(using Lemma~\ref{dualbasislemma}). Since $A \otimes_R R \cong A$,
we can identify $A \otimes_R R$ with $A$ and can consider $\id_A
\otimes f_i$ as a map from $A^e$ to $A$. This map is $A$-linear, and
we claim that $\{ \id_A \otimes f_i , 1 \otimes a_i \}$ forms a dual
basis for $A^e$ as a projective left $A$-module. Let $a \otimes b
\in A^e$ be arbitrary. Since $f_i (b) = 0$ for all but a finite
number of subscripts $i$, we also have $(\id_A \otimes f_i)( a
\otimes b) = a \otimes f_i (b) = 0$ for all but a finite number of
$i$. Then
\begin{align*}
\sum_{i \in I} (\id_A \otimes f_i)( a \otimes b) (1 \otimes a_i) 
& = \sum_{i \in I} a \otimes f_i (b) a_i \\
&= a \otimes b.
\end{align*}
Extended linearly, this holds for all $u \in A^e$.  So $\{  \id_A
\otimes f_i , 1 \otimes a_i \}$ forms a dual basis for $A^e$ over
$A$.

Let $a \in A^{\op}$ be a fixed arbitrary element. We will show that
$a$ can be written as an $R$-linear combination of a finite subset
of $A^{\op}$, where this finite subset is independent of $a$. Let
$e= \sum_j x_j \otimes y_j$ be a
separability idempotent for $A$ over $R$ 
 and define $u= (1 \otimes a)e \in A^e$. We have $u \ast 1 = \sum_j
x_j y_j a  = a$ and, from Lemma~\ref{dualbasislemma}, $u= \sum_i
(\id_A \otimes f_i)\left((1 \otimes a) e \right) (1 \otimes a_i)$.
Then
\begin{align} \label{dualbasissum}
a \; = u \ast 1 \, 
&= \sum_i \Big(   \big((\id_A \otimes f_i)( (1 \otimes a)e)
\otimes a_i \big) \ast 1 \Big) \nono \\
&= \sum_i \big( (\id_A \otimes f_i)( (1 \otimes a)e)\big) \cdot a_i
.
\end{align}
Using Proposition~\ref{separablelemma}(5),
\begin{align*}
(\id_A \otimes f_i)( (1 \otimes a)e) &= (\id_A \otimes f_i)( (a
\otimes 1)e) \\
&= a(\id_A \otimes f_i)( e) \\
&= (a \otimes 1) * \big( (\id_A \otimes f_i)( e) \big).
\end{align*}
Since $\{  \id_A \otimes f_i , 1 \otimes a_i \}$ forms a dual basis
for $A^e$, $(\id_A \otimes f_i)(e) =0$ for all but a finite subset
of $i \in I$. So the set of subscripts $i$ for which $(\id_A \otimes
f_i)\big( (1 \otimes a)e \big)$ is non-zero is contained in the
finite set of subscripts for which $(\id_A \otimes f_i)(e)$ is
non-zero, which is independent of $a$. Then since $(\id_A \otimes
f_i)\big( (1 \otimes a)e \big)$ is non-zero for only finitely many
$i \in I$, the sum (\ref{dualbasissum}) may be taken over a finite
set. Again writing $e = \sum_j x_j \otimes y_j$,
(\ref{dualbasissum}) says:
\begin{align*}
a 
&= \sum_{i, j} \left( x_j \otimes f_i( y_j \;\! a) \right) \cdot a_i \\
&= \sum_{i, j}  x_j f_i( y_j \;\! a) \;\! a_i \\
&= \sum_{i, j} f_i( y_j \;\! a) \;\! x_j  \;\! a_i.
\end{align*}
So the finite set $\{x_j a_i\}$ generates $A^{\op}$ over $R$, and
therefore generates $A$ over $R$. This completes the proof that $A$
is finitely generated.
\end{proof}

\section{Other definitions of separability}
\label{sectionseparabletwo}

In Definition~\ref{separabledefin} above, we define separability
over a commutative ring $R$. If $R$ is a field, we also have the
classical definition of separability. For a field $R$, an
$R$-algebra $A$ is said to be \emph{classically
separable}\index{classically separable
algebra}\index{algebra!classically separable} if, for every field
extension $L$ of $R$, the Jacobson radical of $A\otimes_R L$ is
zero, where the Jacobson radical\index{Jacobson radical} of $A
\otimes_R L$ is the intersection of the maximal left ideals. The
following theorem shows the connection between the two definitions
of separability when $R$ is a field.

\begin{thm} \label{classicallyseparablethm}
Let $R$ be a field and $A$ be an $R$-algebra. Then $A$ is separable
over $R$ if and only if $A$ is classically separable over $R$ and
the dimension of $A$ as a vector space over $R$ is finite.
\end{thm}

\begin{proof}
See \cite[Thm.~II.2.5]{demeyer}.
\end{proof}

There is a further definition of separability for fields. For a
field $R$, an irreducible polynomial $f(x) \in R[x]$ is separable
over $R$ if $f$ has no repeated roots in any splitting field. 
An algebraic field extension $A$ of $R$ is said to be a
\emph{separable field extension}\index{separable field
extension}\index{field extension!separable} of $R$ if, for every $a
\in A$, the minimal polynomial of $a$ over $R$ is separable. The
theorem below shows that for a finite field extension, this
definition agrees with the definition of classical separability
given above. The theorem also shows their connection with the
definition of a separable algebra given in
Definition~\ref{separabledefin}.

\begin{thm}
Let $R$ be a field, and let $A$ be a finite field extension of $R$.
Then the following are equivalent:
\begin{enumerate}
\item $A$ is separable as an $R$-algebra,

\item $A$ is classically separable over $R$,

\item $A$ is a separable field extension of $R$.
\end{enumerate}
\end{thm}



\begin{proof}
(1) $\Lra$ (2): Follows immediately from
Theorem~\ref{classicallyseparablethm}.

\vspace{3pt}

(2) $\Lra$ (3): See \cite[Lemma~9.2.8]{weibelhomalg}.
\end{proof}

\section{Azumaya algebras}
\label{sectionazumayaalgebras}



Let $R$ be a commutative ring and $A$ be an $R$-algebra. There is a
natural $R$-algebra homomorphism $\psi_A : A^e \ra \End_R(A)$
defined by $\psi_A (a \otimes b) (x) = axb$, extended linearly. If
the context is clear, we will drop the subscript $A$. We now are
ready to define an Azumaya algebra: this is the definition from
\cite[p.~186]{farbdennis} and \cite[p.~134]{knus}.

\begin{defin} \label{azumayadefinition}
An $R$-algebra $A$ is called an \emph{Azumaya algebra}\index{Azumaya
algebra}\index{algebra!Azumaya} if the following two conditions
hold:
\begin{enumerate}
\item $A$ is a faithfully projective $R$-module.

\item The map $\psi_A : A^{e} \ra \End_R(A)$ defined above is an
isomorphism.
\end{enumerate}
\end{defin}

\begin{example}
Any finite dimensional central simple algebra $A$ over a field $F$
is an Azumaya algebra. A central simple algebra is free, so it is
projective and faithful, and we know $A \otimes_F A^{\op} \cong
M_n(F) \cong \End_F(A)$ (see \cite[Thm.~8.3.4]{sch}).
\end{example}

We will see some further examples of Azumaya algebras on pages
\pageref{azumayaexamples} and \pageref{azumayaexample3}.

\begin{prop}\label{homisoprop}
Let $E_1$, $E_2$, $F_1$, $F_2$ be $R$-modules. When one of the
ordered pairs $(E_1, E_2)$, $(E_1, F_1)$, $(E_2, F_2)$ consists of
finitely generated projective $R$-modules, the canonical
homomorphism 
\begin{align*}
\Hom(E_1, F_1) \otimes \Hom(E_2, F_2) & \lra \Hom (E_1 \otimes E_2 ,
F_1 \otimes F_2)
\end{align*}
 is bijective.
\end{prop}

\begin{proof}
See \cite[\S II.4.4, Prop.~4]{bouralgebra}.
\end{proof}

We observe that if $E_1, \ldots, E_n, F_1, \ldots, F_m$ are any
$R$-modules and
$$
\phi:E_1 \oplus \cdots \oplus E_n \lra F_1 \oplus \cdots \oplus F_m
$$
is an $R$-module homomorphism, then $\phi$ can be represented by a
unique matrix
$$ \left( \begin{array}{l c l}
\phi_{11} & \cdots & \phi_{1n} \\
\;\; \vdots &&\;\; \vdots  \\
\phi_{m1} & \cdots & \phi_{mn}
\end{array} \right) $$
where $\phi_{ij} \in \Hom_R (E_j , F_i)$. In particular for
$R$-modules $M$ and $N$ there are $R$-module homomorphisms
$$
\begin{array}{rlcll} \label{ijhomomorphisms}
\End_R (M) & \stackrel{i}{\lra} & \End_R (M \oplus N) &
\stackrel{j}{\lra} & \End_R (M)  \\
f & \stackrel{i}{\longmapsto} & \left( \begin{array}{c c} f & 0 \\
0 & 0 \end{array} \right) && \\
&& \left( \begin{array}{c c} x & y \\ z & w \end{array} \right) &
\stackrel{j}{\longmapsto} & x
\end{array}
$$
Note that $j \circ i = \id$, so that $i$ is injective and $j$ is
surjective.

\vspace{3pt}

The proposition below shows how to construct some examples of
Azumaya algebras. The proof is from Farb, Dennis
\cite[Prop.~8.3]{farbdennis}.

\begin{prop} \label{endomorphismprop}
If $P$ is a faithfully projective $R$-module, then $\End_R(P)$ is an
Azumaya algebra over $R$.
\end{prop}

\begin{proof}
Using Proposition~\ref{faithfullyprojprop}, $P$ is a finitely
generated projective $R$-module, so we can choose an $R$-module $Q$
with $P \oplus Q \cong R^n$ for some $n$. Hence $\End_R (P \oplus Q)
\cong \End_R (R^n) \cong M_n (R) \cong R^{n^2}$ as $R$-modules, so
$\End_R (P \oplus Q)$ is a free module. We saw above that there are
homomorphisms $\End_R (P) \stackrel{i}{\ra} \End_R (P \oplus Q)
\stackrel{j}{\ra} \End_R (P)$ with $j \circ i= \id$.
Then $\End_R(P)$ is isomorphic to a direct summand of $\End_R(P
\oplus Q)$ and therefore $\End_R(P)$ is finitely generated and
projective as an $R$-module. Suppose $r \in R$ annihilates
$\End_R(P)$. In particular $r$ annihilates the identity map on $P$,
and so $rp= 0$ for all $p \in P$. But $P$ is a faithful module, so
$r=0$, which shows that $\End_R(P)$ is also a faithful $R$-module.
By Proposition~\ref{faithfullyprojprop}, this shows that $\End_R
(P)$ is faithfully projective.

It remains to show the second condition. 
Consider the following diagram:
\begin{displaymath}
\xymatrix{ \End_R(P) \otimes \End_R(P)^{\textrm{op}}
\ar@<-2ex>[d]_{i'}
\ar[rr]^{\;\;\;\;\;\;\;\; \psi_P} && \End_R(\End_R(P)) \ar@<-1ex>[d]_{i''} \\
\End_R(P \oplus Q) \otimes (\End_R(P \oplus Q))^{\textrm{op}}
\ar[u]_{j'} \ar[rr]_{\;\;\;\;\;\;\;\;\;\;\;\; \psi_{P \oplus Q }} &&
\End_R ( \End_R(P \oplus Q)) \ar@<-1ex>[u]_{j''} }
\end{displaymath}
where $\psi_P$ and $\psi_{P \oplus Q}$ are defined as in
Definition~\ref{azumayadefinition}, and the maps $i', j' , i''$ and
$j''$ come from the homomorphisms $i$ and $j$ on
page~\pageref{ijhomomorphisms}. For an element $f \otimes g$ in
$\End_R(P) \otimes \End_R(P)^{\textrm{op}}$, we have
\begin{eqnarray*}
\psi_{P \oplus Q} \circ i' (f \otimes g) = i'' \circ \psi_P (f
\otimes g) :& \End_R(P \oplus Q) \!\!\! & \lra \;\; \End_R(P
\oplus Q) \\
&  \begin{pmatrix}\al_{11} & \al_{12} \\ \al_{21} & \al_{22}
\end{pmatrix} \!\!\! & \lmps \;\; \begin{pmatrix} f \circ \al_{11} \circ g & 0 \\
0 & 0 \end{pmatrix}
\end{eqnarray*}
and for  $f \otimes g = \begin{pmatrix} f_{11}& f_{12} \\ f_{21} &
f_{22}
\end{pmatrix}  \otimes \begin{pmatrix} g_{11}& g_{12} \\ g_{21} & g_{22}
\end{pmatrix}  \in \End_R (P \oplus Q) \otimes (\End_R(P \oplus
Q))^{\textrm{op}} $, we have \vspace*{-14pt}
\begin{eqnarray*}
\psi_{P} \circ j' (f \otimes g) = j'' \circ \psi_{P \oplus Q} (f
\otimes g) :& \End_R(P) \!\!\! & \lra \;\; \End_R(P) \\
& \al  \!\!\! & \lmps \;\; f_{11} \circ \al \circ g_{11}.
\end{eqnarray*}
So the diagram above commutes and we can show that $j'' \circ i'' =
\id_{\End (\End (P))}$ and $j' \circ i' = \id_{\End (P) \otimes
{\End (P)}^{\op}}$. Therefore to show that $\psi_P$ is an
isomorphism it is sufficient to show that $\psi_{P \oplus Q}$ is an
isomorphism.

Let $\{e_1, \ldots , e_n \}$ be a basis for the free $R$-algebra $P
\oplus Q \cong R^n$. Let $E_{ij} \in \End_R(R^n)$ be defined by
$E_{ij}(e_k) = \delta_{jk}e_i$. We can consider $E_{ij}$ as the $n
\times n$ matrix with $1$ in the $i$-$j$ entry, and zeros elsewhere.
Then 
$\{E_{ij} : 1 \leq i,j \leq n \}$ is an $R$-module basis for
$\End_R(R^n)$ and $\{E_{ij} \otimes E_{kl} : 1 \leq i,j,k,l \leq n
\}$ is an $R$-module basis for $\End_R(R^n) \otimes
(\End_R(R^n))^{\textrm{op}}$. By definition of $\psi_{P \oplus Q}$,
we have
\begin{equation} \label{deltaeijeqn}
\psi_{P \oplus Q} (E_{ij} \otimes E_{kl}) (E_{st}) = E_{ij} \circ
E_{st} \circ E_{kl} = \delta_{js} \delta_{tk} E_{il}.
\end{equation}

Using this, we will show that $\psi_{P \oplus Q}$ is an isomorphism.
Let $h : \End_R (R^n) \ra \End_R (R^n)$ be an $R$-module
homomorphism. For an arbitrary basis element $E_{xy} \in \End_R
(R^n)$, suppose $h(E_{xy})= \sum_{i,j} r_{ij}^{(x,y)} E_{ij}$, where
$r_{ij}^{(x,y)}$ is an element of $R$ indexed by $i,j,x,y$. Then
define a map \vspace*{-3pt}
\begin{align*}
\phi : \End_R ( \End_R(R^n)) & \lra \End_R(R^n) \otimes
(\End_R(R^n))^{\textrm{op}} \\
h & \lmps \sum_{x,y} \sum_{i,j} r_{ij}^{(x,y)} E_{ix} \otimes
E_{yj}.
\end{align*}
Then we can show that $\phi$ is an $R$-algebra homomorphism inverse
to $\psi_{P \oplus Q}$, completing the proof.
\end{proof}

\begin{example}\label{azumayaexamples}
For any commutative ring $R$, $M_n (R)$ is an Azumaya algebra over
$R$. This follows by  applying Proposition~\ref{endomorphismprop}:
since $R^n$ is free and therefore faithfully projective over $R$,
$\End_R (R^n)$ is an Azumaya algebra over $R$, and we know that $M_n
(R) \cong \End_R (R^n)$ as $R$-algebras.
\end{example}

In Proposition~\ref{tensorproductprop} below we show that the tensor
product of two $R$-Azumaya algebras is again an $R$-Azumaya algebra.
The proof is from Farb, Dennis \cite[Prop.~8.4]{farbdennis}.

\begin{prop} \label{tensorproductprop}
If $A$ and $B$ are Azumaya algebras over $R$, then $A \otimes_R B$
is an Azumaya algebra over $R$.
\end{prop}

\begin{proof} Since $A$ and $B$ are finitely generated projective
$R$-modules (by Proposition~\ref{faithfullyprojprop}), we can choose
$R$-modules $A'$ and $B'$ with $A \oplus A' \cong R^n$ and $B \oplus
B' \cong R^m$. Then $(A \oplus A') \otimes (B \oplus B') \cong R^n
\otimes R^m \cong R^{n m}$. So there is an $R$-module $Q$ with $(A
\otimes B) \oplus Q \cong R^{n m}$, proving $A \otimes B$ is
finitely generated and projective.

We know that any projective module is flat, so $A$ is flat over $R$.
Since $R \ra B; r \mapsto r. 1_B$ is injective, it follows that
\begin{eqnarray*}
f: A & \lra \;\;  A \otimes_R R & \lra \;\; A \otimes_R B \\  a &
\longmapsto \;\; a \otimes 1_R \;\, & \longmapsto \;\;  a \otimes
1_B
\end{eqnarray*}
is injective. Let $r \in R$ be such that $ r \in \Ann (A \otimes
B)$. Then $r(a \otimes 1)=0$ for all $a \in A$, so $0= r(f(a)) =
f(ra)$, and since $f$ is injective this implies that $ra=0$ for all
$a \in A$. Thus $r \in \Ann(A)$ and so $r = 0$, as $A$ is faithful,
proving $A \otimes B$ is also faithful. So by
Proposition~\ref{faithfullyprojprop}, $A \otimes B$ is faithfully
projective.

Let $\psi_{A \otimes B}$ be the homomorphism defined in the
definition of an Azumaya algebra. Then the following diagram is
commutative:
\begin{displaymath}
\xymatrix{ (A \otimes A^{\textrm{op}}) \otimes (B \otimes
B^{\textrm{op}}) \ar[d]_{\theta} \ar[rr]^{\;\;\;\; \psi_A \otimes
\psi_B} && \End_R(A) \otimes \End_R(B) \ar[d]^{w} \\
(A \otimes B) \otimes ( A \otimes B)^{\textrm{op}}
\ar[rr]_{\;\;\;\;\;\;\;\;\;\;\;\; \psi_{A \otimes B }} && \End_R (A
\otimes B) }
\end{displaymath}
Here $\psi_A$ and $\psi_B$ are isomorphisms since $A$ and $B$ are
Azumaya algebras, $w$ is the isomorphism given by
Proposition~\ref{homisoprop}, and the isomorphism $\theta$ comes
from the commutativity of the tensor product and the fact that $(A
\otimes B)^{\textrm{op}} \cong A^{\textrm{op}} \otimes
B^{\textrm{op}}$. This shows that $\psi_{A \otimes B}$ is an
isomorphism.
\end{proof}


\section{Further characterisations of Azumaya algebras}
\label{sectionfurthercharacterisations}

In this section, unless otherwise stated, $R$ denotes a commutative
ring and $A$ is an $R$-algebra. The definition of an Azumaya algebra
(Definition~\ref{azumayadefinition}) has a number of equivalent
reformulations which are shown in Theorem~\ref{azumayadefinthm}
below. We firstly require some additional definitions. 

We say that an $R$-algebra $A$ is \emph{central}\index{central
algebra}\index{algebra!central} over $R$ if $A$ is faithful as an
$R$-module and the centre of $A$ coincides with the image of $R$ in
$A$. Thus an $R$-algebra $A$ is central if and only if the ring
homomorphism $f :R \ra Z(A)$ is both injective and surjective.
%
%
%
A commutative $R$-algebra $S$ is said to be a \emph{finitely
presented algebra}\index{algebra!finitely presented}\index{finitely
presented algebra} if $S$ is isomorphic to the quotient ring $R [
x_1, \ldots , x_n ] / I$ of a polynomial ring $R [ x_1, \ldots , x_n
]$ by a finitely generated two-sided ideal $I$.
%
%
%
A commutative $R$-algebra $S$ is called \emph{\'{e}tale}\index{etale
algebra@\'{e}tale algebra}\index{algebra!etale@\'{e}tale} if $S$ is
flat, finitely presented and separable over $R$.

Let $R$ be a ring, which is not necessarily commutative. Consider a
covariant additive functor $T$ from the category of (left or right)
$R$-modules $\mathcal C$ to some category of modules $\mathcal D$.
The functor $T$ induces a function
$$
T_{X,Y} : \Hom_{\mathcal C} (X,Y) \lra \Hom_{\mathcal D} (T(X),T(Y))
$$
for every pair of objects $X$ and $Y$ in $\mathcal C$. The functor
$T$ is said to
be\index{functor!full}\index{functor!faithful}\index{functor!fully
faithful}\index{faithful functor}
\begin{itemize}
\item \emph{faithful} if $T_{X,Y}$ is injective;

\item \emph{full} if $T_{X,Y}$ is surjective;

\item \emph{fully faithful} if $T_{X,Y}$ is bijective;
\end{itemize}
for each $X$ and $Y$ in $\mathcal C$. An $R$-module $M$ is called a
\emph{generator}\index{module!generator}\index{generator module} if
the functor $\Hom_R(M,-)$ is a faithful functor from the category of
left $R$-modules to the category of abelian groups.


There is another definition of a generator module in
\cite[p.~5]{demeyer}, which is shown to be equivalent to the
definition above (see Proposition~\ref{generatorprop}). For a  ring
$R$ (not necessarily commutative) and any $R$-module $M$, consider
the subset $\mathcal{T}_R (M)$ of $R$ consisting of elements of the
form $\sum_i f_i (m_i)$ where the $f_i$ are from $\Hom_R (M,R)$ and
the $m_i$ are from $M$. Then $\mathcal T_R (M)$ is a two-sided ideal
of $R$, called the \emph{trace ideal}\index{trace ideal} of $M$.
Then \cite{demeyer} defines $M$ to be a generator module if
$\mathcal T_R(M) = R$.

\begin{prop} \label{generatorprop}
Let  $R$ be a (possibly non-commutative) ring. For any $R$-module
$M$, the functor $\Hom_R(M,-)$ is a faithful functor if and only if
$\mathcal T_R(M) = R$.
\end{prop}

\begin{proof}
See \cite[Thm.~18.8]{lam}.
\end{proof}

Let $R$ be a  ring, which is not necessarily commutative. An
$R$-module $M$ is called a \emph{projective
generator}\index{module!projective generator}\index{projective
generator module} (or \emph{progenerator}\index{progenerator
module}) if $M$ is finitely generated, projective and a generator.

\begin{thm}
Let $R$ be a commutative ring. Then an $R$-module $M$ is an
$R$-progenerator if and only if $M$ is finitely generated,
projective and faithful.
\end{thm}

\begin{proof}
See \cite[Cor.~I.1.10]{demeyer}.
\end{proof}


\begin{thm}
\label{azumayadefinthm}\index{Azumaya
algebra}\index{algebra!Azumaya} Let $R$ be a commutative ring and
$A$ be an $R$-algebra. The following are equivalent:
\begin{enumerate}
\item $A$ is an Azumaya algebra.

\item $A$ is central and separable as an $R$-module.

\item $A$ is central over $R$ and $A$ is a generator as an $A^e$-module.

\item The functors \begin{flalign*}& A \textrm{-} \mathcal{M}\mathfrak{od}
\textrm{-} A \lra \mathcal{M}\mathfrak{od} \textrm{-} R,\; M \longmapsto M^A \\
& \mathcal{M}\mathfrak{od} \textrm{-} R \lra A \textrm{-}
\mathcal{M}\mathfrak{od} \textrm{-} A,\; N \longmapsto N \otimes A
\end{flalign*}
are inverse equivalences of categories. Further projective modules
correspond to projective modules.

\item $A$ is a finitely generated $R$-module and $A/ \m A$ is a
central simple $R/ \m $-algebra for all $\m \in \Max(R)$.

\item There is a faithfully flat \'{e}tale $R$-algebra $S$ and a faithfully projective
$S$-module $P$ such that $A \otimes_R S \cong \End_S(P)$. If $R$ is
local, $S$ can be taken as finite \'{e}tale.

\end{enumerate}

\end{thm}

\begin{proof} The proof follows by combining various parts from
\cite[Thm.~II.3.4]{demeyer}, \cite[Thm.~III.4.1]{bass} and
\cite[Thm.~III.5.1.1]{knus}.

\vspace*{2pt}

(1) $\Lra$ (2):  See DeMeyer, Ingraham \cite[Thm.~II.3.4]{demeyer}.

\vspace*{2pt}

(1) $\Lra$ (3): See Bass \cite[Thm.~III.4.1]{bass}.

\vspace*{2pt}

(1) $\Lra$ (4) $\Lra$ (5) $\Lra$ (6): See Knus
\cite[Thm.~III.5.1.1]{knus}.
\end{proof}




We state another example of an Azumaya algebra, which follows from
the theorem above.

\begin{example} \label{azumayaexample3}
Let $R$ be a commutative ring in which $2$ is invertible. Define the
quaternion algebra $Q$  over $R$ to be the free $R$-module with
basis $\{ 1, i, j, k\}$ and with multiplication satisfying $i^2 =
j^2 = k^2 = -1$ and $ij = -ji = k.$ As for quaternion algebras over
fields (see \cite[Lemma~1.6]{pierce}), it follows that $Q$ is a
central $R$-algebra. Then \cite[Cor.~4]{szetoseparable} says that
$Q$ is separable over $R$. Their proof considers the following
element of $Q$:
$$
e = \frac{1}{4} ( 1 \otimes 1 - i  \otimes i - j \otimes j - k
\otimes k).
$$
It is routine to show that $e$ is a separability idempotent for $Q$,
so $Q$ is separable over $R$. It follows from
Theorem~\ref{azumayadefinthm}(2) that $Q$ is an Azumaya algebra over
$R$.
\end{example}

\begin{remark}\label{dimensionazumayasquare}
If $A$ is an Azumaya algebra over $R$, then by
Theorem~\ref{azumayadefinthm},  $A/ \m A$ is a central simple $R/ \m
$-algebra for all $\m \in \Max(R)$. If $A$ is free over $R$, then
$$
[A:R] = [R/\m \otimes_R A: R/ \m] = [ A/ \m A :R/ \m].
$$
Since we know the dimension of a central simple algebra is a square
number (see \cite[Cor.~8.4.9]{sch}), the same is true for $A$.
\end{remark}





We also remark that Azumaya algebras are closely related to
Polynomial Identity rings, thanks to the Artin-Procesi Theorem (see
\cite[\S1.8]{rowenpolynomialidentities},
\cite[\S6.1]{rowenringtheoryII}). While we do not consider
polynomial identity theory here, we mention a result of Braun in the
following theorem. In \cite{braunartinsthm}, Braun generalises the
Artin-Procesi Theorem and, as a consequence of this generalisation,
he gives another characterisation of Azumaya algebras in
\cite[Thm.~4.1]{braunartinsthm}, which we state in the theorem
below. The proof given below is a direct proof of this
characterisation, which is due to Dicks \cite{dicksazumaya}. The
notation used in the proof is defined in
Section~\ref{sectionseparable}.


\begin{thm} \label{braunscharacterisationthm}
Let $A$ be a central $R$-algebra. Then $A$ is an Azumaya algebra
over $R$ if and only if there is some $e \in A^e$ such that $e*1 =1$
and $e*A \subseteq R$.
\end{thm}

\begin{proof}
Using Theorem~\ref{azumayadefinthm}, $A$ is an Azumaya algebra over
$R$ if and only if $A$ is separable over $R$ (since we assumed
central in the statement of the theorem). Assume $A$ is an Azumaya
algebra over $R$. By Theorem~\ref{separablelemma}(5), this implies
there exists an idempotent $e \in A^e$ such that $e \ast 1 = 1$ and
$(a \otimes 1) e = (1 \otimes a)e$ for all $a \in A$. Let $a \in A$
be arbitrary. Then $a (e*A) = (a \otimes 1)*(e*A) = ((a \otimes 1)
e)*A = ((1 \otimes a) e)*A = (1 \otimes a)*(e*A) = (e*A)a$, proving
$e*A \subseteq Z(A) =R$.

Conversely, suppose there is some $e \in A^e$ with $e \ast 1 = 1$
and $e \ast A \subseteq R$. We want to show that $A$ is separable
over $R$; that is, there is an $e \in A^e$ such that $e*1 = 1$ and
$Je = 0$ (using Theorem~\ref{separablelemma}). We will first show
that $A^e e A^e = A^e$, where $A^e e A^e $ is the ideal of $A^e$
generated by $e$. Let $I = \{ a \in A : a \otimes 1 \in A^e e A^e
\}$, which is a two-sided ideal of $A$. If $I=A$, then as $1 \in A$,
$1 \otimes 1 \in A^e e A^e$ and so $A^e e A^e =A^e$.

Assume $I \neq A$. Then there is a maximal ideal $M$ of $A$ such
that $I \subseteq M \subsetneqq A$. We can give $A^e$ two left
$A^e$-module structures as follows:
\begin{align*}
A^e \times A^e &\lra A^e \\
(u, a \otimes b) & \longmapsto u *_1 (a \otimes b) = (u *a) \otimes b\\
(u, a \otimes b) & \longmapsto u *_2 (a \otimes b) = a \otimes (u *
b)
\end{align*}
It is routine to check that these are well-defined. Then $e *_2 (a
\otimes b) = a \otimes (e*b) = a (e*b) \otimes 1$ and we can show
that $A^e *_2(A^e e A^e) \subseteq A^e e A^e$. It follows that every
element of $e*_2 (A^e e A^e)$ is of the form $a \otimes 1$ for some
$a \in I \subseteq M$.

Let $\ov A = A/M$ and $\ov R = R / (R \cap M)$. Then $\ov R
\hookrightarrow \ov A$ and $ \ov R \subseteq Z(\ov A)$. Let $\ov
{A^e} = \ov A \otimes_{\ov R} \ov A^{\op}$. Then $\ov e * \ov 1 =
\ov 1$ and $\ov e * \ov A \subseteq \ov R$. To show $ \ov R = Z(\ov
A)$, let $\ov x \in Z(\ov A)$. Then $\ov x = \ov e * \ov x$, which
shows $\ov x \in \ov R$ since we know $\ov e * \ov x \in \ov R$.

Since $M$ is a maximal ideal of $A$, $\ov A$ is simple, and
therefore central simple, over $\ov R$. Also $\ov A^{\op}$ is
simple, so $\ov {A^e}$ is simple. Since $\ov e * \ov 1 = \ov 1$,
$\ov e \neq 0$ and thus $\ov {A^e} \ov e \ov{A^e}= \ov {A^e }$, as
$\ov {A^e} \ov e \ov{A^e}$ is the two-sided ideal generated by $\ov
e$. Then $\ov 1 \otimes \ov 1 = \ov e
*_2 (\ov 1 \otimes \ov 1) \in \ov e *_2  \ov {A^e} \ov e \ov{A^e} =
\ov{e *_2 {A^e} e A^e}$. But we observed above that every element of
$e *_2 {A^e} e A^e$ is of the form $a \otimes 1$ for some $a \in M$,
so every element of $\ov{e *_2 {A^e} e A^e}$ is of the form $\ov {a
\otimes 1} = \ov a \otimes \ov 1 =0$, and therefore $\ov{e *_2 {A^e}
e A^e} =0$. Since $\ov {1 \otimes 1} \in \ov{e *_2 {A^e} e A^e}$,
this is a contradiction, so our assumption that $I \neq A$ was
incorrect, proving ${A^e} e A^e = A^e$.

It remains to prove $Je = 0$. Let $u \in A^e$. We claim that ${A^e}
e A^e *_2 u \subseteq (u *_1 A^e) A^e$. Let $v, w \in A^e$ and let
$u = \sum_i a_i \otimes b_i$. Then
\begin{align*}
(vew)*_2 (\sum_i a_i \otimes b_i) &= \sum_i (a_i \otimes (vew *b_i))\\
&= \sum_i \bigg( a_i \otimes \Big(v*\big( e*(w*b_i)\big) \!\!\:
\Big)\! \bigg)
\end{align*}
But $e*(w*b_i) \in R$ as $e*A \subseteq R$, so $v*\big( e*(w*b_i)
\big) = (v*1)(ew*b_i)$. Then letting $ew = \sum_j c_j \otimes d_j$,
we have
\begin{align*}
(vew)*_2 (\sum_i a_i \otimes b_i) &= \sum_i a_i (ew*b_i) \otimes
(v*1)\\
&= \sum_i \sum_j (a_i c_j b_i d_j) \otimes (v*1)\\
&=\sum_j \bigg( \!\!\!\; \Big( \sum_i a_i \otimes b_i \Big) *_1
\Big( c_j \otimes (v*1)\Big) \!\!\!\;  \bigg) (d_j \otimes 1),
\end{align*}
proving the claim.

Since $*$ defines a left module action on $A$ and since $e* A
\subseteq R$, it follows that $Je* A = J* (e*A) \subseteq J*R$. Then
it is easy to see that $J*R =0$, and so $Je*A =0$. But $Je \subseteq
A^e *_2 (Je)= A^e e A^e *_2 (Je) \subseteq (Je*_1 A^e) A^e = 0$,
proving $Je=0$ as required.
\end{proof}

We recall that for two rings $A$ and $A'$, an
\emph{anti-homomorphism}\index{anti-homomorphism} of $A$ into $A'$
is a map $\sigma : A \ra A'$ satisfying the following conditions:
\begin{enumerate}
\item $\sigma(a + b) = \sigma(a) + \sigma(b)$ for all $a, b \in A$;

\item $\sigma (1_A) = 1_{A'}$;

\item $\sigma(ab) = \sigma(b) \sigma(a)$ for all $a, b \in A$.

\suspend{enumerate}%
If $A = A'$ and $\sigma : A \ra A$ is a bijective anti-homomorphism,
then $\sigma$ is called an anti-automorphism. An
\emph{involution}\index{involution} is an anti-automorphism
satisfying the additional condition: %
\resume{enumerate}

\item $\sigma^2 (a) = a$ for all $a \in A$.
\end{enumerate}
If $A$ is a central $R$-algebra, then an involution $\sigma$ of $A$
is said to be \emph{of the first kind} if the restriction of
$\sigma$ to $R$ is the identity map. We note that an involution of
the first kind is an $R$-linear involution; that is, the map $\sigma
:A \ra A$ is an $R$-linear map.

For a central $R$-algebra $A$ admitting an involution of the first
kind, Braun \cite[Thm.~5]{braundefiningaxioms} gives a further
characterisation of when $A$ is an Azumaya algebra, which is stated
in Theorem~\ref{braunsthminvolutionoffirstkind}. The theorem shows
that, in this setting, the condition that $A$ is a finitely
generated projective $R$-module is not required.

\begin{thm} \label{braunsthminvolutionoffirstkind}
Let $A$ be a central $R$-algebra admitting an involution of the
first kind. Then $A$ is an Azumaya algebra if and only if $\psi_A: A
\otimes_R A^{\op} \ra \End_R (A)$ is an $R$-linear isomorphism.
\end{thm}

\begin{proof}
See Braun \cite[Thm.~5]{braundefiningaxioms}.
\end{proof}

This theorem has been further generalised by Rowen
\cite[Cor.~1.7]{rowen}. For a central $R$-algebra $A$ admitting an
involution of the first kind, the result of Rowen proves that for
$A$ to be an Azumaya algebra, it is sufficient to assume that
$\psi_A$ is an epimorphism.


\section{The development of the theory of Azumaya algebras}
\label{sectiondevelopmentofazumayaalgebras}

The above results show some of the most general reformulations of
the definition of an Azumaya algebra to date. We will now consider
the historical development of these concepts, beginning with the
1951 paper of Azumaya \cite{az}.

In \cite{az}, Azumaya introduced the term ``proper maximally central
algebra'' (p.~128). An $R$-algebra $A$ which is free and finitely
generated as a module over $R$ is defined to be {\em proper
maximally central}\index{proper maximally central} over $R$ if $A
\otimes A^{\op}$ coincides with $\End_R (A)$. It is known that a
free $R$-module is both faithful and projective as a module over
$R$, so this definition implies that $A$ is faithfully projective as
an $R$-module, and $A \otimes A^{\op} \cong \End_R(A)$. This shows
that the definition of a proper maximally central algebra is
equivalent to Definition~\ref{azumayadefinition} under the
additional assumption that $A$ is free.

Assuming that $R$ is a Noetherian ring, Auslander and Goldman prove
Theorem~\ref{maxseparable} (see \cite{ausgold}, Cor.~4.5 and
Thm.~4.7). Further, Endo and Watanabe \cite[Prop.~1.1]{endo}
generalise this result by removing the Noetherian condition on $R$,
as stated above in Theorem~\ref{maxseparable}.

The equivalence of statements (1) and (2) of
Theorem~\ref{azumayadefinthm} was proven by Auslander and Goldman
\cite[Thm.~2.1]{ausgold} under the assumption that $A$ is a central
algebra over $R$, which we can see in Theorem~\ref{azumayadefinthm}
is not required. In Knus \cite[Thm.~III.5.1.1(2)]{knus},
Theorem~\ref{azumayadefinthm}(2) has the condition that $A$ is
finitely generated, which isn't required for the equivalence of the
statements. However, we note that given the equivalence of (1) and
(2), the fact that $A$ is finitely generated follows from
Proposition~\ref{separableprojectivefg}.


Part (3) of Theorem~\ref{azumayadefinthm} 
generalises a result of DeMeyer, Ingraham
\cite[Thm.~II.3.4(2)]{demeyer}. Their result says that $A$ is
central over $R$ and $A$ is a progenerator over $A^e$ if and only if
$A$ is an
Azumaya algebra over $R$. But we can see that 
$\Hom_{A^e}(A , -)$ being an exact functor is superfluous.

Azumaya \cite[Thm.~15]{az} proves the equivalence of
Theorem~\ref{azumayadefinthm} parts (1) and (5) with the extra
condition that $A$ is free over $R$. Further, with the assumption
that $A$ is projective, Bass \cite[Thm.~III.4.1]{bass} proves this
equivalence. Theorem~\ref{azumayadefinthm} shows that neither of
these extra conditions on $A$ are required for the equivalence of
these two statements. Bass \cite[Thm.~III.4.1]{bass} shows that the
definition of an Azumaya algebra is equivalent to the following:
\begin{quote}
There exists an $R$-algebra $S$ and a faithfully projective
$R$-module $P$ such that $A \otimes_R S \cong \End_R (P)$.
\end{quote}
Comparing this with Theorem~\ref{azumayadefinthm}(6), the result of
Knus \cite[Thm.~III.5.1.1(5)]{knus} refines this result by giving
that $S$ is a faithfully flat \'{e}tale $R$-algebra.

We prove our main theorem of Chapter~\ref{chapterktheoryofazalg},
Theorem~\ref{azumayafreethm}, for Azumaya algebras which are free
over their centres. So for an Azumaya algebra $A$ as originally
defined by Azumaya (with $A$ free over its centre $R$), our
Theorem~\ref{azumayafreethm} covers its $K$-theory.



\chapter{Algebraic {\itshape K}-Theory} \label{chktheory}



Algebraic $K$-theory defines a sequence of functors $K_i(R)$, for $i
\geq 0$, from the category of rings to the category of abelian
groups. The lower $K$-groups $K_0$, $K_1$ and $K_2$ were developed
in the 1950s and 60s by Grothendieck, Bass and Milnor respectively.
After much uncertainty, the ``correct'' definition of the higher
$K$-groups was given by Quillen in 1974. For an introduction to the
lower $K$-groups, see Magurn \cite{magurn} or Silvester
\cite{silvester}, and for an introduction to higher algebraic
$K$-theory, see Rosenberg \cite{rosenberg} or Weibel
\cite{weibelkbook}.

We begin this chapter by recalling the definitions of $K_0$, $K_1$
and $K_2$, and by observing some properties and examples of these
lower $K$-groups. We then specialise to the lower $K$-groups of
central simple algebras. Lastly, we look at the higher $K$-groups,
and note some of their properties.


\section{Lower {\itshape K}-groups} \label{sectionlowerkgroups}

\section*{{\itshape K}$\mathbf{_0}$}

Let $R$ be a ring. 
Let $\Proj (R)$ denote the monoid of isomorphism classes of finitely
generated projective $R$-mod\-ules, with direct sum as the binary
operation and the zero module as the identity element. Then $K_0
(R)$ is defined to be the free abelian group based on $ \mathcal{P}
\mathrm{roj}(R)$ modulo the subgroup generated by elements of the
form $[P] + [Q] - [P \oplus Q]$, for $P$, $Q \in \pr (R)$.
Alternatively, the group $K_0 (R)$ can be defined as the group
completion of the monoid $\Proj (R)$, which is shown in
\cite[Thm.~1.1.3]{rosenberg} to be equivalent to the definition
given here. By \cite[p.~5, Remarks]{rosenberg}, the group completion
construction forms a functor from the category of abelian semigroups
to the category of abelian groups.

\begin{thm} \label{k0functorgrphomomorphism}
For rings $R$ and $S$, let $T$ be an additive functor from the
category of $R$-modules to the category of $S$-modules, with $T(R)
\in \pr (S)$. Then $T$ restricts to an exact functor from $\pr (R)$
to $\pr (S)$, which induces a group homomorphism $f: K_0 (R) \ra K_0
(S)$ with $f([P]) = [T(P)]$ for each $P \in \pr (R)$.
\end{thm}

\begin{proof}
(See \cite[Prop.~6.3]{magurn}.) If $P \in \pr (R)$, then $P \oplus Q
\cong R^n$ for some $R$-module $Q$, and we have $T(P) \oplus T(Q)
\cong T(R)^n$. Since  $T(R)^n \in \pr (S)$, it follows that the
functor $T$ takes $\pr (R)$ to $\pr (S)$. All short exact sequences
in $\pr (R)$ are split, so the restriction of $T$ to $\pr (R) \ra
\pr (S)$ is an exact functor. Since functors preserve isomorphisms,
$T$ induces a monoid homomorphism
$$
\Proj (R) \lra \Proj (S); \, \big[P\big] \lmps \big[T(P)\big].
$$
Since the group completion is functorial, there is a group
homomorphism $K_0 (R) \ra K_0 (S);$ $\big[P\big] \mps
\big[T(P)\big].$
\end{proof}

For rings $R$ and $S$, if $\phi :R \ra S$ is a ring homomorphism,
then $S$ is a right $R$-module via $\phi$. There is an additive
functor
$$
S \otimes_R - : R \Mod \lra S \Mod
$$
which, by Theorem~\ref{k0functorgrphomomorphism}, induces a group
homomorphism $f:K_0 (R) \ra K_0 (S)$; $f([P]) = [S \otimes_R P]$ for
each $P \in \pr (R)$.

\begin{thm} \label{k0functorthm}
$K_0$ is a functor from the category of rings to the category of
abelian groups.
\end{thm}

\begin{proof}
See \cite[Thm.~6.22]{magurn} for the details.
\end{proof}

\begin{thm} \label{k0rtimessthm}
For rings $R$ and $S$, there is a group isomorphism $K_0 (R \times
S) \cong K_0 (R) \times K_0 (S)$.
\end{thm}

\begin{proof}
See \cite[Thm.~6.6]{magurn}.
\end{proof}

Recall that rings $R$ and $S$ are Morita equivalent\index{Morita
equivalence}, if $R \Mod$ and $S \Mod$ are equivalent as categories;
that is, there are functors
\begin{align*}
\xymatrix{ R \Mod \ar@<.7ex>[r]^{T} & S \Mod \ar@<1ex>[l]^{U} }
\end{align*}
with natural equivalences $T \circ U \cong \id_S$ and $U \circ T
\cong \id_R$. It follows from \cite[Prop.~II.10.2]{mitchell} that
$T$ and $U$ are additive functors. In particular, $M_n (R)$ is
Morita equivalent\label{moritaequivalenceofmnrandr} to $R$ since the
functors
\begin{align*}
R^n \otimes_{M_n(R)} - : M_n(R) \;\! \Mod & \lra  R \Mod \\
\textrm{ and \;\;\;\;\;\; }  R^n \otimes_R - : R \Mod & \lra  M_n(R) \;\! \Mod 
\end{align*}
form mutually inverse equivalences of categories.

\begin{thm} \label{k0moritathm}
For rings $R$ and $S$, if $R$ and $S$ are Morita equivalent then
$K_0(R) \cong K_0(S)$. In particular, 
$K_0 (R) \cong K_0 (M_n(R))$.
\end{thm}

\begin{proof}
(See \cite[Cor.~II.2.7.1]{weibelkbook}.) This follows from
Theorem~\ref{k0functorgrphomomorphism} since the functors
\begin{align*}
\xymatrix{ R \Mod \ar@<.7ex>[r]^{T} & S \Mod \ar@<1ex>[l]^{U} }
\end{align*}
induce an equivalence between the categories $\pr (R)$ and $\pr
(S)$, which induces a group isomorphism on the level of $K_0$.
\end{proof}

\begin{examples} \label{k0examples}
\begin{enumerate}
\item If $R$ is a field or a division ring, then $K_0(R) \cong
\mathbb Z$. A finitely generated projective module over $R$ is free,
with a uniquely defined dimension. Since any two finite dimensional
free modules with the same dimension are isomorphic, we have $\Proj
(R) \cong \mathbb N$, so $K_0 (R) \cong \mathbb Z$.

\item Recall that a ring $R$ is a semi-simple ring\index{semi-simple
ring}\index{ring!semi-simple} if it is semi-simple as a left module
over itself; that is, $R$ is a direct sum of simple $R$-submodules.
By the Artin-Wedderburn theorem, $R$ is isomorphic to $M_{n_1}(D_1)
\times \cdots \times M_{n_r}(D_r)$, where each $D_i$ is a division
ring and each $n_i$ is a positive integer. Then using
Theorems~\ref{k0rtimessthm} and \ref{k0moritathm},
$$
K_0(R) \cong \bigoplus K_0\big(M_{n_i}(D_i)\big) \cong \bigoplus K_0
(D_i) \cong {\mathbb Z}^r.
$$


\item If $R$ is a principal ideal domain\index{principal ideal
domain} (a commutative integral domain in which every ideal can be
generated by a single element), or if $R$ is a local ring, then $K_0
(R) \cong \mathbb Z$. For a proof of these, see \cite{rosenberg},
Thm.~1.3.1 and Thm.~1.3.11 respectively.


\end{enumerate}
\end{examples}


\section*{{\itshape K}$\mathbf{_1}$}

Let $\GL_n (R)$, the general linear group\index{general linear
group} of $R$, denote the group of invertible $n \times n$ matrices
with entries in $R$ and matrix multiplication as the binary
operation. Then $\GL_n (R)$ is embedded into $\GL_{n+1} (R)$ via
\begin{align*}
\GL_n (R) & \lra \GL_{n+1} (R) \\
M_n & \lmps \left( \begin{array}{cc}
M_n & 0  \\
0 & 1 \end{array} \right).
\end{align*}
The union of the resulting sequence $\GL_1 (R) \subset \GL_2 (R)
\subset \cdots \subset \GL_n (R) \subset \cdots$ is called the
infinite general linear group $\GL(R) = \bigcup_{n=1}^{\infty} \GL_n
(R)$. Then $K_1 (R)$ is defined to be the abelianisation of $\GL
(R)$; that is,
\begin{align*}
K_1 (R) = \GL (R) / [\GL (R) , \GL (R)].
\end{align*}
By \cite[Prop.~9.3]{magurn}, the abelianisation construction forms a
functor from the category of groups to the category of abelian
groups.

The elementary matrix\index{elementary matrix} $e_{ij}(r)$, for $r
\in R$ and $i \neq j$, $1 \leq i,j \leq n$, is defined to be the
matrix in $\GL_n(R)$ which has $1$'s on the diagonal, $r$ in the
$i$-$j$ entry and zeros elsewhere. Then $E_n(R)$ denotes the
subgroup of $\GL_n(R)$ generated by all elementary matrices
$e_{ij}(r)$ with $1 \leq i , j \leq n$. We note that the inverse of
the matrix $e_{ij}(r)$ is $e_{ij}(-r)$ and the elementary matrices
satisfy the following relations:\label{elementarymatrices}
\begin{align*}
e_{ij}(r) \, e_{ij}(s) \, & = \; e_{ij}(r+s) \\
[e_{ij}(r), e_{kl}(s)]  &= \left\{
\begin{array}{l l}
1 & \quad \mbox{if $j \neq k$, $i \neq l$}\\
e_{il}(rs) & \quad \mbox{if $j=k$, $i \neq l$}\\
e_{kj}(-sr) & \quad \mbox{if $j \neq k$, $i = l$}.
\end{array} \right.
\end{align*} Then $E_n(R)$ embeds in $E_{n+1}(R)$, and $E(R)$ is the
infinite union of the $E_n(R)$. The Whitehead Lemma (see
\cite[Lemma~9.7]{magurn}) states that $[\GL (R) , \GL (R)] = E(R)$,
so it follows that
$$
K_1 (R) = \GL (R) / E (R).
$$


\begin{thm} \label{k1functor}
For rings $R$ and $S$ and a ring homomorphism $\phi :R \ra S$, there
is an induced group homomorphism $K_1 (R) \ra K_1 (S)$ taking
$(a_{ij} ) E(R)$ to $(\phi(a_{ij})) E(S)$. Then $K_1$ is a functor
from the category of rings to the category of abelian groups.
\end{thm}

\begin{proof}
A ring homomorphism $\phi :R \ra S$ defines a group homomorphism
$\GL(R) \ra \GL(S)$; $(a_{ij}) \mps (\phi(a_{ij}))$. Then there is
an induced group homomorphism $K_1 (R) \ra K_1 (S)$ taking $(a_{ij}
) E(R)$ to $(\phi(a_{ij})) E(S)$. See \cite[Prop.~9.9]{magurn} for
the details.
\end{proof}



\begin{thm} \label{k1moritathm}
For a ring $R$ and a positive integer $n$, $K_1 (M_n (R) ) \cong K_1
(R)$.
\end{thm}

\begin{proof}
(See the proof of \cite[Thm.~1.2.4]{rosenberg}.) Since $M_k (M_n
(R)) \cong M_{kn} (R)$, we have a commutative diagram:
\[
\xymatrix{ \GL_k (M_n (R)) \ar[rr]^{\cong} \ar[d] && \GL_{kn} (R) \ar[d] \\
\GL_{k+1}(M_n(R)) \ar[rr]_{\cong} && \GL_{(k+1)n} (R). }
\]
It follows that $\GL (M_n (R)) \cong \GL (R)$, which induces the
required isomorphism on the level of $K_1$.
\end{proof}

\begin{examples} \label{k1examples}
\begin{enumerate}
\item If $R$ is a field then  $K_1(R) \cong R^{\! \ast}$.  See
\cite[Eg.~III.1.1.2]{weibelkbook} for the details. This follows
since, for a commutative ring $R$, the determinant $\det : \GL (R)
\ra R^*$ induces a split surjective group homomorphism ${\det}: K_1
(R) \ra R^*$; $(a_{ij}) E(R) \mps \det(a_{ij})$, with right inverse
given by $R^* \ra K_1(R)$; $a \mps a E(R)$. When $R$ is a field, the
kernel of the map ${\det}$ is trivial, so $K_1 (R) \cong R^*$.

\item If $R$ is a division ring, then $K_1(R) \cong R^{\! \ast}/
[ R^{\! \ast}, R^{\! \ast}]$. See \cite[Eg.~III.1.3.5]{weibelkbook}
for the details. This follows from the Dieudonn\'{e}
determinant\index{Dieudonn\'{e} determinant} (see \cite[p.~124]{silvester}
). The Dieudonn\'{e} determinant defines a surjective group
homomorphism $\operatorname{Det}: \GL_n (R) \ra R^*/ [R^*, R^*]$
with kernel $E_n(R)$, such that
$$
\operatorname{Det} 
\begin{pmatrix}
x_1 &&0\\
& \ddots & \\
0&& x_n
\end{pmatrix}
= \prod_i x_{i} [R^*, R^*]
$$
and the following diagram commutes:
\begin{displaymath}
\xymatrix{ \GL_n (R) \ar[rr] \ar[dr]_{\operatorname{Det}} && \GL_{n+1} (R) \ar[dl]^{\operatorname{Det}} \\
& R^*/ R' &   }
\end{displaymath}
It follows that for a division ring $R$, $K_1(R) \cong R^{\! \ast}/
[ R^{\! \ast}, R^{\! \ast}]$. If $R$ is a field, then the
Dieudonn\'{e} determinant coincides with the usual determinant.


\item If $R$ is a semi-local ring (subject to some conditions,
detailed below), then $K_1(R) \cong R^{\! \ast}/ [ R^{\! \ast},
R^{\! \ast}]$. See \cite[Thm.~2]{vaserstein} for the details. Recall
that $R$ is a semi-local ring\index{semi-local
ring}\index{ring!semi-local} if $R/\mathrm{rad}(R)$ is semi-simple,
where $\mathrm{rad}(R)$ denotes the Jacobson radical of $R$. By the
Artin-Wedderburn Theorem, $R/\mathrm{rad}(R)$ is isomorphic to a
finite product of matrix rings over division rings. We assume that
none of these matrix rings are isomorphic to $M_2 (\mathbb Z / 2
\mathbb Z)$ and that no more than one of the matrix rings has order
$2$.
The required result follows since there is a Whitehead determinant
which gives a surjective group homomorphism $R^* \ra K_1 (R)$.
Vaserstein \cite{vaserstein} shows that the kernel of this map is
$[R^*, R^*]$, giving the required isomorphism.

\end{enumerate}
\end{examples}


\section*{{\itshape K}$\mathbf{_2}$}

For $n \geq 3$, the Steinberg group\index{Steinberg group}
$\St_n(R)$ of $R$ is the group defined by generators $x_{ij}(r)$,
with $i$,~$j$ a pair of distinct integers between $1$ and $n$ and $r
\in R$, subject to the following relations which are called the
Steinberg relations:
\begin{align*}
x_{ij}(r) \, x_{ij}(s) \, & = \; x_{ij}(r+s) \\
[x_{ij}(r), x_{kl}(s)]  &= \left\{
\begin{array}{l l}
1 & \quad \mbox{if $j \neq k$, $i \neq l$}\\
x_{il}(rs) & \quad \mbox{if $j=k$, $i \neq l$}\\
x_{kj}(-sr) & \quad \mbox{if $j \neq k$, $i = l$}.
\end{array} \right.
\end{align*}
We observed on page~\pageref{elementarymatrices} that the elementary
matrices satisfy the Steinberg relations. So there is a surjective
homomorphism $\St_n(R) \ra E_n(R)$ sending $x_{ij} (r)$ to
$e_{ij}(r)$. As the Steinberg relations for $n+1$ include the
Steinberg relations for $n$, there are natural maps $\St_n(R) \ra
\St_{n+1} (R)$. We write $\St(R)$ for the direct limit
$\underrightarrow{\textrm{lim}} \, \St_n (R)$. Then there is a
canonical surjective map $\St(R) \ra E(R)$, and $K_2 (R)$ is defined
to be the kernel of this map.


\begin{thm}
$K_2$ is a functor from the category of rings to the category of
abelian groups.
\end{thm}

\begin{proof}
See \cite[Prop.~12.6]{magurn}. 
\end{proof}

\begin{thm}
For rings $R$ and $S$, if $R$ and $S$ are Morita equivalent then
$K_2(R) \cong K_2(S)$. In particular, since $M_n (R)$ is Morita
equivalent to $R$, $K_2 (R) \cong K_2 (M_n(R))$. 
\end{thm}

\begin{proof}
See \cite[Cor.~III.5.6.1]{weibelkbook}. 
\end{proof}

\begin{example}
\begin{enumerate}
\item \hspace*{-3pt} If $R$ is a finite field, then $K_2(R) =0$. See
\cite[Cor.~4.3.13]{rosenberg}.

\item If $R$ is a field, then Matsumoto's Theorem says that
$$
K_2(R) = R^* \otimes_{\mathbb Z} R^* / \left< a \otimes (1-a) : a
\neq 0,1 \right> .
$$
See \cite[Thm.~4.3.15]{rosenberg}.

\item If $R$ is a division ring, let $U_R$ denote the group generated by
$c(x,y)$, $x , y \in R^*$, subject to the relations:
\begin{description}
\item[](U$0$) \;\;\; $c (x , 1-x) =1  \;\;\; (x \neq 1,0)$,

\item[](U$1$) \;\;\; $c(xy, z) = c(xyx^{-1}, x z x^{-1})\, c(x,z)$

\item[](U$2$) \;\;\; $c(x,y z)\, c(y,zx)\, c(z, xy) =1.$
\end{description}
Then there is an exact sequence
$$
0 \lra K_2 (R) \lra U_R \lra [R^*, R^*] \lra 0
$$
where $U_R \ra [R^*, R^*]$; $c(x,y) \mps [x,y]$. See Rehmann
\cite[Cor.~2, p.~101]{rehmann}.
\end{enumerate}
\end{example}


\section{Lower {\itshape K}-groups of central simple algebras}
\label{sectionlowerkgroupsofcsalgebras}

We now specialise to central simple algebras. In this section, all
central simple algebras are assumed to be finite dimensional. Let
$F$ be a field and let $A$ be a central simple algebra over $F$.
Clearly there is a ring homomorphism $F \ra A$. Since $K_0$ is a
functor from the category of rings to the category of abelian groups
(see Theorem~\ref{k0functorthm}), we have an exact sequence
\begin{equation} \label{exactseqk0}
0 \lra \ZK[0](A) \lra K_0(F) \lra K_0(A) \lra \CK[0](A) \lra 0
\end{equation}
where $\ZK[0](A)$ and $\CK[0](A)$ are the kernel and cokernel of the
map $K_0(F) \rightarrow K_0(A)$, respectively.

In this section, using the definition of $K_0$, we will show what
sequence \eqref{exactseqk0} looks like. Using this, we will observe
that $\CK[0](A)$ and $\ZK[0](A)$ are torsion abelian groups, and
that $\CK[0]$ and $\ZK[0]$ are functors which do not respect Morita
equivalence, from the category of central simple algebras over $F$
to the category of abelian groups. We have a similar exact sequence
for $K_1$, and we show that the same results hold. In
Chapter~\ref{chapterktheoryofazalg}, we will generalise this result
to cover Azumaya algebras which are free over their centres, and to
cover all $K_i$ groups for $i \geq 0$. This is the key to proving
that $K_i(A) \otimes \mathbb Z[1/n] \cong K_i(R) \otimes \mathbb
Z[1/n]$ for an Azumaya algebra $A$ free over its centre $R$ of
dimension $n$ (see Theorem~\ref{azumayafreethm}).



\section*{{\itshape K}$\mathbf{_0}$}

Let $A$ be a central simple algebra over a field $F$. Wedderburn's
theorem says that $A$ is isomorphic to $M_n(D)$ for a unique
division algebra $D$ and a unique positive integer $n$. Since $K_0$
is a functor from the category of rings to the category of abelian
groups, we have $K_0 (A) \cong K_0 \big( M_n (D) \big)$, and
similarly for $\CK[0]$ and $\ZK[0]$. Writing $M_n (D)$ instead of
$A$, sequence~\eqref{exactseqk0} can be written as
\begin{equation*}
0 \lra \ZK[0]\big(  M_n (D)  \big) \lra K_0(F) \lra K_0 \big( M_n
(D)  \big) \lra \CK[0] \big( M_n (D)  \big) \lra 0.
\end{equation*}
From the ring homomorphism $F \ra M_n (D)$ and since $M_n (D)$ is
Morita equivalent to $D$, there are induced functors
\begin{equation*} \label{prfeqn}
\begin{array}{rlcll}
\pr (F) & \lra & \pr ( M_n (D) ) & \lra & \pr (D)\\
P \cong F^k & \longmapsto &  M_n (D) ^k & \lmps & D^n
\otimes_{M_n(D)} {M_n(D)}^k \cong D^{kn},
\end{array}
\end{equation*}
where every finitely generated projective module $P$ over $F$ is
free. By Theorem~\ref{k0functorgrphomomorphism}, these induce group
homomorphisms
\begin{equation*}
\begin{array}{rlcll}
K_0(F) & \stackrel{\gamma}{\lra} & K_0\big( M_n (D) \big) &
\stackrel{\delta}{\lra}  & K_0(D)\\
\big[ \, F^k \, \big] & \lmps & \left[ \,  M_n (D) ^k \, \right] &
\lmps & \left[ \, D^{kn} \, \right] \\
\end{array}
\end{equation*}
where, for $\left[ \, X \, \right] \in K_0 \big(  M_n (D)  \big)$,
$\de \left( \left[ \, X \, \right] \right)$ is defined to be $\left[
\, D^n \otimes_{M_n (D)}  X \, \right]$.

Since $F$ is a field and $D$ is a division ring, from
Example~\ref{k0examples}(1), $K_0 (F) \cong \mathbb Z$; $[F^k] \mps
k$ and $K_0 (D) \cong \mathbb Z$; $[D^k] \mps k$. We have a
commutative diagram:
\begin{displaymath}
\xymatrix{0 \ar[r] & \ZK[0]\! \big( M_n (D) \big)  \ar[d] \ar[r] &
K_0(F) \ar[d]_{\id} \ar[r]^{\gamma \;\;\;\;\;} & K_0 \big( M_n (D)
\big) \ar[d]^{\delta} \ar[r] & \CK[0] \! \big( M_n (D) \big)  \ar[d]
\ar[r] & 0 \\
0 \ar[r] & \ker(\de \! \circ \! \ga) \ar[r]   \ar[d] & K_0(F)
\ar[r]_{\delta \circ \ga \;\;} \ar[d]_{\cong}
& K_0(D) \ar[r]  \ar[d]^{\cong} & \coker(\de \! \circ \! \ga)  \ar[d] \ar[r] & 0 \\
0 \ar[r] & 0 \ar[r] & \mathbb Z \ar[r]_{\eta_n \;\;} & \mathbb Z
\ar[r] & \mathbb Z_{n} \ar[r] & 0}
\end{displaymath}
where the map $\eta_n : \mathbb Z \ra \mathbb Z$ is defined by
$\eta_n (k ) = kn$.

We know that $M_n(D)$ is Morita equivalent to $D$, so $\de$ is an
isomorphism and thus all the vertical maps in the above commutative
diagram are isomorphisms. So the exact sequence~\eqref{exactseqk0}
can be written as:
\begin{equation} \label{exactseqk0ofcsalgebra}
0 \lra  \mathbb Z \stackrel{\eta_n}{\lra} \mathbb Z \lra \mathbb Z/n
\mathbb Z \lra 0
\end{equation}
since $\ZK[0]\big( M_n (D) \big) \cong \ker (\eta_n)= 0$ and
$\CK[0]\big( M_n (D) \big) \cong \coker (\eta_n) = \mathbb Z_n$. A
group $G$ is said to be \emph{$n$-torsion}\index{torsion group} if
$x^n= e$ for all $x \in G$. We note that clearly both $\ZK[0]\big(
M_n (D) \big)$ and $\CK[0]\big( M_n (D) \big)$ are $n$-torsion
groups.

Let $F$ be a fixed field. We will observe that $\CK[0]$ and $\ZK[0]$
form functors from the category of central simple algebras over $F$
(with $F$-algebra homomorphisms) to the category of abelian groups.
For any central simple algebra $A$ over $F$, $\CK[0](A)$ is defined
to be the cokernel of the map $K_0(F) \ra K_0 (A)$; $[P] \mps [A
\otimes_F P]$, and $\ZK[0]$ is its kernel. They are clearly both
abelian groups since $K_0 (F)$ and $K_0 (A)$ are abelian groups. For
any two central simple algebras $A$ and $B$ over $F$ with $\phi :A
\ra B$ an $F$-algebra homomorphism, then $\phi$ restricted to $F$
gives the identity map on $F$. There is a commutative diagram:
\begin{displaymath}
\xymatrix{0 \ar[r] & \ZK[0](A) \ar[d] \ar[r] & K_0(F) \ar[d]_{\id}
\ar[r] & K_0(A) \ar[d]^{\phi} \ar[r] & \CK[0](A) \ar[d]
\ar[r] & 0 \\
0 \ar[r] & \ZK[0](B) \ar[r]  & K_0(F) \ar[r]  & K_0(B) \ar[r] &
\CK[0](B)   \ar[r] & 0 }
\end{displaymath}
where $\phi :K_0 (A) \ra K_0 (B)$ is defined by $\phi ([P]) = [B
\otimes_A P]$ for each $[P] \in K_0 (A)$. One can easily check that
$\CK[0]$ and $\ZK[0]$ form the required functors.

We also observe that $\CK[0]$ does not respect Morita invariance.
For a division algebra $D$ over the field $F$, we know that $D$ is
Morita equivalent to $M_n (D)$. For $D$, the exact
sequence~\eqref{exactseqk0} can be written as
$$
0 \lra 0 \lra \mathbb Z \stackrel{\id}{\lra} \mathbb Z \lra 0 \lra 0
$$
since the homomorphism $K_0 (F) \ra K_0 (D)$ maps $[F^k ]$ to $[D
\otimes_F F^k] = [D^k]$. From sequence~\eqref{exactseqk0ofcsalgebra}
we saw that $\CK[0](M_n(D)) \cong {\mathbb Z}_n$, which is clearly
not isomorphic to $\CK[0](D) =0$.





\section*{{\itshape K}$\mathbf{_1}$}



Let $A$ be a central simple algebra over a field $F$, such that $A$
is isomorphic to $M_r(D)$ for a unique division algebra $D$ and a
unique positive integer $r$. Since $K_1$ is a functor from the
category of rings to the category of abelian groups (see
Theorem~\ref{k1functor}), we will write $M_r (D)$ instead of $A$,
since $K_1 (A) \cong K_1 \big( M_r (D) \big)$. We have an exact
sequence
\begin{equation}\label{exactseqk1}
1 \lra \ZK[1]\big( M_r (D) \big) \lra K_1(F) \lra K_1\big( M_r (D)
\big) \lra \CK[1]\big( M_r (D) \big) \lra 1,
\end{equation}
where $\ZK[1]\big( M_r (D) \big)$ and $\CK[1]\big( M_r (D) \big)$
are the kernel and cokernel of the map $K_1(F) \rightarrow K_1 \big(
M_r (D) \big)$ respectively.

The ring homomorphism $F \ra  M_r(D)$; $f \mps  f \cdot I_r$ induces
a map
\begin{align*}
\GL(F) & \lra  \GL(M_r(D))  \\
(a_{ij}) &\lmps (a_{ij} I_r) .
\end{align*}
Using Theorems  \ref{k1functor} and \ref{k1moritathm}, there are
induced group homomorphisms
\begin{equation*}
\begin{array}{rcccl}
K_1(F) & \stackrel{\gamma}{\lra} & K_1 \big( M_r (D) \big) &
\stackrel{\de}{\lra} & K_1 (D)\\
(a_{ij})E(F) & \lmps & ( {a_{ij}} I_r)E(M_r (D)) & &\\
&&(d_{ij}) E( M_r(D)) & \lmps & ( d_{ij})E(D).
\end{array}
\end{equation*}
So we have a commutative diagram:
\begin{displaymath}
\xymatrix{1 \ar[r] & \ZK[1]\! \big( M_r (D) \big) \ar[d] \ar[r] &
K_1(F) \ar[d]_{\id} \ar[r]^{\gamma \;\;\;\;} & K_1\big( M_r (D)
\big) \ar[d]^{\delta} \ar[r] & \CK[1]\! \big( M_r (D) \big) \ar[d]
\ar[r] & 1 \\
1 \ar[r] & \ker(\de \circ \ga) \ar[r] \ar[d] & K_1(F) \ar[r]_{\delta
\circ \ga \;\;} \ar[d]_{\det} & K_1(D) \ar[r]
\ar[d]^{\operatorname{Det}} & \coker(\de \circ \ga)
\ar[d] \ar[r] & 1 \\
1 \ar[r] & D' \cap {F^*}^r \ar[r] & F^* \ar[r]_{\beta \;\;\;\;} &
D^*/D' \ar[r] & D^*/{F^*}^r D' \ar[r] & 1.}
\end{displaymath}
The maps $\det$ and $\Det$ come from Examples~\ref{k1examples} and
the map $\beta :F^* \ra D^*/D'$ is defined by $\beta(f) = \Det \circ
(\de \! \circ \! \ga) \! \circ {\det}^{-1} (f) = f^r D'$, where $D'
=[D^*,D^*]$.

We saw in Examples~\ref{k1examples}(1),(2)  that $K_1 (F) \cong
F^*$ and $K_1 (D) \cong D^*/ D'$. 
So all the vertical maps in the above diagram are isomorphisms and
the exact sequence \eqref{exactseqk1} can be written as
\begin{equation} \label{exactseqk1ofcsalgebra}
1 \lra D' \cap {F^*}^r \lra  F^* \lra  D^*/D' \lra  D^*/{F^*}^r D'
\lra 1.
\end{equation}
We will show that the groups $\ZK[1]\!\big( M_r (D) \big) \cong D'
\cap {F^*}^r$ and $\CK[1]\!\big( M_r (D) \big) \cong D^*/{F^*}^r D'$
are both $nr$-torsion groups for $n = \ind (D)$, where $\ind (D)$,
the index of $D$ over $F$, is defined to be the square root of the
dimension of $D$ over $F$. Note that by \cite[Cor.~8.4.9]{sch}, the
dimension of $D$ over $F$ is a square number.

\begin{lemma}\label{divisionalgebraofindexnlemma}
Let $D$ be a division algebra with centre $F$ of index $n$. Then for
any $a \in D$, $a^n = \Nrd_D (a) d_a$ where $d_a \in D'$.
\end{lemma}

\begin{proof}
Let $a \in D$ be arbitrary and let $f_a (x)$ be the minimal
polynomial of $a$ of degree $m$.  By Wedderburn's Factorisation
Theorem (see \cite[Thm.~16.9]{lam1}), $f_a (x) = (x - d_1 a
{d_1}^{-1}) \cdots (x- d_m a {d_m}^{-1})$ where $d_i \in D$, and by
\cite[p.~124, Ex.~1]{rei}, we have ${f_a (x)}^{n/m} = x^n - \Trd_{D}
(a)x^{n-1} + \cdots + (-1)^n \Nrd_D (a)$. Then
\begin{align*}
\Nrd_D (a) &= (d_{1} a d_{1}^{-1} \cdots d_{m} a d_{m} ^{-1})^{(n/m)} \\
&= ([d_1 , a] a [d_2 , a] a \cdots a [d_m , a] a)^{(n/m)} \\
&= a^n {d'}_a \;\;\;\;\; \mathrm{where} \;\; {d'}_a \in D'.
\end{align*}
So $a^n = \Nrd_D (a) d_a$ for some $d_a \in D'$.
\end{proof}

For $a^r \in \ZK[1] \! \big( M_r (D) \big) \cong D' \cap {F^*}^r$,
we have $a^r \in {F^*}^r \subseteq F^*$, so $(a^r)^{n} = \Nrd_D
(a^r)$ where $n = \ind (D)$. Since $a^r \in D'$, $\Nrd_D (a^r) = 1$,
proving $\ZK[1]\!\big( M_r (D) \big)$ is $nr$-torsion. The
equivalent result for $\CK[1] \!\big( M_r (D) \big)$ will follow
from the previous lemma. For $\ov a \in \CK[1] \!\big( M_r (D) \big)
= D^*/{F^*}^r D'$, we have $a \in D^*$ with $a^n = \Nrd_D (a) d_a$
for $d_a \in D'$ (by Lemma~\ref{divisionalgebraofindexnlemma}).
Since $\Nrd_D : D^* \ra F^*$, $\Nrd_D (a) \in F^*$ and therefore
$a^n \in F^* D'$. Then $a^{nr}  \in {F^*}^r D'$, which shows ${\ov
a}^{nr}= 1$, proving $\CK[1]\! \big( M_r (D) \big)$ is also
$nr$-torsion.

Let $F$ be a fixed field. As for $K_0$, we will observe that
$\CK[1]$ and $\ZK[1]$ form functors from the category of central
simple algebras over $F$  to the category of abelian groups. For any
two central simple algebras $A$ and $B$ over $F$ with an $F$-algebra
homomorphism $\phi :A \ra B$, there is a commutative diagram:
\begin{displaymath}
\xymatrix{1 \ar[r] & \ZK[1](A) \ar[d] \ar[r] & K_1(F) \ar[d]_{\id}
\ar[r] & K_1(A) \ar[d]^{\phi} \ar[r] & \CK[1](A) \ar[d]
\ar[r] & 1 \\
1 \ar[r] & \ZK[1](B) \ar[r]  & K_1(F) \ar[r]  & K_1(B) \ar[r] &
\CK[1](B)   \ar[r] & 1 }
\end{displaymath}
where $\phi :K_1 (A) \ra K_1 (B)$ is defined by $\phi \big((a_{ij})
E(A)\big) = \big(\phi(a_{ij})\big) E(B)$ for each $(a_{ij}) E(A) \in
K_1 (A)$. It follows that $\CK[1]$ and $\ZK[1]$ are functors from
the category of central simple algebras over $F$ to the category of
abelian groups.

We also observe that $\CK[1]$ does not respect Morita invariance.
For a division algebra $D$ over the field $F$, the exact sequence
\eqref{exactseqk1} can be written as
$$
1 \lra D'  \cap F^* \lra F^* {\lra} D^*/D' \lra D^*/F^* D' \lra 1
$$
since the map $K_1 (F) \ra K_1 (D)$ takes $(a_{ij})E(F)$ to
$(a_{ij})E(D)$. From sequence \eqref{exactseqk1ofcsalgebra} we saw
that $\CK[1](M_r(D)) = D^*/{F^*}^r D'$, but $\CK[1](D) = D^*/F^*
D'$. In general they are not isomorphic (see for example
\cite[Eg.~7]{hazwadmax}).

Let us also mention another group which exhibits some similar
properties to $\CK[1] (D)$. For a central simple algebra $A$ over a
field $F$, the group $G(A) = A^*/ (A^*)^2$, called the square class
group of $A$, has been studied by Lewis and Tignol
\cite{lewistignol}. They note that the group $G(A)$ is a torsion
abelian group of exponent two. They show that when $A$ is a central
simple algebra of odd degree over a field $F$, then the map $G(F)
\ra G(A)$ induced by inclusion is an isomorphism (see \cite[Cor.~2,
p.~367]{lewistignol}). Although, in some aspects, the behaviour of
$G(A)$ is similar to that of $\CK[1]$, \cite[Prop.~5]{lewistignol}
shows an aspect where they differ. It says that if $D$ is a division
ring and $n$ is a positive integer greater than $2$, then $G(M_n (D)
) \cong G(D)$, which we observed above does not hold for
$\CK[1](D)$.


\section{Higher {\itshape K}-Theory} \label{sectionhigherktheory}

For higher algebraic $K$-theory, the $K$-groups were defined by
Quillen in the early 1970s. Quillen gave two different constructions
for higher $K$-theory, called the $+$-construction and the
$Q$-construction. The $+$-construction defines the higher $K$-groups
of a ring $R$. The $Q$-construction defines the $K$-groups of an
exact category and, for a ring $R$, the $K$-groups $K_i(R)$ are
defined to be the $K$-groups $K_i(\pr (R))$ where $\pr (R)$ is the
category of finitely generated projective $R$-modules. The two
constructions do in fact give the same $K$-groups for a ring $R$,
although in appearance they are very different. (The proof is very
involved, see \cite[\S IV.7]{weibelkbook}.) For $i=0$, $1$, $2$ the
construction agrees which the definitions given in
Section~\ref{sectionlowerkgroups} (see \cite[\S 5.2.1]{rosenberg}).
The $K$-groups, although complicated to define, are functorial in
construction. We recall below some of their basic properties.


\begin{example} (See \cite[Thm.~5.3.2]{rosenberg}.)
Let $\mathbb F_q$ be a finite field with $q$ elements. Then
$K_0(\mathbb F_q) =  \mathbb Z$ and for $i \geq 1$,
\begin{align*}
K_i (\mathbb F_q) \cong
\begin{cases}\mathbb Z_{q^n - 1}  & \text{if $i = 2n - 1$,} \\
0 & \text{if $i$ is even.}
\end{cases}
\end{align*}
\end{example}

\begin{thm} \label{kifunctorthm}
If $\mathcal F: \mathcal C \ra \mathcal D$ is an exact functor
between exact categories, then $\mathcal F$ induces a map $\mathcal
F_* : K_i (\mathcal C) \ra K_i (\mathcal D)$. In particular, each
$K_i$ is a functor from the category of exact categories with exact
functors to the category of abelian groups. Moreover, isomorphic
functors induce the same map on the $K$-groups.
\end{thm}

\begin{proof}
See \cite[p.~IV.51]{weibelkbook}, \cite[p.~19]{quillen}.
\end{proof}

\begin{cor} \label{kifunctorringstoabgroups}
For $i \geq 0$, each $K_i$ is a functor from the category of rings
to the category of abelian groups. 
\end{cor}

\begin{proof}
See \cite[\S IV.1.1.2]{weibelkbook} or \cite[Eg.~5.3.22]{rosenberg}.
This follows from Theorem~\ref{kifunctorthm}, since  a ring
homomorphism $\phi :R \ra S$ induces an exact functor $S \otimes_R -
:\pr (R) \ra \pr (S)$.
\end{proof}

\begin{thm} \label{kimoritaequivalentthm}
If the rings $R$ and $S$ are Morita equivalent, then $K_i (R) \cong
K_i (S)$ for each $i \geq 0$.
\end{thm}

\begin{proof}
See \cite[\S IV.6.3.5]{weibelkbook}. This follows from
Theorem~\ref{kifunctorthm}, since if $R$ and $S$ are Morita
equivalent then there is an equivalence of categories $\pr (R) \cong
\pr (S)$. It follows that $K_i (R) \cong K_i (S)$ for each $i \geq
0$.
\end{proof}


Let $\mathcal C$ and $\mathcal D$ be exact categories. The category
of functors from $\mathcal C$ to $\mathcal D$ is an exact category
which is denoted by $[\mathcal C , \mathcal D]$ and with morphisms
defined to be natural transformations. Then by
\cite[p.~22]{quillen}, a sequence of functors from $\mathcal C$ to
$\mathcal D$
$$
0 \lra \mathcal F' \lra \mathcal F \lra \mathcal F'' \lra 0
$$
is  an exact sequence of exact functors if for all $A \in \mathcal
C$,
$$
0 \lra \mathcal F'(A) \lra \mathcal F(A) \lra \mathcal F''(A) \lra 0
$$
is an exact sequence in $\mathcal D$.

\begin{thm}\label{kirespectsdirectsums}
Let $\mathcal C$ and $\mathcal D$ be exact categories and let
$$
0 \lra \mathcal F' \lra \mathcal F \lra \mathcal F'' \lra 0
$$
be an exact sequence of exact functors from $\mathcal C$ to
$\mathcal D$. Then
$$
\mathcal F_* = \mathcal F'_* + \mathcal F''_* : K_i (\mathcal C)
\lra K_i (\mathcal D).
$$
\end{thm}

\begin{proof}
See \cite[Cor.~1, p.~22]{quillen}.
\end{proof}

In the above theorem, suppose $\mathcal C = \mathcal D$ is the
category of finitely generated projective modules (or the category
of graded finitely generated projective modules: see
Section~\ref{sectiongrmorita} for its definition). Take both
$\mathcal F'$ and $\mathcal F''$ to be the identity functor and let
$\mathcal F : \mathcal C \ra \mathcal C$; $A \mps A \oplus A$. Then
clearly
$$
0 \lra A \stackrel{\imath}{\lra} A \oplus A \stackrel{\pi}{\lra} A
\lra 0
$$
is an exact sequence, with maps $\imath : A \ra A \oplus A; x \mps
(x,0)$ and $\pi : A \oplus A \ra A; (x,y) \mps y$. So the
homomorphism $\mathcal F_* :K_i (\mathcal C) \ra K_i (\mathcal C)$
induced by $\mathcal F$ is
\begin{align*}
\mathcal F_* = \id + \id  : K_i (\mathcal C) & \lra K_i (\mathcal
C)\\
a & \lmps a + a.
\end{align*}
By induction, if $\mathcal F : \mathcal C \ra \mathcal C$; $A \mps
A^k$ for $k \in \mathbb N$, then the induced homomorphism is
$\mathcal F_* : K_i (\mathcal C) \ra K_i (\mathcal C)$; $a \mps k
a$, which is also multiplication by $k$. We will use this result in
the proofs of Propositions~\ref{ckidfunctorprop}
and~\ref{grckidfunctorprop}.

The following result was proven by Green et al.\ \cite{green}.

\begin{thm}  \label{kiofinnerautomorphism}
For a ring $R$, let $f : R \ra R$ be an inner automorphism of $R$
with $f (r) = a^{-1} r a$, where $a$ is a unit of $R$. Then $K_i (f)
: K_i (R) \ra K_i (R)$ is the identity.

\end{thm}

\begin{proof}
(See \cite[Lemma~2]{green}.) Since $f : R \ra R$ is a ring
homomorphism, there is a functor
$$
F : \pr (R) \lra \pr (R); \, P \lmps R \otimes_f P,
$$
where $R$ is a right $R$-module via $f$ and we note that $r \otimes
r' p = r f (r') \otimes p$. We will show that there is a natural
isomorphism $\phi$ from the identity functor to the functor $F$. For
$P \in \pr (R)$, define
$$
\phi_P :P \lra R \otimes_f P \, ; \;\; p \lmps 1 \otimes a p.
$$
Note that $\phi_P$ is an $R$-module homomorphism, since
\begin{align*}
\phi_P (rp) &=\, 1 \otimes arp \,=\, 1 \otimes ara^{-1}ap \\
& =\, 1 \otimes f^{-1} (r) ap \, = \, r \otimes ap \,=\, r (\phi_P
(p)).
\end{align*}
Then for an $R$-module homomorphism $g : P \ra P'$ in $\pr (R)$, the
following diagram commutes:
\begin{displaymath}
\xymatrix{ P \ar[d]_{\id (g)} \ar[rr]^{\phi_P \;\;\;\;\;\;} && R
\otimes_f P \ar[d]^{F (g)=1 \otimes g} \\
P' \ar[rr]_{\phi_{P'} \;\;\;\;\;\;}  && R \otimes_f P' }
\end{displaymath}
So $\phi$ forms a natural transformation from the identity functor
to $F$. Each $\phi_P$ is an isomorphism in $\pr (R)$ since the map
$$
\phi^{-1}_P : R \otimes_f P \lra P \, ; \;\; r \otimes p \lmps r
a^{-1} p
$$
is an $R$-module homomorphism which is an inverse of $\phi_P$. So
the inverses $\phi^{-1}_P$ form a natural transformation which is an
inverse of $\phi$. By applying Theorem~\ref{kifunctorthm}, this
shows that the induced map $K_i (R) \ra K_i (R)$ is the identity.
\end{proof}

The above theorem allowed Green et al.\ to prove the main theorem of
their paper \cite[Thm.~4]{green}, which shows that
$$
K_i(D) \otimes \mathbb Z[1/n] \cong K_i(F) \otimes \mathbb Z[1/n]
$$
for a division algebra $D$ over $F$ of dimension $n^2$. In
Chapter~\ref{chapterktheoryofazalg}, we will generalise this result
to cover Azumaya algebras which are free over their centres. Their
proof uses that fact that $K_i(R) \ra K_i(M_t(R)) \ra K_i(R)$ is
multiplication by $t$, where $R\rightarrow M_t (R)$ is the diagonal
homomorphism $r\mapsto rI_t$ (see \cite[Lemma~1]{green}), and also
the Skolem-Noether theorem which guarantees that algebra
homomorphisms in the setting of central simple algebras are inner
automorphisms. Their proof combines these results with
Theorem~\ref{kiofinnerautomorphism} and with the main result of
\cite{hal} which states that
$$\lim_{i\rightarrow \infty} M_{n^{2i}}(F) \cong \lim_{i\rightarrow
\infty}M_{n^{2(i+1)}}(D).$$



\chapter{\textit{K}-Theory of Azumaya Algebras}
\label{chapterktheoryofazalg}

As we noted in the previous chapter, Green et al.\ \cite{green}
proved that for a division algebra finite dimensional over its
centre, its $K$-theory is ``essentially the same'' as the $K$-theory
of its centre; that is, for a division algebra $D$ over its centre
$F$ of index $n$,
\begin{equation} \label{keqn}
K_i(D) \otimes \mathbb Z[1/n] \cong K_i(F) \otimes \mathbb Z[1/n].
\end{equation}
In this chapter, we prove that the isomorphism (\ref{keqn}) holds
for any Azumaya algebra free over its centre (see
Theorem~\ref{azumayafreethm}). A corollary of this is that the
isomorphism holds for Azumaya algebras over semi-local rings. This
extends the results of Hazrat (see \cite{sk12001,
hazratreducedkthofazumaya}) where (\ref{keqn}) type properties have
been proven for central simple algebras and Azumaya algebras over
local rings respectively.

In Section~\ref{sectiondfunctors}, we  introduce an abstract functor
called a $\mathcal D$-functor, which is defined on the category of
Azumaya algebras free over a fixed base ring. This is a continuation
of \cite{sk12001} and \cite{hazratreducedkthofazumaya} where Hazrat
defines similar functors, over categories of central simple algebras
and Azumaya algebras respectively. Here we show in
Theorem~\ref{torsionthm} that the range of a $\mathcal D$-functor is
the category of bounded torsion abelian groups.  We then prove that
the kernel and cokernel of the $K$-groups are $\mathcal D$-functors,
which allows us to prove \eqref{keqn} type properties for Azumaya
algebras which are free over their centres (see
Theorem~\ref{azumayafreethm} and \cite[Thm.~6]{hm}).

The Hochschild homology of  Azumaya algebras behaves in a similar
way to the $K$-theory of Azumaya algebras. Corti\~nas and Weibel
\cite{cortinasweibel} have shown that there is a similar result to
Theorem~\ref{azumayafreethm} for the Hochschild homology of an
Azumaya algebra. In Section~\ref{sectionhomology}, we begin by
recalling the definition of Hochschild homology, which can be found
in \cite{loday, weibelhomalg}. We then recall the result of
Corti\~nas and Weibel which shows that $HH^k_*(A) \cong HH^k_*(R)$
for an $R$-Azumaya algebra $A$ of constant rank.


\section{$\mathbf{\mathcal{D}}$-functors} \label{sectiondfunctors}

Throughout this section, let $R$ be a fixed commutative ring and
$\mathcal Ab$ be the category of abelian groups. Let $\text{Az}(R)$
be the category of Azumaya algebras free over $R$ with $R$-algebra
homomorphisms.

\begin{defin}
Consider a functor $\mathcal F:\text{Az}(R) \ra \mathcal Ab; \: A
\mps \mathcal F (A)$. Such a functor is called a \emph{$\mathcal
D$-functor}\index{D@$\mathcal D$-functor} if it satisfies the
following three properties:
\begin{description}
\item[(1)] $\mathcal F (R)$ is  the trivial group,

\item[(2)] For any Azumaya algebra $A$ free over  $R$ and any $k
\in \mathbb N$, there is a homomorphism
$$
\rho: \mathcal F (M_k (A)) \lra \mathcal F (A)
$$
such that the composition $\mathcal F(A) \ra \mathcal F (M_k (A))
\ra \mathcal F(A)$ is $\eta_{k}$, where $\eta_k(x)=x^k$.

\item[(3)] With $\rho$ as in property (2), then $\ker(\rho)$ is
$k$-torsion.
\end{description}

\end{defin}

The name $\mathcal D$-functor comes from Dieudonn\'{e}, since
Dieudonn\' e determinant behaves in a similar way to Property~(2).
Note that $M_k (A)$ is an Azumaya algebra free over $R$ by
Proposition~\ref{tensorproductprop}, since $A$ and $M_k (R)$ are
Azumaya algebras free over $R$ (see Example~\ref{azumayaexamples})
and $M_k (A) \cong A \otimes_R M_k (R)$. Also note that the natural
$R$-algebra homomorphism $A \ra M_k (A)$; $a \mps a I_k$ induces a
group homomorphism $\mathcal F(A) \ra \mathcal F (M_k (A))$.

The following theorem shows that the range of a $\mathcal D$-functor
is a bounded torsion abelian group.

\begin{thm} \label{torsionthm} Let $A$ be an Azumaya algebra free
over $R$ of dimension $n$ and let $\mathcal F$ be a $\mathcal
D$-functor. Then $\mathcal F(A)$ is $n^2$-torsion.
\end{thm}

\begin{proof}
We begin by applying the definition of a $\mathcal D$-functor to the
$R$-Azumaya algebra $R$. Property (1) says that $\mathcal F (R)$ is
the trivial group and property (2) says that the composition
$\mathcal F(R) \ra \mathcal F (M_n (R)) \ra \mathcal F(R)$ is
$\eta_{n}$. So by the third property, $\mathcal F (M_n(R))$ is
$n$-torsion. Since $A$ is an Azumaya algebra free over $R$, we have
$A \otimes_R A^{\op} \cong \End_R(A) \cong M_n (R)$, which means
that $\mathcal F (A \otimes_R A^{\op}) \cong \mathcal F ( M_n (R))$
is $n$-torsion.

The natural $R$-algebra homomorphisms $A \ra A \otimes_R
A^{\text{op}}$, $A^{\text{op}} \ra \End_R(A^{\text{op}})$, $\End_R
(A^{\op}) \cong M_n (R)$ and $A \otimes_R M_n (R) \cong M_n (A)$
combine to give the following $R$-algebra homomorphisms
\begin{align*}
A \lra A \otimes_R A^{\op} \lra A \otimes \End_R (A^{\op}) \lra A
\otimes M_n (R) \lra M_n(A).
\end{align*}
These induce the homomorphisms $\mathcal F(A) \stackrel{i}{\ra}
\mathcal F(A \otimes_R A^{\op}) \stackrel{r}{\ra} \mathcal
F(M_n(A))$. By property~(2) in the definition of a $\mathcal
D$-functor, we have a homomorphism $\rho :\mathcal F(M_n (A))
\rightarrow \mathcal F(A)$ such that the composition $\mathcal F(A)
\ra \mathcal F (M_n (A)) \ra \mathcal F(A)$ is $\eta_n$. Consider
the following diagram
\begin{equation}
\begin{split}
\xymatrix{
\mathcal F(A) \ar[d]_i \ar@/^/[ddrr]^{\eta_n}&&  \\
\mathcal F(A \otimes_R A^{\text{op}}) \; \ar[d]_r &&  \\
\mathcal F(M_n(A))\ar[rr]^{\rho} && \mathcal F(A)}
\end{split}
\end{equation}
which is commutative by property (2). Let $a \in \mathcal F (A)$.
Then
\begin{align*}
a^{n^2} = \big(\eta_n (a)\big)^n = \rho \circ r \big( (i (a))^n
\big) = \rho \circ r (1) = 1
\end{align*}
since $\mathcal F(A \otimes A^{\op})$ is $n$-torsion. Thus $\mathcal
F(A)$ is $n^2$-torsion.
\end{proof}

For a ring $A$ with centre $R$, consider the inclusion $R \ra A$. By
Corollary~\ref{kifunctorringstoabgroups}, this induces the map
$K_i(R) \rightarrow K_i(A)$ for $i \geq 0$. Consider the exact
sequence
\begin{equation}\label{exse}
1 \lra \ZK[i](A) \lra K_i(R) \lra K_i(A) \lra \CK[i](A) \lra 1
\end{equation}
where $\ZK[i]$ and $\CK[i]$ are the kernel and cokernel of the map
$K_i(R) \rightarrow K_i(A)$ respectively.

We will observe that $\CK[i]$ can be considered as the following
functor \label{ckifunctorazalgtoabgroups}
\begin{align*}
\CK[i]: \text{Az}(R) &\lra \mathcal Ab \\
A &\lmps \CK[i](A).
\end{align*}
For an Azumaya algebra $A$ free over $R$, clearly $\CK[i] (A) =
\coker \big(K_i (R) \ra K_i (A)\big)$ is an abelian group. For a
homomorphism $f :  A \ra A'$ of $R$-Azumaya algebras, there is an
induced group homomorphism $f_* : K_i (A) \ra K_i (A')$. The map $f$
restricted to $R$ is the identity map, so it induces the identity on
the level of the $K$-groups. Since $K_i$ preserves compositions of
maps, the following diagram is commutative
\begin{equation*}
\begin{split}
\xymatrix{ K_i(R)  \ar[r] \ar[d]_{\id} & K_i(A)  \ar[r] \ar[d]^{f_*}
& \CK[i](A) \ar[d]^{\CK[i](f)} \\
K_i(R)  \ar[r] & K_i(A')  \ar[r]  & \CK[i](A') }
\end{split}
\end{equation*}
Then it can be easily checked that $\CK[i]$ forms the required
functor. Similarly, we can consider $\ZK[i]$ as the functor
\begin{align*}
\ZK[i]: \text{Az}(R) &\lra \mathcal Ab \\
A &\lmps \ZK[i](A).
\end{align*}


\begin{prop} \label{ckidfunctorprop}
With $\CK[i]$ defined as above, $\CK[i]$ is a $\mathcal D$-functor.
\end{prop}


\begin{proof}
Property (1) in the definition of a $\mathcal D$-functor is clear,
since $K_i(Z(R)) \ra K_i(R)$ is the identity map, and so clearly
$\CK[i](R)$ is trivial.

To prove the second property, let $A$ be an Azumaya algebra free
over $R$ and let $k \in \mathbb N$.
Consider the functors
\begin{align*}
\pr (A) & \stackrel{\phi}{\lra} \pr (M_k (A)) &
\textrm{\!\!\!\!\!\!\!\! and \;\;\;\;\;\;\;\;}  \pr
(M_k (A)) &\stackrel{\psi}{\lra} \pr (A) \\
X  & \longmapsto M_k (A) \otimes_A X &  Y  & \longmapsto A^k
\otimes_{M_k (A)} Y  \notag
\end{align*}
By Theorem~\ref{kifunctorthm}, these functors induce homomorphisms
$\phi$ and $\psi$ on the level of the $K$-groups. Morita theory
shows that $\psi$ is an equivalence of categories (see
page~\pageref{moritaequivalenceofmnrandr}), so using
Theorem~\ref{kimoritaequivalentthm}, it induces an isomorphism from
$K_i(M_k (A))$ to $K_i(A)$. For each $X \in \pr (A)$, we have $\psi
\circ \phi(X) \cong X^k$. Since $K_i$ are functors which respect
direct sums (see the remarks after
Theorem~\ref{kirespectsdirectsums}), this composition induces a
multiplication by $k$ on the level of the $K$-groups.
We have the following commutative diagram
\begin{equation}
\begin{split}
\xymatrix{ K_i(R)  \ar[r]^{f \;} \ar[d]^{\id} & K_i(A)  \ar[r]^{\ov
f} \ar[d]^\phi & \CK[i](A) \ar[r] \ar[d]&1  \\
K_i(R)  \ar[r]^{g \;\;} \ar[d]^{\eta_k} & K_i(M_k(A))  \ar[r]^{\ov
g} \ar[d]^\psi_\cong & \CK[i](M_k(A)) \ar[r] \ar[d]^{\rho}& 1 \\
K_i(R)  \ar[r]_{f \;}  & K_i(A)  \ar[r]_{\ov f} & \CK[i](A) \ar[r]&
1 }
\end{split}
\end{equation}
where compositions of columns are $\eta_k$, and thus property~(2)
holds.

Now let $x \in \CK[i](M_k(A))$ such that $\rho (x)=1$. Then there is
$a \in K_i (M_k (A))$ with $\ov g (a) = x$ so that $1 = \ov f \circ
\psi (a)$. As the rows are exact, there is $b \in K_i (R)$ with
$f(b) = \psi (a)$. Taking powers of $k$, we have $f (b^k) = \psi
(a^k)$ and $\ov g (a^k) = x^k$. As $\psi$ is an isomorphism and $f
\circ \eta_k (b) = \psi \circ g (b)$, it follows that $a^k = g(b)$.
Then by the exactness of the rows, $\ov g \circ g (b) = 1$; that is,
$x^k = 1$. So $\CK[i]$ satisfies property~(3) of a $\mathcal
D$-functor.
\end{proof}

\begin{prop} \label{zkidfunctorprop}
With $\ZK[i]$ defined as above, $\ZK[i]$ is a $\mathcal D$-functor.
\end{prop}

\begin{proof}
This follows in exactly the same way as Proposition
\ref{ckidfunctorprop}.
\end{proof}

We are now in  a position to prove that the $K$-theory of Azumaya
algebras free over their centres and the $K$-theory of Azumaya
algebras over semi-local rings are isomorphic to the $K$-theory of
their centres up to (their ranks) torsions.

\begin{thm}\label{azumayafreethm}
Let $A$ be an Azumaya algebra free over its centre $R$ of dimension
$n$. Then for any $i \geq 0$, $$K_i(A) \otimes \mathbb Z[1/n] \cong
K_i(R) \otimes \mathbb Z[1/n].$$
\end{thm}

\begin{proof}
Propositions \ref{ckidfunctorprop} and \ref{zkidfunctorprop} show
that $\CK[i]$ and $\ZK[i]$ are both $\mathcal D$-functors. By
Theorem~\ref{torsionthm}, it follows that $\CK[i](A)$ and
$\ZK[i](A)$ are $n^2$-torsion abelian groups. Tensoring the exact
sequence \eqref{exse} by $\mathbb Z[1/n]$, since $\CK[i](A)\otimes
\mathbb Z[1/n]$ and $\ZK[i](A)\otimes \mathbb Z[1/n]$ vanish, the
result follows.
\end{proof}

We recall the definition of a projective module of constant rank,
which is used in the corollary below.

\begin{defin}
Let $R$ be a commutative ring and let $P$ be a prime ideal of $R$.
Set $S= R \mi P$ and write $R_P$ for the localisation $S^{-1} R$,
which is a local ring by \cite[Ex.~6.45(i)]{magurn}. (Recall a ring
$R$ is \emph{local}\index{local ring}\index{ring!local} if the
non-invertible elements of $R$ constitute a proper 2-sided ideal of
$R$.) For an $R$-module $M$, we write $M_P = S^{-1} M$. If $M$ is a
finitely generated projective $R$-module, then, by
\cite[Prop.~6.44]{magurn}, $M_P$ is a finitely generated projective
module over $R_P$. Every finitely generated projective module over a
local ring is free with a uniquely defined rank (see
\cite[Thm.~1.3.11]{rosenberg}). Then $\rank_P (M)$ is defined to be
the rank of $M_P$ as an $R_P$-module. We say that $M$ is of
\emph{constant rank}\index{constant rank} if $\rank_P (M)= n$ is the
same for all prime ideals $P$ and we write $\rank (M) =n$. If $M$ is
a free module over the ring $R$, then $M$ is of constant rank since
$M_P \cong R_P \otimes_R M$ (see \cite[Prop.~6.55]{magurn}) and
$\dim_R (M) = \dim_{R_P} (R_P \otimes_R M) = \rank (M)$.
\end{defin}

\begin{cor}\label{semilocal}
Let $R$ be a semi-local ring and let $A$ be an Azumaya algebra over
its centre $R$ of rank $n$. Then for any $i \geq 0$, $$K_i(A)
\otimes \mathbb Z[1/n] \cong K_i(R) \otimes \mathbb Z[1/n].$$
\end{cor}

\begin{proof}
Since $A$ is finitely generated projective of constant rank and $R$
is a semi-local ring, it follows that $A$ is a free module over $R$
(see \cite{bourbaki}, \S II.5.3, Prop.~5), and thus the corollary
follows from Theorem~\ref{azumayafreethm}.
\end{proof}

Theorem~\ref{azumayafreethm} covers Azumaya algebras which are free
over their centres. We also mention here a theorem of Hazrat,
Hoobler \cite[Thm.~12]{hazrathoobler}, which shows that a similar
result holds for the $K$-theory of Azumaya algebras over sheaves.
Their result covers the case of Azumaya algebras over Noetherian
centres, but it remains as a question whether this result holds for
any Azumaya algebra of constant rank.

\begin{ques}\label{ques}
Let $A$ be an Azumaya algebra over its centre $R$ of constant rank
$n$. Then is it true that for any $i \geq 0$, $K_i(A) \otimes
\mathbb Z[1/n] \cong K_i(R) \otimes \mathbb Z[1/n]?$
\end{ques}


\section{Homology of Azumaya algebras} \label{sectionhomology}

We begin this section by recalling the definition of Hochschild
homology, which can be found in \cite{loday, weibelhomalg}.


Let $R$ be a ring. Recall that a chain complex of $R$-modules $(C_*,
d_*)$ is a family of right $R$-modules $\{C_n \}_{ n \in \mathbb Z}$
together with $R$-module homomorphisms $d_n : C_n \ra C_{n-1}$, such
that the composition of any two consecutive maps is zero: $d_n \circ
d_{n+1} = 0$ for all $n$. The maps $d_n$ are called the
differentials of $C_{*}$. The chain complex is usually written as:
$$
\cdots \lra C_{n+1} \stackrel{d_{n+1}}{\lra} C_n
\stackrel{d_n}{\lra} C_{n-1} \stackrel{d_{n-1}}{\lra} C_{n-2} \lra
\cdots
$$
The kernel of $d_n$ is denoted by $Z_n$, and the image of $d_{n+1}$
is denoted by $B_n$. Then for all $n$, $ 0 \subseteq B_n \subseteq
Z_n \subseteq C_n.$ The $n^{\mathrm{th}}$ homology module of $C_{*}$
is the quotient
$$
H_n (C_{*}) = Z_n /B_n.
$$

Let $K$ be a commutative ring, $R$ be a $K$-algebra, and $M$ an
$R$-bimodule. We will write $R^{\otimes n}$ for the $n$-fold tensor
product of $R$ over $K$. The chain complex which gives rise to
Hochschild homology is given by $C_{*} = M \otimes R^{\otimes *}$
and is written as
$$
\cdots \lra M \otimes R^{\otimes n} \stackrel{\; d_n}{\lra} M
\otimes R^{\otimes n-1} \stackrel{\; d_{n-1}}{\lra} \cdots
\stackrel{\; d_2}{\lra} M \otimes R \stackrel{\; d_1}{\lra} M \lra
0.
$$
Its differentials are $d_n=\sum_{i=0}^n (-1)^i \partial_i$ where
$\partial_i$ are defined by
$$
\partial_i(m\otimes a_1 \otimes \cdots \otimes a_n) = \left\{
\begin{array}{l l}
ma_1 \otimes a_2 \otimes \cdots \otimes a_n  & \quad \text{if $i=0$}\\
m\otimes a_1 \otimes \cdots \otimes a_i a_{i+1} \otimes \cdots
\otimes a_n & \quad \text{if $1 < i < n$}\\
a_n m\otimes a_1 \otimes \cdots \otimes a_{n-1} & \quad \text{if
$i=n$}
\end{array} \right.
$$
for $a_i \in R$ and $m \in M$. Then by \cite[Lemma~1.1.2]{loday}
$d_n \circ d_{n+1} = 0$, so $(C_*, d_*)$ is a chain complex called
the Hochschild complex. The \emph{$n^{\mathrm{th}}$ Hochschild
homology module}\index{Hochschild homology module} of $R$ with
coefficients in $M$ is defined to be $H_n (M \otimes R^{\otimes *})$
and is denoted by $H_n (R, M)$. The direct sum $\bigoplus_{ n \geq
0} H_n (R,M)$ is denoted by $H_* (R,M)$, and we will write $HH_{*}
(R)$ for $H_* (R, R)$. By \cite[\S1.1.5]{loday}, if $R$ is
commutative then $HH_* (R)$ is an $R$-module.




For a commutative ring $R$ and an $R$-algebra $A$, suppose that $K$
is a commutative subring of $R$. Then we consider the Hochschild
homology, denoted by $HH_*^K$, of $R$ and $A$ as $K$-algebras.





The following result of Corti\~nas and Weibel \cite{cortinasweibel}
shows that for an $R$-Azumaya algebra $A$ of constant rank, its
Hochschild homology behaves in a similar way to its $K$-theory (see
Theorem~\ref{azumayafreethm}).

\begin{thm}
Let $A$ be an Azumaya algebra over a $K$-algebra $R$. If $A$ has
constant rank, then there is an isomorphism $HH^K_{*} (A) \cong
HH^K_{*} (R)$.
\end{thm}


\begin{proof}
See \cite[p.~53]{cortinasweibel}.
\end{proof}





\chapter{Graded Azumaya Algebras} \label{chgraded}

This chapter contains some basic definitions in the graded setting.
These definitions can be found in \cite{dade, hwcor, grrings},
though not always in the generality that we require. We also include
in this chapter a number of results in the graded setting which will
be used in Chapter~\ref{chgradedktheoryofazumayaalgebras}. In
Section~\ref{sectiongradedrings} we define graded rings and give the
graded versions of the definitions of ideals, factor rings and ring
homomorphisms. In Section~\ref{sectiongradedmodules} we give the
corresponding definitions for graded modules. We then define graded
division rings and show that a graded module over a graded division
ring is graded free with a uniquely defined dimension.

In Section~\ref{sectiongradedcentralsimplealgebras}, we study graded
central simple algebras graded by an arbitrary abelian group. We
observe that the tensor product of two graded central simple
$R$-algebras is graded central simple
(Propositions~\ref{tensorgradedsimple} and
\ref{tensorgradedcentral}). This result has been proven by Wall  for
$\mathbb Z / 2 \mathbb Z$ -graded central simple algebras (see
\cite[Thm.~2]{wall}), and by Hwang and Wadsworth for $R$-algebras
with a totally ordered abelian, and hence torsion-free, grade group
(see \cite[Prop.~1.1]{hwcor}). We then observe that a graded central
simple algebra, graded by an abelian group, is an Azumaya algebra
(Theorem~\ref{gcsaazumayaalgebra}). This result extends the result
of Boulagouaz \cite[Prop.~5.1]{boulag} and Hwang, Wadsworth
\cite[Cor.~1.2]{hwcor} (for a totally ordered abelian grade group)
to graded rings in which the grade group is not totally ordered.

We define grading on matrices in
Section~\ref{sectiongradedmatrixrings}, and  observe some properties
of these graded matrix rings. We generalise a result of Caenepeel et
al.\ \cite[Thm.~2.1]{can2} in Theorem~\ref{grmatrixisothm}. We have
also  rewritten a number of known results on simple rings in the
graded setting, which were required for the proof of this theorem.
In Section~\ref{sectiongrmorita}, we define graded projective
modules and graded Azumaya algebras. We prove that for a graded ring
$R$, the graded matrix ring over $R$ is Morita equivalent to $R$
(see Proposition~\ref{grmorita}).


\section{Graded rings}
\label{sectiongradedrings}


A ring $R = \bigoplus_{ \ga \in \Ga} R_{\ga}$ is called a
\emph{$\Ga$-graded ring}, or simply a \emph{graded ring},
\index{graded ring} if $\Ga$ is a group, each $R_{\ga}$ is an
additive subgroup of $R$ and $R_{\ga} \cdot R_{\delta} \subseteq
R_{\ga \delta}$ for all $\ga, \delta \in \Ga$. We remark that
initially $\Ga$ is an arbitrary group which is not necessarily
abelian, so we will write $\Ga$ as a multiplicative group with
identity element $e$.
Each $x \in R$ can be uniquely expressed as a finite sum $x=
\sum_{\ga \in \Ga} x_{\ga}$ with each $x_{\ga} \in R_{\ga}$. For
each $\ga \in \Ga$, the elements of $R_\ga$ are said to be
\emph{homogeneous of degree $\ga$}\index{homogeneous element} and we
write $\deg(r) = \ga$ if $r \in R_{\ga}$. We let $R^{h} =
\bigcup_{\ga \in \Ga} R_{\ga}$ be the set of homogeneous elements of
$R$. The set
$$
\Ga_{R} = \big \{ \ga \in \Ga : R_{\ga} \neq \{0 \} \big \},
$$
which is also denoted by $\Supp (R)$, is called the
\emph{support}\index{support} (or grade set) of $R$. We note that
the support of $R$ is not necessarily a group.

\begin{examples}\label{grringexamples}
\begin{enumerate}
\item Let $(\Ga, \cdot )$ be a group and $R$ be a ring. Set $R =
\bigoplus_{\ga \in \Ga} R_{\ga}$ where $R_e = R$ and for all $\ga
\in \Ga$ with $\ga \neq e$, let $R_\ga = 0$. Then $R$ can be
considered as a trivially $\Ga$-graded ring, with $\Supp (R)=
\{e\}$.

\item Let $A$ be a ring. Then the polynomial ring $R=A[x]$ is a
$\mathbb Z$-graded ring, with $R = \bigoplus_{n \in \mathbb Z} R_n$
where $R_n =A x^n$ for $n \geq 0$ and $R_n = 0$ for $n <0$. Here
$\Supp (R) = \mathbb N \cup \{0\}$.

\item
Let $(G, \cdot )$ be a group and let $A$ be a ring. Then the group
ring $R = A[G]$ is graded ring, with $R = \bigoplus_{g \in G} R_{g}$
where $R_g =  \{ a g : a \in A \}$, and $\Supp (R) = G$.

\end{enumerate}
\end{examples}

\begin{prop} \label{basicsofgradedrings}
Let $R = \bigoplus_{ \ga \in \Ga} R_{\ga}$ be a $\Ga$-graded ring.
Then:
\begin{enumerate}
\item $1_R$ is homogeneous of degree $e$,

\item $R_e$ is a subring of $R$,

\item Each $R_{\de}$ is an $R_e$-bimodule,

\item For an invertible element $r \in R_{\de}$, its inverse
$r^{-1}$ is homogeneous of degree $\de^{-1}$.
\end{enumerate}
\end{prop}

\begin{proof}
(1) Suppose $1_R = \sum_{\ga \in \Ga} r_{\ga}$ for $r_{\ga} \in
R_{\ga}$. For some $\de \in \Ga$, let $s \in R_{\de}$ be an
arbitrary non-zero element. Then $s = s 1_R = \sum_{\ga \in \Ga} s
r_{\ga}$ where $s r_{\ga} \in R_{\de \ga}$ for all $\ga \in \Ga$.
The decomposition is unique, so $s r_{\ga} = 0$ for all $\ga \in
\Ga$ with $\ga \neq e$. But as $s$ was arbitrary, this holds for all
$s \in R$ (not necessarily homogeneous), and in particular $1_R
r_\ga = r_\ga =0$ if $\ga \neq e$. For $\ga =e$, we have $1_R =
r_e$, so $1_R \in R_e$.

\vspace{3pt}

(2) This follows since $R_e$ is an additive subgroup of $R$ with
$R_e R_e \subseteq R_e$ and $1 \in R_e$.

\vspace{3pt}

(3) This is immediate.

\vspace{3pt}

(4) Let $x = \sum_{\ga} x_{\ga}$ (with $\deg(x_{\ga})= \ga$) be the
inverse of $r$, so that $1=rx = \sum_{\ga} r x_{\ga}$ where $r
x_{\ga} \in R_{\de \ga}$. Since $1$ is homogeneous of degree $e$ and
the decomposition is unique, it follows that $r x_{\ga} = 0$ for all
$\ga \neq \de^{-1}$. Since $r$ is invertible, 
we have $x_{\de^{-1}} \neq 0$, so $x=x_{\de^{-1}}$ as required.
\end{proof}

We say that a $\Ga$-graded ring $R=\bigoplus_{ \ga \in \Ga} R_{\ga}$
is a \emph{strongly graded ring}\index{strongly graded ring} if
$R_{\ga} R_{\de} = R_{\ga \de}$ for all $\ga, \de \in \Ga$. A graded
ring $R$ is called a \emph{crossed product}\index{crossed product}
if there is an invertible element in every homogeneous component
$R_\ga$ of $R$; that is, $R^* \cap R_\ga \neq \emptyset$ for all
$\ga \in \Ga$. For a $\Ga$-graded ring $R$, let $R^{h*}$ be the set
of invertible homogeneous elements of $R$. Then $R^{h*}$ is a
subgroup of $R^*$, and the degree map $\deg : R^{h*} \ra \Ga$ is a
group homomorphism. Let
$$
\Ga_{R}^* = \big \{ \ga \in \Ga : R^* \cap R_{\ga} \neq \emptyset
\big \},
$$
be the support of the invertible homogeneous elements of $R$. We
note that $R$ is a crossed product if and only if $\Ga_R^* = \Ga$,
which is equivalent to the degree map being surjective.

\begin{prop} \label{crossedproductstronglygradedprop}
 Let $R = \bigoplus_{ \ga \in \Ga} R_{\ga}$ be a $\Ga$-graded
ring. Then:
\begin{enumerate}
\item $R$ is strongly graded if and only if $1 \in R_{\ga} R_{\ga^{-1}}$
for all $\ga \in \Ga$,

\item $R$ is a crossed product if and only if the degree map is
surjective,

\item If $R$ is a crossed product, then $R$ is a strongly graded ring.
\end{enumerate}

\end{prop}

\begin{proof}
(1) The forward direction is immediate. Suppose $1 \in  R_{\ga}
R_{\ga^{-1}}$ for all $\ga \in \Ga$. Then for $\si, \de \in \Ga$,
$$
R_{\si \de} = R_e R_{\si \de} = (R_\si R_{\si^{-1}}) R_{\si \de}  =
R_\si (R_{\si^{-1}} R_{\si \de}) \subseteq R_{\si} R_{\de}
$$
proving $R_{\si \de} = R_\si R_\de$, so $R$ is strongly graded.

\vspace{3pt}

(2) This is immediate.

\vspace{3pt}

(3) For $\de \in \Ga$, there exists $r \in R^* \cap R_\de$. So
$r^{-1} \in R_{\de^{-1}}$ and $1 = r r^{-1} \in R_\de R_{\de^{-1}}$.
\end{proof}

A two-sided ideal $I$ of $R$ is called a \emph{homogeneous
ideal}\index{homogeneous ideal} (or \emph{graded ideal}\index{graded
ideal}) if
$$
I= \bigoplus_{ \ga \in \Gamma} (I \cap R_{\ga}).
$$
There are similar notions of graded subring, graded left ideal and
graded right ideal. When $I$ is a homogeneous ideal of $R$, the
factor ring $R/I$ forms a graded ring, with
\[
R/I = \bigoplus_{\ga \in \Ga} (R/I)_{\ga} \mathrm{ \;\;\; where
\;\;\; } (R/I)_{\ga} = (R_{\ga} +I)/I.
\]
A graded ring $R$ is said to be {\em graded simple}\index{graded
ring!graded simple} if the only homogeneous two-sided ideals of $R$
are $\{ 0 \}$ and $R$.

\begin{prop} \label{homogeneousidealhomogeneouselementsprop}
An ideal $I$ of a graded ring $R$ is a homogeneous ideal if and only
if $I$ is generated as a two-sided ideal of $R$ by homogeneous
elements.

\end{prop}

\begin{proof}
Suppose $I = \bigoplus_{\ga \in \Ga} \left( I \cap R_{\ga} \right)$,
and consider $I \cap R^h$. This is a subset of $I$ consisting of
homogeneous elements. Any $x \in I$ can be written as $x = \sum_{\ga
\in \Ga} x_{\ga}$ where $x_\ga \in I \cap R_{\ga}$. So $x_{\ga} \in
I \cap R^h$, and thus $x$ can be written as a sum of elements of $I
\cap R^h$.

Conversely, suppose $I$ is generated by the set $\{x_j \}_{j \in
\mathcal J} \subseteq I$ consisting of homogeneous elements, for
some indexing set $\mathcal J$. Then for any $x \in I$, we can write
$x$ as $x = \sum r_k x_j s_l$ where $r_k , s_l \in R^h$. Then for
some $r_k, x_j, s_l$ with $\al = \deg(r_k x_j s_l)$, we have $r_k
x_j s_l \in I \cap R_\al$, since $x_j \in I$. This shows that $I =
\bigoplus_{\ga \in \Ga} \left( I \cap R_{\ga} \right)$, as required.
\end{proof}

Let $R$ and $S$ be $\Ga$-graded rings. Then a \emph{graded ring
homomorphism}\index{graded ring homomorphism} $f:R \ra S$ is a ring
homomorphism such that $f(R_{\ga}) \subseteq S_{\gamma}$ for all
$\ga \in \Ga$. Further, $f$ is called a \emph{graded isomorphism} if
$f$ is bijective and, when such a graded isomorphism exists, we
write $R \conggr S$. We remark that if $f$ is a graded ring
homomorphism which is bijective, then its inverse $f^{-1}$ is also a
graded ring homomorphism.

For a graded ring homomorphism $f:R \ra S$, we know from the
non-graded setting \cite[p.~122]{hungerford} that $\ker (f)$ is an
ideal of $R$ and $\im (f)$ is a subring of $S$. It can easily be
shown that $\ker (f)$ is a graded ideal of $R$ and $\im (f)$ is a
graded subring of $S$. Note that if $\Ga$ is an abelian group, then
the centre $Z(R)$ of a graded ring $R$ is a graded subring of $R$.
If $\Ga$ is not abelian, then the centre of $R$ may not be a graded
subring of $R$, as is shown by the following example.

\begin{example}
Let $G = S_3 = \{e,a,b,c,d,f \}$ be the symmetric group of order
$3$, where
$$
a = (23),\;\; b=(13) ,\;\; c=(12) ,\;\; d=(123) ,\;\; f=(132).
$$
Let $A$ be a ring, and consider the group ring $R = A[G]$, which is
a $G$-graded ring by Example~\ref{grringexamples}(3). Let $x = 1d+
1f \in R$, where $1= 1_A$, and we note that $x$ is not homogeneous
in $R$. Then $x \in Z( R)$, but the homogeneous components of $x$
are not in the centre of $R$. As $x$ is expressed uniquely as the
sum of homogeneous components, we have $x \notin \bigoplus_{g \in G}
(Z(R) \cap R_g)$.
\end{example}

This example can be generalised by taking a non-abelian finite group
$G$ with a subgroup $N$ which is normal and non-central. Let $A$ be
a ring and consider the group ring $R= A[G]$ as above. Then $x =
\sum_{n \in N} 1 n$ is in the centre of $R$, but the homogeneous
components of $x$ are not all in the centre of $R$.

\section{Graded modules}
\label{sectiongradedmodules}

Let $\Ga$ be a multiplicative group and let  $R = \bigoplus_{ \ga
\in \Ga} R_{\ga}$ be a $\Ga$-graded ring.  We say that the group
$(\Ga , \cdot)$ {\em acts freely} (as a left action) on a set $\Ga'$
if for all $\ga$, $\ga' \in \Ga$, $\de \in \Ga'$, we have $\ga \de =
\ga' \de$ implies $\ga = \ga'$, where $\ga \de$ denotes the image of
$\de$ under the action of $\ga$. A \emph{$\Ga'$-graded left
$R$-mod\-ule}\index{graded module} $M$ is defined to be a left
$R$-mod\-ule with a direct sum decomposition $M=\bigoplus_{\ga \in
\Ga'} M_{\ga}$, where each $M_{\ga}$ is an additive subgroup of $M$
and $\Ga$ acts freely on the set $\Ga'$, such that $R_{\ga} \cdot
M_{\la} \subseteq M_{\ga \la}$ for all $\ga \in \Ga, \la \in \Ga'$.
{\it From now on, unless otherwise stated, a graded module will mean
a graded left module.}

We remark that here we define a graded $R$-module $M$ which is
graded by a set $\Ga'$, where the grade group of $R$ acts freely on
$\Ga'$. It is also possible to fix a group $\Ga$ and, for all graded
rings and modules, to work instead with this fixed group $\Ga$ as
the grade group. This is the approach taken throughout the book
\cite{grrings}. While our approach is more general than taking a
fixed group $\Ga$, in most cases their weaker approach suffices.

As for graded rings, $M^h = \bigcup_{\ga \in \Ga'} M_{\ga}$ and
$\Ga'_M = \big\{ \ga \in \Ga' : M_{\ga} \neq \{ 0\}  \big\}$. A
\emph{graded submodule}\index{graded submodule} $N$ of $M$ is
defined to be an $R$-submodule such that $N= \bigoplus_{\ga \in
\Ga'} (N \cap M_{\ga})$. For a graded submodule $N$ of $M$, we
define the graded quotient structure on $M/N$ by
\[
M/N = \bigoplus_{\ga \in \Ga'} (M/N)_{\ga} \mathrm{ \;\;\; where
\;\;\; } (M/N)_{\ga} = (M_{\ga} +N)/N.
\]
A graded $R$-module $M$ is said to be {\em graded
simple}\index{graded module!graded simple} if the only graded
submodules of $M$ are $\{ 0 \}$ and $M$. A \emph{graded free
$R$-module}\index{graded module!graded free} $M$ is defined to be a
graded $R$-module which is free as an $R$-module with a homogeneous
base.

Let $N= \bigoplus_{\ga \in \Ga''} N_{\ga}$ be another graded
$R$-module, such that there is a set $\Delta$ containing $\Ga'$ and
$\Ga''$ as subsets, where $\Ga$ acts freely on $\Delta$. The graded
$R$-module $M$ can be written as $M = \bigoplus_{ \ga \in \Delta}
M_{\ga}$ with $M_\ga = 0$ if $\ga \in \Delta \mi \Ga'_M$, and
similarly for $N$. A \emph{graded $R$-module
homomorphism}\index{graded module homomorphism} $f:M \ra N$ is an
$R$-module homomorphism such that $f(M_{\de}) \subseteq N_{\de}$ for
all $\de \in \Delta$. Let $\Hom_{R \hygrmod}(M,N)$ denote the group
of graded $R$-module homomorphisms, which is an additive subgroup of
$\Hom_R (M,N)$. If the graded $R$-module homomorphism $f$ is
bijective, then $f$ is called a \emph{graded isomorphism} and we
write $M \conggr N$. If $f$ is a bijective graded $R$-module
homomorphism, then its inverse $f^{-1}$ is also a graded $R$-module
homomorphism.

Suppose $\Delta$  is a group containing $\Ga'$ and $\Ga''$ as
subgroups, where the group $\Ga$ acts freely on $\Delta$, and
suppose $R$, $M$ and $N$ are defined as above. A graded $R$-module
homomorphism from $M$ to $N$ may also shift the grading on $N$. For
each $\de \in \Delta$, we have a subgroup of $\mathrm{Hom}_R(M,N)$
of $\de$-shifted homomorphisms
$$
 {\mathrm{Hom}_R(M,N)}_{\de} = \{ f \in \mathrm{Hom}_R(M,N) :
f(M_{\ga}) \subseteq N_{\ga \de} \textrm{ for all } \ga \in \Delta
\}.
$$
For some $\de \in \Delta$, we define the $\de$-shifted $R$-module
$M(\de)$\label{deshiftedmodule} as $M(\de) =\bigoplus_{\ga \in
\Delta} M(\de)_\ga$ where $M(\de)_\ga = M_{\ga \de}$. Then
$$
\Hom_R(M,N)_\de = \Hom_{R \hygrmod}(M,N(\de)) = \Hom_{R \hygrmod}
(M( \de^{-1}),N).
$$
Let $\mathrm{HOM}_R(M,N) = \bigoplus_{\de \in \Delta}
\mathrm{Hom}_R(M,N)_{\de}.$\label{hompageref}

Let $M,N$ be $\Delta$-graded $R$-modules as above. Then $M \oplus N$
forms a $\Delta$-graded $R$-module with
$$
M \oplus N = \bigoplus_{\ga \in \Delta} (M \oplus N)_{\ga} \text{
\;\; where } (M \oplus N)_{\ga} = M_\ga \oplus N_\ga.
$$
With $R$ defined as above and for $(d) = (\de_1 , \ldots , \de_n)
\in \Ga^n$, consider\label{rngradedfree}
$$
R^n(d) = R(\de_1) \oplus \cdots \oplus R(\de_n) = \bigoplus_{\ga \in
\Ga} (R(\de_1)_{\ga} \oplus \cdots \oplus R(\de_n)_{\ga})
$$
where $R(\de_i)_{\ga} = R_{\ga \de_i}$ is the $\ga$-component of the
$\de_i$-shifted graded $R$-module $R(\de_i)$. Note that for each
$i$, with $1 \leq i \leq n$, the element $e_i$ of the standard basis
for $R^n(d)$ is homogeneous of degree $\de_i^{-1}$. Similarly for
$\{\de_i\}_{i \in I}$ where $I$ is an indexing set and $\de_i \in
\Ga$, consider the $\Ga$-graded $R$-module $\bigoplus_{i \in I} R
(\de_i)$. Again the $i$-th basis element $e_i$ of $\bigoplus_{i \in
I} R (\de_i)$ is homogeneous of degree $\de_i^{-1}$.

Suppose $F$ is a $\Ga$-graded $R$-module which is graded free with a
homogeneous base $\{b_i\}_{i \in I}$, where $\deg (b_i) = \delta_i$.
Then the map
\begin{align*}
\varphi : \bigoplus\nolimits_{i \in I} R(\de_i^{-1}) & \lra F \\
e_i & \lmps b_i
\end{align*}
is a graded $R$-module isomorphism, and we write $F \conggr
\bigoplus_{i \in I} R(\de_i^{-1})$. Conversely, suppose $F$ is a
$\Ga$-graded $R$-module with $\bigoplus_{i \in I} R(\de_i^{-1})
\conggr F$, as graded $R$-modules, for some $\{\de_i\}_{i \in I}$
with $\de_i \in \Ga$. Since $\{e_i \}_{i \in I}$ forms a homogeneous
basis of $\bigoplus_{i \in I} R(\de_i^{-1})$, the images of the
$e_i$ under the graded isomorphism form a homogeneous basis for $F$.
Thus $F$ is graded free if and only if $F \conggr \bigoplus_{i \in
I} R(\de_i^{-1})$ for some $\{\de_i\}_{i \in I}$ with $\de_i \in
\Ga$. In the same way, $F$ is graded free with a finite basis if and
only if $F \conggr R^n (d)$ for some $(d) = (\de_1 , \ldots , \de_n)
\in \Ga^n$.

\begin{prop} \label{rdeconggrr}
Let $R$ be a $\Ga$-graded ring and let $\de \in \Ga$. Then $R(\de)
\conggr R$ as graded $R$-modules if and only if $\de \in \Ga_R^*$.
\end{prop}

\begin{proof}
If $\de \in \Ga_R^*$, then let $x \in  R^* \cap R_{\de}$. Then there
is a graded $R$-module isomorphism $\mar_x :R \ra R(\de);$ $r \mps
rx$. Conversely, if $\phi : R \conggr R(\de)$, then $\phi (1) \in
R_\de$ with inverse $\phi^{-1} (1)$, so $\phi (1) \in R^* \cap
R_{\de}$.
\end{proof}

We note that it follows from the above proposition that $R(\de)
\conggr R$ for all $\de \in \Ga$ if and only if $R$ is a crossed
product. For a graded ring $R$ and $\de , \al \in \Ga$, it is clear
that $R(\al) \conggr R(\de)$ as graded $R$-modules if and only if
$R(\al)(\de^{-1}) \conggr R(\de)(\de^{-1})$; that is, if and only if
$R(\de^{-1} \al) \conggr R$. Then by Proposition~\ref{rdeconggrr},
$R(\al) \conggr R(\de)$ as graded $R$-modules if and only if
$\de^{-1} \al \in \Ga_R^*$.

For a $\Ga$-graded ring $R$, consider $M_{n\times m} (R)$, the set
of $n \times m$ matrices over $R$. For $n=m$, in
Section~\ref{sectiongradedmatrixrings}, we will define grading on
the matrix ring $M_n (R)$, which gives $M_n (R)$ the structure of a
graded ring.
Here we define shifted matrices $M_{n \times m} (R)$, which will be
used in the following proposition.
%
Note that they do not have the structure of either graded rings or
graded $R$-modules.
For $(d)=(\de_1, \ldots ,\de_n) \in \Ga^n$, $(a) = (\al_1 , \ldots
\al_m) \in \Ga^m$, let
$$ \label{pagerefmatrixwithshifting}
M_{n \times m} (R)[d][a] =
\begin{pmatrix}
R_{\de_1^{-1} \al_1} & R_{ \de_1^{-1} \al_2} & \cdots &
R_{\de_1^{-1} \al_m} \\
R_{\de_2^{-1} \al_1} & R_{ \de_2^{-1} \al_2} & \cdots &
R_{\de_2^{-1} \al_m} \\
\vdots  & \vdots  & \ddots & \vdots  \\
R_{\de_n^{-1} \al_1} & R_{ \de_n^{-1} \al_2} & \cdots &
R_{\de_n^{-1} \al_m}
\end{pmatrix}
$$
So $M_{n\times m} (R)[d][a]$ consists of matrices with the
$ij$-entry in $R_{\de_i^{-1} \al_j}$. Further, we let $\GL_{n\times
m} (R)[d][a]$ denote the set of invertible $n \times m$ matrices
with shifting as above. The following proposition extends
Proposition~\ref{rdeconggrr}. A similar argument will be used in the
proof of Proposition~\ref{grckidfunctorprop}.



\begin{prop} \label{rndconggrrna}
Let $R$ be a $\Ga$-graded ring and let $(d)=(\de_1, \ldots
,\de_n)\in \Ga^n,$ $(a) = (\al_1 , \ldots , \al_m) \in \Ga^m$. Then
$R^n(d) \conggr R^m (a)$ as graded $R$-modules if and only if there
exists $(r_{ij}) \in \GL_{n\times m} (R)[d][a]$.
\end{prop}

\begin{proof}
If $r = (r_{ij}) \in \GL_{n \times m} (R)[d][a]$, then there is a
graded $R$-module homomorphism
\begin{align*}
\mar_r :R^n (d)& \lra R^m(a) \\
(x_1 , \ldots, x_n) & \lmps (x_1 , \ldots, x_n) r.
\end{align*}
Since $r$ is invertible, there is a matrix $t \in \GL_{m \times n}
(R)$ with $rt=I_{n}$ and $tr = I_{m}$. So there is an $R$-module
homomorphism $\mar_t : R^m (a) \ra R^n(d)$, which is an inverse of
$\mar_r$. This proves that $\mar_r$ is bijective, and therefore it
is a graded $R$-module isomorphism.

Conversely, if $\phi : R^n (d) \conggr R^m(a)$, then we can
construct a matrix as follows. Let $e_i$ denote the basis element of
$R^n(d)$ with $1$ in the $i$-th entry and zeros elsewhere. Then let
$\phi (e_i) = (r_{i1}, r_{i2} , \ldots , r_{im})$, and let $r=
(r_{ij})_{n \times m}$. It can be easily verified that $r \in M_{n
\times m} (R)[d][a]$.
%
%
In the same way, using $\phi^{-1} : R^m (a) \ra R^n(d)$ construct a
matrix $t$. Let $e'_i$ denote the $i$-th element of the standard
basis for $R^m (a)$. Since
$$
e_i = \phi^{-1} \circ \phi (e_i) = r_{i1} \phi^{-1}(e'_1) + r_{i2}
\phi^{-1} (e'_2)+ \cdots + r_{im} \phi^{-1}(e'_m)
$$
for each $i$, and in a similar way for $\phi \circ \phi^{-1}$, we
can show that $rt=I_n$ and $tr= I_m$. So $(r_{ij}) \in \GL_{n \times
m} (R)[d][a]$.
%
\end{proof}

For convenience, in the above definition of $M_{n \times m} (R)
[d][a]$, if $(a) = (e, \ldots , e)$, we will write $M_{n \times m}
(R)[d]$ instead of $M_{n \times m} (R)[d][e]$. We let
$$
\Ga^*_{M_{n \times m}(R)} = \big \{ (d) \in \Ga^n : \GL_{n \times
m}(R)[d] \neq \emptyset \big \}.
$$
Then it is immediate from the above proposition that $R^n(d) \conggr
R^n$ as graded $R$-modules if and only if $(d) \in \Ga^*_{M_n(R)}$.

A $\Ga$-graded ring $D = \bigoplus_{ \ga \in \Ga} D_{\ga}$ is called
a \emph{graded division ring} \index{graded division ring} if every
non-zero homogeneous element has a multiplicative inverse. It
follows from Proposition~\ref{basicsofgradedrings}(4) that $\Ga_D$
is a group, so we can write $D = \bigoplus_{ \ga \in \Ga_D}
D_{\ga}$. Then, as a $\Ga_D$-graded ring, $D$ is a crossed product
and it follows from
Proposition~\ref{crossedproductstronglygradedprop}(3) that $D$ is
strongly $\Ga_D$-graded.

\begin{examples}\label{egofgrdivisionrings}
\begin{enumerate}
\item Let $E$ be a division ring and let $D= E[x, x^{-1}]$ be the
Laurent polynomials over $E$. Then $D$ is a graded division ring,
with $D = \bigoplus_{n \in \mathbb{Z}} D_n$ where $D_n = \{ a x^n :
a \in E\}$.

\item Let $\H = \R \+ \R i \+ \R j \+ \R k$ be the real quaternion
algebra, with multiplication defined by $i^2 = -1$, $j^2= -1$ and $i
j = -ji = k$. It is known that $\H$ is a non-commutative division
ring with centre $\R$. We note that $\H$ can be given two different
graded division ring structures, with grade groups $\Z_2$ and $\Z_2
\times \Z_2$ respectively.

For $\Z_2$: \; Let $\C = \R \+ \R i$. Then $\H = \C_0 \+ \C_1$,
where $\C_0 = \C$ and $\C_1 = \C j$.

For $\Z_2 \times \Z_2$: \; Let $\H = R_{(0,0)} \+ R_{(1,0)} \+
R_{(0,1)} \+ R_{(1,1)}$, where
$$
R_{(0,0)} = \, \R ,\;\; R_{(1,0)} = \, \R i,\;\; R_{(0,1)} = \, \R
j,\;\; R_{(1,1)} = \, \R k.
$$
In both cases, it is routine to show $\H$ that forms a graded
division ring.
\end{enumerate}

\end{examples}

In the following propositions we are considering graded modules over
graded division rings. We note that the grade groups here are
defined as above; that is, we do not initially assume them to be
abelian or torsion-free. The proofs follow the standard proofs in
the non-graded setting (see \cite[\S IV,
Thms.~2.4,~2.7,~2.13]{hungerford}), or the graded setting (see
\cite[Thm.~3]{boulaggradedandtame}, \cite[Prop.~4.6.1]{grrings},
\cite[p.~79]{hwcor}); however extra care needs to be given since the
grading is neither abelian nor torsion-free.

\begin{prop}\label{gradedfree}
Let $\Ga$ be a group which acts freely on a set $\Ga'$. Let $R$ be a
$\Ga$-graded division ring and $M$ be a $\Ga'$-graded module over
$R$. Then $M$ is a graded free $R$-module. More generally, any
linearly independent subset of $M$ consisting of homogeneous
elements can be extended to form a homogeneous basis of $M$.
\end{prop}

\begin{proof}
Note that the first statement is an immediate consequence of the
second, since for any $m \in M^h$, $\{m \}$ is a linearly
independent subset of $M$. Fix a linearly independent subset $X$ of
$M$ consisting of homogeneous elements. Let
$$
F= \{ Q \subseteq M^h : X \subseteq Q \textrm{ and $Q$ is
$R$-linearly independent} \}.
$$
This is a non-empty partially ordered set with inclusion, and every
chain $Q_1 \subseteq Q_2 \subseteq \ldots$ in $F$ has an upper bound
$\bigcup Q_i \in F$. By Zorn's Lemma, $F$ has a maximal element,
which we denote by $P$. If $\left< P \right> \neq M$, then there is
a homogeneous element $m \in M^h \mi \left< P \right>$. We will show
that $P \cup \{m \}$ is a linearly independent set containing $X$,
contradicting our choice of $P$.

Suppose $rm + \sum r_i p_i = 0$, where $r, r_i \in R$, $p_i \in P$
and $r \neq 0$.
%
Since $r \neq 0$ there is a homogeneous component of $r$, say
$r_{\la}$, which is also non-zero. Considering the $\la
\deg(m)$-homogeneous component of this sum, we have $m= r_{\la}^{-1}
r_{\la} m =- \sum r_{\la}^{-1} r_i p_i$ for $r_i$ homogeneous, which
contradicts our choice of $m$. Hence $r=0$, which implies each $r_i
= 0$. This gives the required contradiction, so $M= \left< P
\right>$, completing the proof.
\end{proof}

\begin{prop}
Let $\Ga$ be a group which acts freely on a set $\Ga'$. Let $R$ be a
$\Ga$-graded division ring and $M$ be a $\Ga'$-graded module over
$R$. Then any two homogeneous bases of $M$ over $R$ have the same
cardinality.
\end{prop}

\begin{proof}
Suppose $M$ has an infinite basis $Z$. Since $R$ is in particular a
ring, we can apply \cite[Thm.~IV.2.6]{hungerford}, which shows that
every basis of $M$ has the same cardinality as $Z$. We now assume
that $M$ has two homogeneous bases $X$ and $Y$, where $X$ and $Y$
are finite. Then $X= \{ x_1, \ldots , x_n\}$ and $Y= \{ y_1 , \ldots
, y_m \}$ for $x_i$, $y_i \in M^h \mi 0$. As $X$ is a basis for $M$,
we can write
$$
y_m = r_1 x_1 + \cdots + r_n x_n
$$
for some $r_i \in R^h$, where $\deg (y_m) = \deg(r_i ) \deg (x_i)$
for each $i$. Since $y_m \neq 0$, we have at least one $r_i \neq 0$.
Let $r_k$ be the first non-zero $r_i$, and we note that $r_k$ is
invertible as it is non-zero and homogeneous in $R$. Then
$$
x_k = r_k^{-1} y_m - r_k^{-1} r_{k+1} x_{k+1} - \cdots - r_k^{-1}
r_n x_n,
$$
and the set $X' = \{y_m ,  x_1, \ldots , x_{k-1} , x_{k+1} , \ldots
, x_n\}$ spans $M$ since $X$ spans $M$, where $X'$ consists of
homogeneous elements. So
$$
y_{m-1} = s_m y_m +  t_1 x_1 + \cdots + t_{k-1} x_{k-1} + t_{k+1}
x_{k+1} + \cdots + t_n x_n
$$
for $s_m$, $t_i \in R^h$. There is at least one non-zero $t_i$,
since if all the $t_i$ are zero, then either $y_m$ and $y_{m-1}$ are
linearly dependent or $y_{m-1}$ is zero. Let $t_j$ denote the first
non-zero $t_i$. Then $x_j$ can be written as a linear combination of
$y_{m-1}$, $y_m$ and those $x_i$ with $i \neq j, k$. Therefore the
set $X'' = \{y_{m-1} , y_m \} \cup \{x_i : i \neq j,k\}$ spans $M$
since $X'$ spans $M$. We can write $y_{m-2}$ as a linear combination
of the elements of $X''$.

Continuing this process of adding a $y$ and removing an $x$ gives,
after the $k$-th step, a set which spans $M$ consisting of $y_m ,
y_{m-1} , \ldots , y_{m-k+1}$ and $n-k$ of the $x_i$. If $n <m$,
then after the $n$-th step, we would have that the set $\{ y_m ,
\ldots , y_{m-n+1} \}$ spans $M$. But if $n <m$, then $m-n +1 \geq
2$, so this set does not contain $y_1$, and therefore $y_1$ can be
written as a linear combination of the elements of this set. This
contradicts the linear independence of $Y$, so we must have $m \leq
n$. Repeating a similar argument with $X$ and $Y$ interchanged gives
$n \leq m$, so $n=m$.
\end{proof}

The propositions above say that for a graded module $M$ over a
graded division ring $R$, $M$ has a homogeneous basis and any two
homogeneous bases of $M$ have the same cardinality. The cardinal
number of any homogeneous basis of $M$ is called the dimension of
$M$ over $R$, and it is denoted by $\dim_R (M)$.

\begin{prop}\label{grdimensionprop}
Let $\Ga$ be a group which acts freely on a set $\Ga'$. Let $R$ be a
$\Ga$-graded division ring and $M$ be a $\Ga'$-graded module over
$R$. If $N$ is a graded submodule of $M$, then
\begin{equation*}
\dim_R (N) + \dim_R(M/N) = \dim_R (M).
\end{equation*}

\end{prop}

\begin{proof}
Let $Y$ be a homogeneous basis of $N$. Then $Y$ is a linearly
independent subset of $M$ consisting of homogeneous elements, so
using Proposition~\ref{gradedfree}, there is a homogeneous basis $X$
of $M$ containing $Y$. We will show that $U= \{ x + N :x \in X \mi Y
\}$ is a homogeneous basis of $M/N$. Note that clearly $U$ consists
of homogeneous elements. Let $m + N \in (M/N)^h$. Then $m \in M^h$
and $m = \sum r_i x_i + \sum s_j y_j$ where $r_i$, $s_j \in R$, $y_j
\in Y$ and $x_i \in X \mi Y$. So $m + N = \sum r_i (x_i +N)$, which
shows that $U$ spans $M/N$. If $\sum r_i (x_i +N) = 0$, for $r_i \in
R$, $x_i \in X \mi Y$, then $\sum r_i x_i \in N$ which implies that
$\sum r_i x_i = \sum s_k y_k$ for $s_k \in R$ and $y_k \in Y$. But
$X = Y \cup (X \mi Y)$ is linearly independent, so $r_i = 0$ and
$s_k = 0$ for all $i$,~$k$. Therefore $U$ is a homogeneous basis for
$M/N$ and as we can construct a bijective map $U \ra X \mi Y$, we
have $|U| = |X \mi Y|$. Then $\dim_R M = |X| = |Y| +|X \mi Y| = |Y|
+ |U| = \dim_R N + \dim_R (M/N)$.
\end{proof}


\section{Graded central simple algebras}
\label{sectiongradedcentralsimplealgebras}


Let $\Ga$ be a multiplicative group and let $R$ be a commutative
$\Ga$-graded ring. A \emph{$\Ga$-graded $R$-algebra}\index{graded
algebra} $A= \bigoplus_{\ga \in \Ga} A_{\ga}$ is defined to be a
graded ring which is an $R$-algebra such that the associated ring
homomorphism $\varphi$ is a graded homomorphism, where $\varphi :R
\ra A$ with $\varphi (R) \subseteq Z(A)$. Graded $R$-algebra
homomorphisms and graded subalgebras are defined analogously to the
equivalent terms for graded rings or graded modules.

Let $A$ and $B$ be graded $R$-algebras, such that $\Ga_A \subseteq
Z_{\Ga} (\Ga_B)$, where $Z_{\Ga} (\Ga_B)$ is the set of elements of
$\Ga$ which commute with $\Ga_B$.  Then $A \otimes_R B$ has a
natural grading as a graded $R$-algebra given by $A \otimes_R B =
\bigoplus_{\ga \in \Ga} (A \otimes_R
B)_{\ga}$\label{pagerefforgradingontensorproduct} where:
$$
(A \otimes_R B)_{\ga} = \left\{ \sum_i a_i \otimes b_i : a_i \in
A^h, b_i \in B^h, \deg(a_i) \deg(b_i) = \ga \right\}
$$
Note that the condition $\Ga_A \subseteq Z_{\Ga} (\Ga_B)$ is needed
to ensure that the multiplication on $A \otimes_R B$ is well
defined.

Let $R$ be a $\Ga$-graded ring, $M$ be a $\Ga$-graded $R$-module and
let $\END_R (M) = \HOM_R (M, M)$, where $\HOM_R (M,M)$ is defined on
page~\pageref{hompageref}. The ring $\END_R (M)$\label{endisagrring}
is a $\Ga$-graded ring with the usual addition and with
multiplication defined as $g \cdot f = f \circ g$ for all $f, g \in
\END_R (M)$. If $M$ is graded free with a finite homogeneous base,
then we have $\END_R (M) = \End_R (M)$
(see \cite[Remark~2.10.6(ii)]{grrings}). In fact, if $M$ and $N$ are
$\Ga$-graded $R$-modules with $M$ finitely generated, then
$\HOM_R(M,N) = \Hom_R(M,N)$ (see \cite[Cor.~2.4.4]{grrings}).
Further, if $R$ is a commutative $\Ga$-graded ring and $\Ga_R
\subseteq Z_{\Ga} (\Ga_M )$, then $\End_R (M)$ is a $\Ga$-graded
$R$-algebra.

A \emph{graded field}\index{graded field} is defined to be a
commutative graded division ring. Note that the support of a graded
field is an abelian group. Let $R$ be a graded field. A graded
algebra $A$ over $R$ is said to be a \emph{graded central simple
algebra}\index{graded central simple algebra}\index{graded
algebra!graded central simple} over $R$ if $A$ is graded simple as a
graded ring, $Z(A) \conggr R$ and $[A:R] < \infty$. Note that since
the centre of $A$ is $R$, which is a graded field, $A$ is graded
free as a graded module over its centre by
Proposition~\ref{gradedfree}, so the dimension of $A$ over $R$ is
uniquely defined.

For a $\Ga$-graded ring $A$, let $A^{\op}$ denote the opposite
graded ring, where the grade group of $A^{\op}$ is the opposite
group $\Ga^{\op}$. So for a graded $R$-algebra $A$, in order to
define $A \otimes_R A^{\op}$, we note that the support of $A$ must
be abelian.
%
%
Moreover for the following propositions, we require that the group
$\Ga$ is an abelian group (see Theorem~\ref{groupringazumayathm} and
the preceding comments). \emph{Since the grade groups are assumed to
be abelian for the remainder of this section, we will write them as
additive groups.}

By combining Propositions~\ref{tensorgradedsimple} and
\ref{tensorgradedcentral}, we show that the tensor product of two
graded central simple $R$-algebras is graded central simple, where
the grade group $\Ga$, as below, is abelian but not necessarily
torsion-free. This has been proven by Wall for graded central simple
algebras with $\mathbb Z / 2 \mathbb Z$ as the support (see
\cite[Thm.~2]{wall}), and by Hwang and Wadsworth for $R$-algebras
with a torsion-free grade group (see \cite[Prop.~1.1]{hwcor}).

\begin{prop}\label{tensorgradedsimple}

Let $\Ga$ be an abelian group. Let $R$ be a $\Ga$-graded field and
let $A$ and $B$ be $\Ga$-graded $R$-algebras. If $A$ is graded
central simple over $R$ and $B$ is graded simple, then $A \otimes_R
B$ is graded simple.
\end{prop}

\begin{proof}
Let $I$ be a homogeneous two-sided ideal of $A \otimes B$, with $I
\neq 0$. We will show that $A \otimes B = I$. Note that since $I$ is
a homogeneous ideal, by
Proposition~\ref{homogeneousidealhomogeneouselementsprop}, it is
generated by homogeneous elements. First suppose $a \otimes b$ is a
homogeneous element of $I$, where $a \in A^h$ and $b \in B^h$. Then
$A$ is the homogeneous two-sided ideal generated by $a$, so there
exist $a_i ,$ $a'_i \in A^h$ with $1= \sum a_i a a'_i$. Then
$$\sum (a_i \otimes 1)( a \otimes b)( a'_i \otimes 1) = 1 \otimes b$$ is an
element of $I$. Similarly, $B$ is the homogeneous two-sided ideal
generated by $b$. Repeating the above argument shows that $1 \otimes
1$ is an element of $I$, proving $I= A\otimes B$ in this case.

Now suppose there is an element $x \in I^h$, where $x = a_1 \otimes
b_1 + \cdots + a_k \otimes b_k$, with $a_j \in A^h$, $b_j \in B^h$
and $k$ as small as possible. Note that since $x$ is homogeneous,
$\deg(a_j) +\deg (b_j) = \deg (x)$ for all $j$. By the above
argument we can suppose that $k >1$. As above, since $a_k \in A^h$,
there are $c_i ,$ $c'_i \in A^h$ with $1= \sum c_i a_k c'_i$. Then
$$
\sum (c_i \otimes 1)x( c'_i \otimes 1) = \left( \sum (c_i a_1
c'_i) \right) \otimes b_1 + \cdots + \left( \sum (c_i a_{k-1} c'_i)
\right) \otimes b_{k-1} + 1 \otimes b_k,
$$
where the terms $\left( \sum_i (c_i a_j c'_i)\right) \otimes b_j$
are homogeneous elements of $A \otimes B$. Thus, without loss of
generality, we can assume that $a_k = 1$. Then $a_k$ and $a_{k-1}$
are linearly independent, since if $a_{k-1} = r a_k$ with $r \in R$,
then $a_{k-1} \otimes b_{k-1} + a_k \otimes b_k = a_k \otimes (r
b_{k-1} +b_k)$, which is homogeneous and thus gives a smaller value
of $k$.

Thus $a_{k-1} \notin R=Z(A)$, and so there is a homogeneous element
$a \in A$ with $a a_{k-1} - a_{k-1} a \neq 0$. Consider the
commutator $$(a \otimes 1) x - x (a \otimes 1) = (aa_1 - a_1 a )
\otimes b_1 + \cdots + (a a_{k-1} - a_{k-1} a)\otimes b_{k-1},$$
where the last summand is not zero. If the whole sum is not zero,
then we have constructed a homogeneous element in $I$ with a smaller
$k$. Otherwise suppose the whole sum is zero, and write $c = a
a_{k-1} - a_{k-1} a $. Then we can write $c \otimes b_{k-1} =
\sum_{j=1}^{k-2} x_j \otimes b_j$ where $x_j = -(a a_j -a_j a)$.
Since $0\not = c \in A^h$ and $A$ is the homogeneous two-sided ideal
generated by $c$, using the same argument as above, we have
\begin{equation} \label{li}
1 \otimes b_{k-1} = x'_1 \otimes b_1 + \cdots + x'_{k-2} \otimes
b_{k-2}\end{equation} for some $x'_j \in A^h$. Since $b_1 , \ldots ,
b_{k-1}$ are linearly independent homogeneous elements of $B$, they
can be extended to form a homogeneous basis of $B$, say $\{b_i\}$,
by Proposition~\ref{gradedfree}. Then $\{ 1 \otimes b_i\}$ forms a
homogeneous basis of $A\otimes_R B$ as a graded $A$-module, so in
particular they are $A$-linearly independent, which is a
contradiction to equation~\eqref{li}. This reduces the proof to the
first case.
\end{proof}

\begin{prop}\label{tensorgradedcentral}

Let $\Ga$ be an abelian group. Let $R$ be a $\Ga$-graded field and
let $A$ and $B$ be $\Ga$-graded $R$-algebras. If $A' \subseteq A$
and $B' \subseteq B$ are graded subalgebras, then
$$Z_{A \otimes B}(A' \otimes B') = Z_A (A') \otimes Z_B(B').$$
In particular, if $A$ and $B$ are central over $R$, then $A
\otimes_R B$ is central.
\end{prop}

\begin{proof}
First note that by Proposition~\ref{gradedfree}, $A',B',Z_A(A')$ and
$Z_B(B')$ are graded free over $R$, so they are flat over $R$. Thus
we can consider $Z_{A \otimes B}(A' \otimes B')$ and $Z_A (A')
\otimes Z_B(B')$ as graded subalgebras of $A\otimes B$.

The inclusion $\supseteq$ follows immediately. For the reverse
inclusion, let $x \in Z_{A \otimes B}(A' \otimes B')$. Let $\{ b_1 ,
\ldots , b_n \}$ be a homogeneous basis for $B$ over $R$ which
exists thanks to Proposition~\ref{gradedfree}. Then $x$ can be
written uniquely as $x = x_1 \otimes b_1 + \cdots + x_n \otimes b_n$
for $x_i \in A$ (see \cite[Thm.~IV.5.11]{hungerford}). For every $a
\in A'$, $(a \otimes 1) x = x (a \otimes 1)$, so
$$ (a x_1) \otimes b_1 + \cdots + (a x_n) \otimes b_n   = (x_1 a)
\otimes b_1 + \cdots + (x_n a) \otimes b_n.$$ By the uniqueness of
this representation we have $x_i a = a x_i$, so that $x_i \in Z_A
(A')$ for each $i$.  Thus we have shown that $x \in Z_A(A')
\otimes_R B$. Similarly, let $\{ c_1 , \ldots , c_k \}$ be a
homogeneous basis of $Z_A(A')$. Then we can write $x$ uniquely as $x
=\linebreak c_1 \otimes y_1 + \cdots + c_k \otimes y_k$ for $y_i \in
B$. A similar argument to above shows that $y_i \in Z_B(B')$,
completing the proof.
\end{proof}

In the following theorem, since $\Ga$ is an abelian group, we define
multiplication in $\End_R (A)$ to be $g \cdot f = g \circ f$.

\begin{thm}\label{gcsaazumayaalgebra}
Let $\Ga$ be an abelian group. Let $A$ be a $\Ga$-graded central
simple algebra over a $\Ga$-graded field $R$. Then $A$ is an Azumaya
algebra over $R$.
\end{thm}

\begin{proof}
Since $A$ is graded free of finite dimension over $R$, it follows
that $A$ is faithfully projective over $R$. There is a natural
graded $R$-algebra homomorphism $\psi: A \otimes_R A^{\op} \ra
\End_R (A)$ defined by $\psi(a\otimes b)(x)=axb$ where $a,x \in A$,
$b \in A^{\op}$. Since graded ideals of $A^{\op}$ are the same as
graded ideals of $A$, we have that $A^{\op}$ is graded simple. By
Proposition~\ref{tensorgradedsimple}, $A \otimes A^{\op}$ is also
graded simple, so $\psi$ is injective. Hence the map is surjective
by dimension count, using Theorem~\ref{grdimensionprop}. This shows
that $A$ is an Azumaya algebra over $R$, as required.
\end{proof}

\begin{cor}\label{gcsacor}
Let $\Ga$  be an abelian group. Let $A$ be a $\Ga$-graded central
simple algebra over a $\Ga$-graded field $R$ of dimension $n$. Then
for any $i \geq 0$,
$$
K_i(A) \otimes \mathbb Z[1/n] \cong K_i(R) \otimes \mathbb
Z[1/n].
$$
\end{cor}

\begin{proof}
By Theorem~\ref{gcsaazumayaalgebra}, a graded central simple algebra
$A$ over $R$ is an Azumaya algebra. From
Proposition~\ref{gradedfree}, since $R$ is a graded field, $A$ is a
free $R$-module. The corollary now follows immediately from Theorem
\ref{azumayafreethm}, since $A$ is an Azumaya algebra free over its
centre.
\end{proof}

For a graded field $R$, Theorem~\ref{gcsaazumayaalgebra} shows that
a graded central simple $R$-algebra, graded by an abelian group
$\Ga$, is an Azumaya algebra over $R$. This theorem can not be
extended to cover non-abelian grading. Consider a finite dimensional
division algebra $D$ and a group $G$. The group ring $DG$ is a
graded division ring, so is clearly a graded simple algebra, and if
$G$ is abelian, Theorem~\ref{gcsaazumayaalgebra} implies that $DG$
is an Azumaya algebra. However in general, for an arbitrary group
$G$, $DG$ is not always an Azumaya algebra. DeMeyer and Janusz
\cite{demjan} have proven the following theorem.

\begin{thm}\label{groupringazumayathm}
Let $R$ be a ring and let $G$ be a group. Then the group ring $RG$
is an Azumaya algebra if and only if the following three conditions
hold:
\begin{enumerate}
\item $R$ is an Azumaya algebra,

\item $[G:Z(G)] < \infty$,

\item the commutator subgroup of $G$ has finite order $m$ and
$m$ is invertible in $R$.
\end{enumerate}
\end{thm}

\begin{proof}
See \cite[Thm.~1]{demjan}.
\end{proof}



\section{Graded matrix rings} \label{sectiongradedmatrixrings}

Let $\Ga$ be a group and $R$ be a $\Ga$-graded ring. We will write
$\Ga$ as a multiplicative group, since $\Ga$ is not necessarily
abelian. Throughout this section, unless otherwise stated, we will
assume that all graded rings, graded modules and graded algebras are
also $\Ga$-graded. Let $\la \in \Ga$ and $(d) = (\de_1 , \ldots ,
\de_n) \in \Ga^n$. Let ${M_n (R)(d)}_{\la}$ denote the $n \times
n$-matrices over $R$ with the degree shifted as follows:
$$
{M_n(R)(d)}_{\la} =
\begin{pmatrix}
R_{\de_1  \la  \de_1^{-1}} & R_{\de_1  \la  \de_2^{-1}} & \cdots &
R_{\de_1  \la  \de_n^{-1}} \\
R_{\de_2  \la  \de_1^{-1}} & R_{\de_2  \la  \de_2^{-1}} & \cdots &
R_{\de_2  \la  \de_n^{-1}} \\
\vdots  & \vdots  & \ddots & \vdots  \\
R_{\de_n  \la  \de_1^{-1}} & R_{\de_n  \la  \de_2^{-1}} & \cdots &
R_{\de_n \la  \de_n^{-1}}
\end{pmatrix}
$$
Thus ${M_n(R)(d)}_{\la}$ consists of matrices with the $ij$-entry in
$R_{\delta_i \la  \delta_j^{-1}}$.

\begin{prop} \label{grmatrixring}\index{graded matrix ring}
With the notation as above, there is a graded ring
$$
M_n (R)(d) = \bigoplus_{\la \in \Ga} {M_n (R)(d)}_{\la}.
$$
\end{prop}

\begin{proof}
(See \cite[Prop.~2.10.4]{grrings}.) For any $\la_1$, $\la_2 \in
\Ga$,
$$
{M_n (R)(d)}_{\la_1} {M_n (R)(d)}_{\la_2} \subseteq {M_n
(R)(d)}_{\la_1  \la_2}.
$$
For any $i, j$ with $1 \leq i,j \leq n$, $R_{\delta_i \la
\delta_j^{-1}} \cap \sum_{\ga \neq \la} R_{\delta_i \ga
\delta_j^{-1}}= 0$, so
$$
{M_n(R)(d)}_{\la} \cap \left( \sum_{\ga \neq \la} {M_n(R)(d)}_{\ga}
\right) =0.
$$
Any matrix in $M_n(R)$ can be written as a sum of matrices with a
homogeneous element in one entry and zeros elsewhere. If $A \in M_n
(R)$ has $a \in R_{\epsilon}$ in the $ij$-entry and zeros elsewhere,
then taking $\la = \de_i^{-1}  \epsilon \hspace{.2 ex} \de_j$ gives
that $A \in {M_n(R)(d)}_{\la}$, and hence any matrix in $M_n (R)$
can be written as an element of ${M_n(R)(d)}$.
\end{proof}

Let $M$ be a graded $R$-module which is graded free with a finite
homogeneous base $\{b_1, \ldots , b_n\}$, where $\deg (b_i) =
\delta_i$. We noted in the previous section that the ring $\End_R
(M)$ is a graded ring with multiplication defined as $g \cdot f = f
\circ g$ for all $f, g \in \End_R (M)$. If we ignore the grading, it
is well-known that $\End_R(M) \cong M_n(R)$ as rings (see
\cite[Thm.~VII.1.2]{hungerford}). When we take the grading into
account, the following proposition shows that this isomorphism is in
fact a graded isomorphism.


\begin{prop}\label{endgrisomatrixring}
Let $M$ be a graded free $\Ga$-graded $R$-module with a finite
homogeneous base $\{b_1, \ldots , b_n\}$, where $\deg (b_i) =
\delta_i$. Then $\End_R(M) \conggr M_n(R)(d)$ as $\Ga$-graded rings,
where $(d) = (\delta_1 , \ldots, \delta_n ) \in \Ga^n$.
\end{prop}

\begin{proof}
(See \cite[Prop.~2.10.5]{grrings}.) The remarks before the
proposition show that $\End_R(M) \cong M_n(R)$. To show that the
isomorphism is graded, let $f \in \End_R (M)_\la$ for some $\la \in
\Ga$. Then as in the non-graded setting, there are elements $r_{ij}
\in R$ such that
\begin{align*}
f (b_1 ) &= r_{11} b_1 + r_{12} b_2 + \cdots +r_{1n} b_n \\
f (b_2 ) &= r_{21} b_1 + r_{22} b_2 + \cdots +r_{2n} b_n \\
\vdots & \\
f (b_n ) &= r_{n1} b_1 + r_{n2} b_2 + \cdots +r_{nn} b_n ,
\end{align*}
with associated matrix $(r_{ij}) \in  M_n(R)$. Since we have $f (b_i
) \in M_{{\de}_i \la}$ for each $b_i$, it follows that each $r_{ij}$
is homogeneous of degree $\de_i  \la  \de_j^{-1}$. So $(r_{ij}) \in
M_n(R)(d)_{\la}$ as required.
\end{proof}

\begin{remark}\label{remarkongrmatrixrings}
Note that above all graded $R$-modules are considered as left
modules. For the $\lambda$-component of the graded matrix ring $M_n
(R)(d)$, we set the degree of the $ij$-entry to be $\de_i \lambda
\de_j^{-1}$. Then in the above proposition, we defined the
multiplication in $\End_R (M)$ as $g\cdot f = f \circ g$ to ensure
that the isomorphism is a graded ring isomorphism.

If we define multiplication in $\End_R (M)$ by $g \cdot f = g \circ
f$, then in the non-graded setting, we have a ring isomorphism
$\End_R (M) \ra M_n (R^{\op})$, where this map is the composition of
the homomorphism in the above proposition with the transpose map. To
make $\End_R (M)$, with this multiplication, into a graded ring, it
will be graded by $\Ga^{\op}$. In the above proposition, we will
have $\End_R (M) \conggr M_n (R^{\op})(d)$, where $M_n (R^{\op})(d)$
is also $\Ga^{\op}$-graded and the degree of the $ij$-entry of $M_n
(R^{\op}) (d)$ is defined to be $\de_i^{-1} \cdot_{\op} \lambda
\cdot_{\op} \de_j$.

In Definition~\ref{grazumayadefin} and in
Section~\ref{sectiongrdfunctors}, since $\Ga$ is an abelian group
and $R$ is a commutative graded ring, we will define multiplication
in endomorphism rings to be $g \cdot f = g \circ f$. Then we will
use the grading on matrix rings mentioned in the previous paragraph.

Suppose $M$ is a graded right $R$-module  and multiplication in
$\End_R (M)$ is defined by $g \cdot f = g \circ f$ (see
page~\pageref{pagerefforgrrightrmodule} for some comments on graded
right $R$-modules). Then to get a graded ring isomorphism $\End_R(M)
\conggr M_n(R)(d)$, we need to define the grading on the matrix ring
$M_n (R)(d)$ as having its $ij$-entry in the $\lambda$-component of
degree $\de_i^{-1} \lambda \de_j$.
\end{remark}


With a graded $R$-module $M$ defined as in the above proposition, we
define the graded $R$-module homomorphisms $E_{ij}$ by $E_{ij} (b_l)
= \mbox{{\boldmath $\de$}}_{il} b_j$ for $1 \leq i , j, l \leq n$,
where $\mbox{{\boldmath $\de$}}_{il}$ is the Kronecker delta. We
note that the graded $R$-module homomorphisms $E_{ij}$, for $1 \leq
i , j, l \leq n$, are homogeneous elements in $\End_R (M)$ of degree
$\de_i^{-1} \de_j$. Then $\{ E_{ij} : 1 \leq i ,j \leq n \}$ forms a
homogeneous basis for $\End_R (M)$. Via the isomorphism $\End_R(M)
\conggr M_n(R)(d)$, the map $E_{ij}$ corresponds to the matrix
$e_{ij}$, where $e_{ij}$ is the matrix with $1$ in the $ij$-entry
and zeros elsewhere. This observation will be used in the following
proposition to prove some basic properties of the graded matrix
rings.




\begin{prop} \label{grmatrixringsproperties}
Let $R$ be a $\Ga$-graded ring and let $(d) = (\delta_1 , \ldots,
\delta_n ) \in \Ga^n$.
\begin{enumerate}
\item If $\pi \in S_n$ is a permutation, then
\begin{equation*}
M_n (R)(\de_1 , \ldots, \de_n) \conggr M_n (R)(\de_{\pi(1)}, \ldots
, \de_{\pi(n)}).
\end{equation*}

\item If 
$(\ga_1, \ldots, \ga_n) \in \Ga^n$ with each $\ga_i \in {\Ga_R^*}$,
then
\begin{equation*}
M_n (R)(\de_1 , \ldots , \de_n) \cong_{\gr} M_n (R)(\ga_1  \de_1 ,
\ldots , \ga_n  \de_n).
\end{equation*}
If $R$ is a graded division ring, then any set of $(\ga_1 , \ldots ,
\ga_n) \in \Ga_R^n$ can be chosen.

\item If $\sigma \in Z(\Ga)$, the centre of $\Ga$, then
\begin{equation*}
M_n (R)(\de_1 , \ldots , \de_n) = M_n (R)(\de_{1} \si , \ldots ,
\de_{n} \si).
\end{equation*}
\end{enumerate}
\end{prop}

\begin{proof}
(See \cite[p.~78]{hwcor}, \cite[Remarks~2.10.6]{grrings}.) We
observed on page~\pageref{rngradedfree} that
$$
M = R(\de_1^{-1}) \oplus \cdots \oplus R(\de_n^{-1})
$$
has a homogeneous basis $\{e_1, \ldots , e_n\}$ with $\deg(e_i) =
\de_i$. So from the above proposition, we have
$\End_R(M) \conggr M_n(R)(\de_1, \dots, \de_n)$, where the map
$E_{ij}$ corresponds to the matrix $e_{ij}$.

\vspace{3pt}

(1)  Let $\de_{\pi (i)} = \tau_i$. Then  $N = R( \tau_1^{-1} )
\oplus \cdots \oplus R(\tau_n^{-1})$ has a standard homogeneous
basis $\{e'_1, \ldots , e'_n \}$ with $\deg (e'_i ) = \tau_i$, and
we have
$$
\End_R(N) \conggr M_n(R)(\tau_1, \dots, \tau_n).
$$
We define a graded $R$-module isomorphism $\phi: M \ra N$ by
$\phi(e_i)=e'_{\pi^{-1}(i)}$, which induces a graded isomorphism
$\phi : \End_R (M) \ra \End_R (N)$; $f \mps \phi \circ f \circ
\phi^{-1}$. Combining these graded isomorphisms gives the required
result:
$$
M_n(R)(\de_1, \dots, \de_n) \conggr \End_A(M) \stackrel{\phi}{\lra}
\End_R(N)\conggr M_n(R)(\de_{\pi(1)}, \dots, \de_{\pi(n)}).
$$

\vspace{3pt}

(2) Let $u_1 , \ldots , u_n \in {R}$ be homogeneous units of $R$
with $\deg(u_i) = \ga_i$. Let $N = R(\de_1^{-1}) \oplus \cdots
\oplus R(\de_n^{-1}),$ where we are considering $N$ with the
homogeneous basis $\{u_1 e_1 , \ldots, u_n e_n\}$ with $\deg (u_i
e_i)= \ga_i \de_i $.
Then with this basis,
$$
\End_R(N) \conggr M_n (R)(\ga_1 \de_1, \dots, \ga_n \de_n).
$$
Define a graded $R$-module isomorphism $\phi: M \ra N$ by $\phi
(e_i) = u_i^{-1} (u_i e_i)$. The required isomorphism follows in a
similar way to (1). Clearly, if $R$ is a graded division ring then
every non-zero homogeneous element is invertible, so any set of
$(\ga_1 , \ldots , \ga_n) \in \Ga_R^n$ can be chosen here.

\vspace{3pt}

(3) Since $\si \in Z(\Ga)$, it is clear that $M_n(R)( \de_1 , \ldots
, \de_n)_{\la} = M_n(R)( \de_1  \si, \ldots , \de_n  \si)_{\la}$ for
each $\la \in \Ga$, as required.
\end{proof}

We note that if $R$ is a $\Ga$-graded ring, then for a $\Ga$-graded
right\label{pagerefforgrrightrmodule} $R$-module $M$, we have $M_\la
R_\ga \subseteq M_{\la \ga}$ for all $\la, \ga \in \Ga$. Then as for
graded left $R$-module homomorphisms, graded right $R$-module
homomorphisms between graded right $R$-modules may shift the
grading. We note that they are left shifted; that is, if $M,N$ are
graded right $R$-modules and $f \in \Hom_R (M,N)_{\de}$, then $f
(M_\ga) \subseteq N_{\de \ga}$ for all $\ga \in \Ga$. The
$\de$-shifted right $R$-module $(\de) M$ is defined as  $(\de) M =
\bigoplus_{\ga \in \Ga} \big((\de) M \big)_\ga$ where $\big( (\de)M
\big)_\ga = M_{\de \ga}$.

The following four results (Proposition~\ref{lemma1.15hungerford} to
Corollary~\ref{prop1.17part2hungerford}) are the graded versions of
some results on simple rings (see \cite[\S IX.1]{hungerford}). These
are required for the proof of Theorem~\ref{grmatrixisothm}.

\begin{prop} \label{lemma1.15hungerford}
Let $D$ be a graded division ring, let $V$ be a finite dimensional
graded module over $D$, and let $R= \End_D (V)$. If $A$ and $B$ are
graded right $R$-modules which are faithful and graded simple, then
$(\ga) A \conggr B$ as graded right $R$-modules for some $\ga \in
\Ga$.
\end{prop}

\begin{proof}
From Proposition~\ref{endgrisomatrixring}, $R \conggr M_n (D)(d)$
for some $(d) \in \Ga^n$. Then we will show that $M_n (D)(d)$
contains a minimal graded right ideal. For some $i$, consider $J_1 =
\{ e_{i,i} X : X \in M_n (D)(d)\}$, where $e_{i,i}$ is the
elementary matrix with $1$ in the $i,i$-entry. Then $J_1$ consists
of all matrices in $M_n (D)(d)$ with $j$-th row zero for $j \neq i$
and $J_1$ is a graded right ideal of $M_n(D)(d)$. If $J_1$ is not
minimal, there is a non-zero graded right ideal $J_2$ of $M_n
(D)(d)$, with $J_2 \subseteq J_1$. Since $J_2$ is a graded ideal, it
contains a non-zero homogeneous element. Then by using the
elementary matrices, we can show that $J_1= J_2$. So $R$ contains a
minimal graded right ideal, which we denote by $I$.


Since $A$ is faithful, its annihilator is zero, so there is a
homogeneous element $a \in A_{\ep}$ for some $\ep \in \Ga$, such
that $aI \neq 0$ (if not, then $aI =0$ for all $a \in A^h$, so $I
\subseteq \Ann (A)=0$, contradicting $I \neq 0$). Then $a I$ is a
graded submodule of $A$ as  it is generated by the homogeneous
elements $a \in A_\ep$ and all the $i \in I^h$. But $A$ is graded
simple, so $aI =A$. Define a map
\begin{align*}
\psi :I  &\lra (\ep)aI = (\ep)A\\
i &\lmps ai.
\end{align*}
This is a graded right $R$-module homomorphism, which is surjective.
Since $\ker (\psi)$ is a graded right ideal of $I$, and $I$ is
minimal, this implies the kernel is zero, so $\psi$ is a graded
isomorphism. Similarly, we have $I \conggr ({\ep'}) B$ for some
$\ep' \in \Ga$. So $(\ep)A \conggr ({\ep'})B$, which says $(\ga)A
\conggr B$ for some $\ga \in \Ga$.
\end{proof}

\begin{prop} \label{lemma1.16hungerford}
Let $D$ be a graded division ring, $V$ be a non-zero graded
$D$-module, and let $R= \End_D (V)$. If $g : V \ra V$ is a left
shifted homomorphism of additive groups such that $g f=fg$ for all
$f \in R$, then there is a homogeneous element $d \in D^h$ such that
$g(v) = dv$ for all $v \in V$.
\end{prop}

\begin{proof}
Let $u \in V^h \mi 0$. We will show $u$ and $g(u)$ are linearly
dependent over $D$. This is clear if $\dim_D (V) =1$. Suppose
$\dim_D (V) \geq 2$, and suppose they are linearly independent. Then
$\{u, g(u) \}$ can be extended to form a homogeneous basis of $V$.
We can define a map $f \in \End_D(V)$ such that $f(u) = 0$ and $f
(g(u)) = v \neq 0$ for some $v \in V$. We assumed $fg=gf$, so $f( g
(u)) = g (f(u)) = g(0) =0$, contradicting the choice of $f(g(u))$.
So $u$ and $g(u)$ are linearly dependent over $D$, and there is some
$d \in D^h$ with $g(u) = du$. Let $v \in V$, and let $h \in R$ with
$h(u)=v$. Then $g(v) = g(h(u) ) = h(g(u)) =  h(du) = d(h(u)) = dv$.
\end{proof}

\begin{prop} \label{prop1.17hungerford}
Let $D$, $D'$ be graded division rings and let $V$, $V'$ be graded
modules of finite dimension $n$, $n'$ over $D$, $D'$ respectively.
If $\End_{D} (V) \conggr \End_{D'} (V')$ as graded rings, then
$\dim_D (V) = \dim_{D'} (V')$ and $D \conggr D'$.
\end{prop}

\begin{proof}
Let $R= \End_D (V)$. Note that $V$ is a graded right $R$-module via
$v r = r(v)$ for all $v \in V$, $r \in R$. We will show that $V$ is
faithful and graded simple as a  graded right $R$-module. If $V r =
0$ for $r \in R$, then $r(v)=0$ for all $v \in V$, so $r=0$. To show
graded simple, let $M$ be a non-zero graded right $R$-submodule of
$V$, and let $m \in M^h \mi 0$. Extend $\{ m\}$ to form a
homogeneous basis of $V$. For some $v \in V \mi M$, define a map
$\theta_v :V \ra V$ by $\theta_v (m) = v$ and $\theta (w)=0$ for all
other elements $w$ in this basis of $V$. Then $\theta_v \in R$, and
$m \theta_v = v \notin M$, so $M$ is not closed under action of $R$.
This shows $V$ is graded simple.

Denote the given graded ring isomorphism by $\sigma : R \ra
\End_{D'} (V')$. Then using the same argument as above, we have that
$V'$ is a faithful and graded simple graded right $\End_{D'}
(V')$-module. Using the isomorphism $\si$, we can give $V'$ the
structure of a graded right $R$-module by defining $wr = w \si (r)$
for each $w \in V'$, $r \in R$. It follows that $V'$ is faithful
with respect to $R$. Any graded right $R$-submodule of $V'$ as also
closed under action of $\End_{D'} (V')$, so it is a graded right
$\End_{D'} (V')$-submodule of $V'$. This shows that $V'$ is graded
simple as a graded right $R$-module.

Applying Proposition~\ref{lemma1.15hungerford}, there is a graded
right $R$-module isomorphism $\phi : (\ga)V \ra V'$ for some $\ga
\in \Ga$. Then for $v \in V$, $f \in R$,
$$
(\phi \circ f)(v) = \phi (f (v)) = \phi (vf) = \phi (v) f = \phi (v)
\si (f) = \si (f) \big(\phi (v)\big),
$$
so that $\phi \circ f \circ \phi^{-1} = \si (f)$ as a homomorphism
of additive groups from $V'$ to $V'$. For $d \in D^h$, let $\al_d :V
\ra V;$ $x \mps dx$. This is a (left shifted) homomorphism of
additive groups. Clearly $\al_d =0$ if and only if $d=0$. Similarly
for $e \in {D'}^h$ we can define $\al_e$.

Let $f \in R$, $v \in V$ and $d \in D^h$. Then
\begin{align*}
(f \circ \al_d )(v)\, =\, f (dv) \, = \, d f(v)  \, = \, (\al_d
\circ f )(v),
\end{align*}
so $f \circ \al_d  = \al_d  \circ f$. Then using the above results
$(\phi \circ \al_d \circ \phi^{-1}) \circ (\si (f)) = (\si (f))
\circ (\phi \circ \al_d \circ \phi^{-1})$. Since $\si$ is surjective
and $\phi \circ \al_d \circ \phi^{-1}$ is a left shifted
homomorphism of additive groups from $V'$ to $V'$, we can apply
Proposition~\ref{lemma1.16hungerford}. There is a homogeneous
element $e \in {D'}^h$ such that $\phi \circ \al_d \circ \phi^{-1}
(w) = ew$ for all $w \in V'$.

Define a map
\begin{align*}
\tau : D^h & \lra {D'}^h \\
d & \lmps e
\end{align*}
and extend it linearly to cover all of $D$. It is routine to show
that $\tau$ is a graded ring homomorphism. For example, we will show
that, for some $\ga \in \Ga$, if $d_1, d_2 \in D_{\ga}$ then $\tau
(d_1 + d_2) = \tau (d_1) + \tau (d_2)$. The others parts are
similar. Let $e_1 , e_2, e_3 \in {D'}^h$ with $\phi \circ \al_{d_1}
\circ \phi^{-1} = \al_{e_1}$, $\phi \circ \al_{d_2} \circ \phi^{-1}
= \al_{e_2}$ and $\phi \circ \al_{d_1+d_2} \circ \phi^{-1} =
\al_{e_3}$. Then for $w \in V'$,
\begin{align*}
e_3 w = \phi \circ \al_{d_1+d_2} \circ \phi^{-1} (w) & = \phi \big(
(d_1 +d_2) \phi^{-1} (w) \big) \\
&= \phi \big(d_1 \phi^{-1} (w)\big) +\phi \big(d_2 \phi^{-1} (w)
\big) \\
&=  \phi \circ \al_{d_1} \circ \phi^{-1} (w) +  \phi \circ \al_{d_2}
\circ \phi^{-1} (w) \\
&= e_1 w + e_2 w = (e_1 +e_2)w.
\end{align*}

Then $\tau$ is injective, since if $\tau (d) = 0$, then $\phi \circ
\al_d \circ \phi^{-1} (w)=0$ for all $w \in V'$. This implies $\al_d
=0$, so $d=0$. Reverse the roles of $D$ and $D'$ and replace $\phi,
\si$ by $\phi^{-1} , \si^{-1}$. For each $k \in {D'}^h$, there is $d
\in D^h$ such that $\phi^{-1} \circ \al_k \circ \phi = \al_d$. From
above there is $\tau (d) \in {D'}^h$ such that $\phi \circ \al_d
\circ \phi^{-1} = \al_{\tau(d)}$. Combining these gives $\al_k =
\al_{\tau (d)}$. Since $k$ and $\tau (d)$ are homogeneous of the
same degree in $D'$, it follows that $k= \tau (d)$, so $\tau$ is
surjective.

Let $d \in D^h$ and $v \in V$. Then $\phi (dv) = \phi \circ \al_d
(v) = \al_{\tau(d)} \circ \phi (v) =\tau (d) \phi (v)$. Then using
this we can show that $\{u_1, \ldots , u_k\}$ are linearly
independent in $V$ if and only if $\{\phi (u_1) , \ldots , \phi
(u_k) \}$ are linearly independent in $V'$. It follows that the
former set spans $V$ if and only if the latter set spans $V'$,
proving $\dim_D (V) = \dim_{D'} (V')$.
\end{proof}

\begin{cor}\label{prop1.17part2hungerford}
Let $D$, $D'$ be graded division rings, and let $(d) \in \Ga^n$,
$(d') \in \Ga^{n'}$. If $M_n(D)(d) \conggr M_{n'}(D')(d')$ as graded
rings, then $n = n'$ and $D \conggr D'$.
\end{cor}

\begin{proof}
As in the proof of Proposition~\ref{grmatrixringsproperties}, we can
choose a graded $D$-module $V$ such that $M_n(D)(d) \conggr \End_D
(V)$, and a graded $D'$-module $V'$ such that $M_{n'}(D')(d')
\conggr \End_{D'}(V')$. Then by
Proposition~\ref{prop1.17hungerford}, we have that $n=n'$ and $D
\conggr D'$.
\end{proof}


Let $D$ be a $\Ga$-graded division ring and let  $(\lambda) =
(\lambda_1, \ldots , \lambda_n) \in \Ga^n$. Consider the partition
of $\Ga$ into right cosets of $\Ga_D$. For distinct elements of
$(\lambda)$, if the right cosets $\Ga_D  \lambda_1 , \ldots , \Ga_D
 \lambda_n$ are not all distinct, then there is a first $\Ga_D
\lambda_i$ such that $\Ga_D  \lambda_i = \Ga_D  \lambda_j$ for some
$j <i$. In $(\lambda)$, we will replace $\lambda_i$ by $\lambda_j$,
so that now $(\lambda) = (\lambda_1 , \ldots , \lambda_{i-1} ,
\lambda_j , \lambda_{i+1}, \ldots , \lambda_n)$. If the right cosets
$$
\Ga_D  \lambda_1 , \ldots , \Ga_D  \lambda_{i -1},  \Ga_D \lambda_{
i +1} , \ldots , \Ga_D  \lambda_n
$$
are still not all distinct, repeat the above process. Continue
repeating this process until all the right cosets of distinct
elements of $(\lambda)$ are distinct. Let $k$ denote the number of
distinct right cosets (which is, after the above process, also the
number of distinct elements in $(\lambda)$). Let $\ep_1 =
\lambda_1$, let $\ep_2$ be the second distinct element of
$(\lambda)$, and so on until $\ep_k$ is the $k$-th distinct element
of $(\lambda)$. For each $\ep_l$, let $r_l$ be the number of
$\lambda_i$ in $(\lambda_1 , \ldots , \lambda_n)$ with $\Ga_D
\lambda_i = \Ga_D  \ep_l$. Using
Proposition~\ref{grmatrixringsproperties} we get
\begin{equation}\label{wdeild}
M_n(D)(\lambda_1, \dots, \lambda_n) \conggr  M_n(D)(\ep_1, \ldots,
\ep_1, \ep_2, \ldots, \ep_2, \ldots, \ep_k, \ldots, \ep_k),
\end{equation}
where each $\ep_l$ occurs $r_l$ times.


Also note that for $\Ga$-graded $D$-modules $M$ and $N$, by
Proposition~\ref{grdimensionprop}, $\dim_D (M \oplus N) = \dim_D (M)
+ \dim_D \big( (M \oplus N) / M \big).$ By the first isomorphism
theorem, $(M \oplus N)/M \conggr N$, so $\dim_D (M \oplus N ) =
\dim_D (M) + \dim_D (N)$.

The following theorem  extends \cite[Thm.~2.1]{can2} from trivially
graded fields to graded division rings.



\begin{thm}\label{grmatrixisothm}
Let $D$ be a $\Ga$-graded division ring, let $\lambda_i,\ga_j \in
\Ga$ for $1\leq i \leq n$, $1\leq j \leq m$ and let $\Omega = \{
\lambda_1 , \ldots , \lambda_n , \ga_1 , \ldots , \ga_m \}$. If the
elements of $\Omega$ mutually commute and if
\begin{equation}\label{eqngrmatrixisothm}
M_n(D)(\lambda_1, \ldots, \lambda_n) \conggr M_m(D)(\ga_1, \ldots,
\ga_m),
\end{equation}
then we have $n=m$ and $\ga_i= \tau_i  \lambda_{\pi(i)} \sigma$ for
each $i$, where $\tau_i \in \Ga_D$, $\pi \in S_n$ is a permutation
and $\sigma \in Z_\Ga (\Omega )$ is fixed.
\end{thm}

\begin{proof}
It follows from Corollary~\ref{prop1.17part2hungerford} that $n=m$.
As in (\ref{wdeild}), we can find $(\epsilon) = (\ep_1, \ldots,
\ep_1, \ep_2, \ldots, \ep_2, \ldots, \ep_k, \ldots, \ep_k)$ in
$\Ga^n$ such that $M_n(D)(\lambda_1, \ldots, \lambda_n)\cong_{\gr}
M_n(D)(\epsilon)$. Let $V = D(\ep_1^{-1}) \oplus \cdots \oplus
D(\ep_1^{-1}) \oplus \cdots \oplus D(\ep_k^{-1}) \oplus \cdots
\oplus D(\ep_k^{-1})$, and let $e_1, \ldots , e_n$ be the standard
homogeneous basis of $V$. Then $\deg (e_i ) = \ep_{s_i}$, where
$\ep_{s_i}$ is the $i$-th element in $(\epsilon)$. Define $E_{ij}
\in \End_D(V)$ by $E_{ij}(e_l)= \mbox{{\boldmath $\de$}}_{il} e_j$,
for $1 \leq i,j,l \leq n$. We observed above (before
Proposition~\ref{grmatrixringsproperties}) that each $E_{ij}$ is a
graded $D$-module homomorphism of degree $\ep_{s_i}^{-1} \ep_{s_j}$,
the set $\{ E_{ij} : 1 \leq i ,j \leq n \}$ forms a homogeneous
basis for $\End_D (V)$ and $\End_D(V) \cong_{\gr} M_n(D)(\epsilon)$,
where $E_{ij}$ corresponds to the matrix $e_{ij}$ in
$M_n(D)(\epsilon)$. For any $i,j,h,l$, we have $E_{ij} E_{hl} =
\mbox{{\boldmath $\de$}}_{il} E_{hj}$, so $\{E_{ii} : 1 \leq i \leq
n \}$ forms a complete system of orthogonal idempotents of $\End_D
(V)$.

Similarly, we can find $(\epsilon')= (\ep_1', \ldots, \ep_1',
\ep_2', \ldots, \ep_2', \ldots, \ep_{k'}', \ldots, \ep_{k'}') \in
\Ga^n$ such that $M_n(D)(\ga_1, \dots, \ga_n) \conggr
M_n(D)(\epsilon')$. Let
$$
W = D({\ep'_1}^{-1}) \oplus \cdots \oplus D({\ep'_1}^{-1}) \oplus
\cdots \oplus D({\ep'_{k'}}^{-1}) \oplus \cdots \oplus
D({\ep'_{k'}}^{-1})
$$
and let $e'_1, \ldots , e'_n$ be the standard homogeneous basis of
$W$ with $\deg(e'_i) = \ep'_{s_i}$. We have $\End_D(W) \conggr
M_n(D)(\epsilon')$, so the graded isomorphism
\eqref{eqngrmatrixisothm} provides  a graded ring isomorphism
$\theta: \End_D(V) \ra \End_D(W)$. Define $E_{ij}':=
\theta(E_{ij})$, for $1\leq i,j \leq n$, and let $E_{ii}'(W)= Q_i$
for each $i$. Since $\{E_{ii} : 1\leq i \leq n\}$ forms a complete
system of orthogonal idempotents, $\{E_{ii}' : 1\leq i \leq n\}$
also forms a complete system of orthogonal idempotents for $\End_D
(W)$. It follows that $W = \bigoplus_{1\leq i \leq n} Q_i$.

For any $i,j,h,l$, as above,  $E_{ij}'E_{hl}' = \mbox{{\boldmath
$\de$}}_{il} E_{hj}'$ so $E_{ij}' E_{ji}' = E_{jj}'$ and $E_{ii}'$
acts as the identity on $Q_i$. By restricting $E_{ij}'$ to $Q_i$,
these relations induce a graded $D$-module isomorphism $E_{ij}': Q_i
\ra Q_j$ of the same degree as $E_{ij}$, namely $\ep_{s_i}^{-1}
\ep_{s_j}$. So $Q_i \conggr Q_1( \ep_{s_i}^{-1} \ep_{1} )$ for any
$1\leq i \leq n$, and it follows that $W \conggr \bigoplus_{1\leq i
\leq n} Q_1( \ep_{s_i}^{-1} \ep_{1} )$. Using the observations
before the theorem regarding dimension count, it follows that
$\dim_D Q_1=1$. So $Q_1$ is generated by one homogeneous element,
say $q$, with $\deg(q) = \al$, and we have $Q_1 \conggr D(
\al^{-1})$. Now $\Ga_D \al =\Supp (D (\al^{-1})) = \Supp (Q_1)
\subseteq \Supp (W) = \bigcup_i (\Ga_D \ep'_i)$. But as the right
cosets of $\Ga_D$ in $\Ga$ are either disjoint or equal, this
implies $\Ga_D \al =  \Ga_D \ep'_j$ for some $j$. Then  $Q_1 \conggr
D( \al^{-1}) = D({\ep'_j}^{-1} \tau_j^{-1})$ for some $\tau_j \in
\Ga_D$. We can easily show that $D({\ep'_j}^{-1} \tau_j^{-1})
\conggr D({\ep'_j}^{-1})$, so $Q_1 \conggr D({\ep'_j}^{-1})$.

Thus $W \conggr \bigoplus_{1\leq i \leq n} D ( \ep_{s_i}^{-1}
\ep_{1} {\ep'_j}^{-1})$. Then $V = \bigoplus_{1\leq i \leq n}
D(\ep_{s_i}^{-1})$, so
$$
V( \ep_1 {\ep'_j}^{-1} ) \conggr \bigoplus_{1\leq i \leq n}
D(\ep_{s_i}^{-1}) ( \ep_1 {\ep'_j}^{-1}) =  \bigoplus_{1\leq i \leq
n} D(\ep_1 {\ep'_j}^{-1} \ep_{s_i}^{-1} ).
$$
Using the assumption that the elements of $\Omega$ mutually commute,
and since the elements of both $(\epsilon)$ and $(\epsilon')$ are
elements of $\Omega$, it follows that $\ep_1 {\ep'_j}^{-1}
\ep_{s_i}^{-1} = \ep_{s_i}^{-1} \ep_1 {\ep'_j}^{-1}$. So we have $W
\conggr V(\ep_1 {\ep'_j}^{-1} )$. Let $\sigma = {\ep'_j} \ep_1^{-1}$
and denote this graded $D$-module isomorphism by $\phi: W \ra
V(\sigma^{-1})$.

Then $\phi(e_i')= \sum_{1 \leq j \leq n} a_j e_j$, where $a_j \in
D^h$ and $e_j$ are homogeneous of degree $\ep_{s_j} \sigma$ in
$V(\sigma^{-1})$. Suppose that $\deg(e_j) \neq \deg(e_l)$ for some
$j,l$ with $a_j$, $a_l \neq 0$. Then since $\deg(\phi(e_i'))=
\ep_{s_i}' = \deg(a_j e_j) =\deg(a_l e_l)$ and $\ep_{s_j} \sigma
\neq \ep_{s_l} \sigma$, it follows that $\ep_{s_j} \ep_{s_l}^{-1}
\in \Ga_D$ which contradicts the fact that $\Ga_D \ep_{s_j}$ and
$\Ga_D \ep_{s_l}$ are distinct. So for all non-zero terms in the
sum, each $e_j$ has the same degree. Thus $\ep_{s_i}' =  \tau_j
\ep_{s_j} \sigma$ where $\tau_j = \deg(a_j) \in \Ga_D$.
Each $\ep_{s_j}$ was chosen as $\ep_{s_j} = \tau_l  \lambda_l$ for
some $l$ with $\tau_l \in \Ga_D$, and similarly each $\ep'_{s_i}$
was chosen as $\ep'_{s_i} = \tau_h  \ga_h$ for some $h$ with $\tau_h
\in \Ga_D$. Combining these gives that $\ga_i = \tau_i
\lambda_{\pi(i)} \sigma$ for $1\leq i \leq n$, where $\tau_i \in
\Ga_D$ and $\sigma \in Z_\Ga(\Omega)$.
\end{proof}

\begin{remark}
We note that for the converse of the above theorem, we need to
assume that $\sigma \in Z( \Ga)$. Suppose $n=m$ and $\ga_i= \tau_i
\lambda_{\pi(i)} \sigma$ for each $i$, where $\tau_i \in \Ga_D$,
$\pi \in S_n$ is a permutation and $\sigma \in Z(\Ga )$ is fixed.
Then it follows immediately from
Proposition~\ref{grmatrixringsproperties} that
\begin{equation*}
M_n(D)(\lambda_1, \ldots, \lambda_n) \conggr M_m(D)(\ga_1, \ldots,
\ga_m).
\end{equation*}
\end{remark}

Given a group $\Ga$ and a field $K$ (which is not graded), it is
also possible to define a grading on $M_n (K)$. Such a grading is
defined to be a \emph{good grading}\index{good grading} of $M_n (K)$
if the matrices $e_{ij}$ are homogeneous, where $e_{ij}$ is the
matrix with $1$ in the $ij$-position. These group gradings on matrix
rings have been studied by D\u{a}sc\u{a}lescu et al
\cite{dascalescu}. The following examples (from
\cite[Eg.~1.3]{dascalescu}) show two examples of $\mathbb
Z_2$-grading on $M_2 (K)$ for a field $K$, one of which is a good
grading; the other is not a good grading.

\begin{examples}
\begin{enumerate}
\item Let $R = M_2 (K)$ with the $\mathbb Z_2$ grading defined by
\begin{align*}
R_{[0]} = \left\{
\begin{pmatrix}
a & 0 \\ 0 & b
\end{pmatrix} : a,b \in K \right\}
\mathrm{ \;\;\; and \;\;\; } R_{[1]} = \left\{
\begin{pmatrix}
0 & c \\d & 0
\end{pmatrix} : c,d \in K \right\}
\end{align*}
Then $R$ is a graded ring with a good grading.

\item Let $S = M_2 (K)$ with the $\mathbb Z_2$ grading defined by
\begin{align*}
S_{[0]} = \left\{
\begin{pmatrix}
a & b-a \\ 0 & b
\end{pmatrix} : a,b \in K \right\}
\mathrm{ \;\;\; and \;\;\; } S_{[1]} = \left\{
\begin{pmatrix}
d & c \\d & -d
\end{pmatrix} : c,d \in K \right\}
\end{align*}
Then $S$ is a graded ring, such that the $\mathbb Z_2$-grading is
not a good grading, since $e_{11}$ is not homogeneous. Note that $R$
and $S$ are graded isomorphic, as the map $f$ is an isomorphism:
$$
f :R \lra S; \;\;\; 
\begin{pmatrix}
a &b \\ c&d
\end{pmatrix}
\lmps
\begin{pmatrix}
a+c & b+d -a-c \\ c &d-c
\end{pmatrix}
$$
So $S$ does not have a good grading, but it is isomorphic as a
graded ring to $R$ which has a good grading.

\end{enumerate}

\end{examples}

Let $\Ga$ be a group and let $R$ be a $\Ga$-graded ring. If $V$ is
an $n$-dimensional $\Ga$-graded $R$-module then, as above (see
\S\ref{sectiongradedmodules} and page \pageref{endisagrring}),
$\END_R (V) = \End_R (V)$ forms a $\Ga$-graded ring. Taking $R= K$
with the trivial $\Ga$-grading, then $V= \bigoplus_{\ga \in \Ga}
V_{\ga}$ is just a vector space with a $\Ga$-grading, for subspaces
$V_{\ga}$ of $V$.

Then $\End_K (V) \cong M_n (K)$, which induces the structure of a
$\Ga$-graded $K$-algebra on $M_n (K)$. Suppose $\{ b_1, \ldots , b_n
\}$ forms a homogeneous basis for $V$, with $\deg(b_i) = \ga_i$.
Then the matrix $e_{ij} \in M_n (K)$ corresponds to the map $E_{ij}
\in \End_K (V)$ defined by $E_{ij} (b_l ) = \mbox{{\boldmath
$\de$}}_{il} b_j$, where $\mbox{{\boldmath $\de$}}$ is the Kronecker
delta. As $E_{ij}$ is homogeneous of degree $\ga_i^{-1}  \ga_j$,
$e_{ij}$ is also homogeneous, so this construction gives a good
$\Ga$-grading on $M_n(K)$. This observation, combined with the
following proposition, shows that a grading is good if and only if
it can be obtained in this manner from an $n$-dimensional
$\Ga$-graded $K$-vector space.

\begin{prop}
Consider a good  $\Ga$-grading on $M_n (K)$. Then there exists a
$\Ga$-graded vector space $V$ such that the isomorphism $\End (V)
\cong M_n (K)$, with respect to  a homogeneous basis of $V$, is an
isomorphism of $\Ga$-graded algebras.
\end{prop}

\begin{proof}
See \cite[Prop.~1.2]{dascalescu}.
\end{proof}




\section{Graded projective modules} \label{sectiongrmorita}

Let $\Ga$ be a group which is not necessarily abelian and $R$ be a
$\Ga$-graded ring. Throughout this section, unless otherwise stated,
we will assume that all graded rings, graded modules and graded
algebras are also $\Ga$-graded.

We will consider the category $R \hygrmod$ which is defined as
follows: the objects are $\Ga$-graded (left) $R$-modules, and for
two objects $M$, $N$ in $R\hygrmod$, the morphisms are defined as
$$
\Hom_{R \hygrmod}(M,N) = \{ f \in \Hom_R (M,N) : f (M_\ga )
\subseteq N_\ga \textrm{ for all } \ga \in \Ga\}.
$$
A graded $R$-module $P= \bigoplus_{\ga \in \Ga} P_{\ga}$ is said to
be {\em graded projective}\index{graded module!graded projective}
(resp.\ {\em graded faithfully projective})\index{graded
module!graded faithfully projective} if $P$ is projective (resp.\
faithfully projective) as an $R$-module.  We use $\Pgr (R)$ to
denote the subcategory of $R \hygrmod$ consisting of graded finitely
generated projective modules over $R$.
%
%
Then Proposition~\ref{grprojectivethm} shows some equivalent
characterisations of graded projective modules. Its proof requires
the following proposition, which is from
\cite[Prop.~2.3.1]{grrings}.

\begin{prop}\label{grrlinearmaps}
Let $L,M,N$ be graded $R$-modules, with $R$-linear maps
\begin{displaymath}
\xymatrix{ & \; M \ar[dr]^{g}  & \\
    L \ar[ur]^h \ar[rr]_{f} &&N }
\end{displaymath}
such that $f = g \circ h$ and $f$ is a morphism in the category $R
\hygrmod$. If $g$ (resp.\ $h$) is a morphism in $R \hygrmod$, then
there exists a morphism $h' : L \ra M$ (resp.\ $g' :M \ra N$) in $R
\hygrmod$ such that $f = g \circ h'$ (resp.\ $f = g' \circ h$).
\end{prop}

\begin{proof}
See \cite[Prop.~2.3.1]{grrings}.
\end{proof}

\begin{thm}
\label{grprojectivethm} Let $R$ be a $\Ga$-graded ring and let $P$
be a graded $R$-module. Then the following are equivalent:
\begin{enumerate}
\item $P$ is graded projective;

\item For each diagram of graded $R$-module homomorphisms
\begin{displaymath}
\xymatrix{ & P \ar[d]^{j}  & \\
     M \ar[r]_{g} & N \ar[r] & 0 }
\end{displaymath}
with $g$ surjective, there is a graded $R$-module homomorphism $h :
P \ra M$ with $g \circ h = j$;

\item $\Hom_{R \hygrmod}(P, -)$ is an exact functor in $R \hygrmod$;

\item Every short exact sequence of graded $R$-module homomorphisms
\begin{displaymath}
0 \lra L \stackrel{f}{\lra} M \stackrel{g}{\lra} P \lra 0
\end{displaymath}
splits via a graded map;

\item $P$ is graded isomorphic to a direct summand of a graded free
$R$-module $F$.

%
\end{enumerate}
\end{thm}

\begin{proof}
(1) $\Ra$ (2): Since $P$ is projective and $g$ is a surjective
$R$-module homomorphism, there is an $R$-module homomorphism $h:P
\ra M$ with $g \circ h = j$. By Proposition~\ref{grrlinearmaps},
there is a graded $R$-module homomorphism $h' :P \ra M$ with $g
\circ h' = j$.

\vspace{3pt}

(2) $\Ra$ (3): In exactly the same way as the non-graded setting
(see \cite[Thm.~IV.4.2]{hungerford}) we can show that $\Hom_{R
\hygrmod}(P, -)$ is left exact. Then it follows immediately from (2)
that it is right exact.

\vspace{3pt}

(3) $\Ra$ (4): This is immediate.

\vspace{3pt}

(4) $\Ra$ (5): Let $\{p_i\}_{i \in I}$ be a homogeneous generating
set for $P$, where $\deg(p_i)= \de_i$. Let $\bigoplus_{i \in I}
R(\de_i^{-1})$ be the graded free $R$-module with standard
homogeneous basis $\{e_i\}_{i \in I}$ where $\deg (e_i) = \de_i$.
Then there is an exact sequence
$$
0 \lra \ker (g) \stackrel{\subseteq}{\lra} \bigoplus_{i \in I}
R(\de_i^{-1}) \stackrel{g}{\lra} P \lra 0,
$$
since the map $g : \bigoplus_{i \in I} R(\de_i^{-1}) \ra P$; $e_i
\mps p_i$ is a surjective graded $R$-module homomorphism. By (4),
there is a graded $R$-module homomorphism $h :P \ra \bigoplus_{i \in
I} R(\de_i^{-1})$ such that $g \circ h = \id_P$.

Since the exact sequence is, in particular, a split exact sequence
of $R$-modules, we know from the non-graded setting
\cite[Prop.~2.5]{magurn} that there is an $R$-module isomorphism
\begin{align*}
\theta: P \oplus \ker (g) & \lra \bigoplus\nolimits_{i \in I}
R(\de_i^{-1})\\
(p , q) & \lmps h(p) + q.
\end{align*}
Clearly this map is also a graded $R$-module homomorphism, so $P
\oplus \ker (g) \conggr \bigoplus_{i \in I} R(\de_i^{-1})$.

\vspace{3pt}

(5) $\Ra$ (1): Graded free modules are free, so $P$ is isomorphic to
a direct summand of a free $R$-module. From the non-graded setting,
we know that $P$ is projective.
\end{proof}

Let $\Ga$ be an abelian group, $R$ be a commutative $\Ga$-graded
ring, and we define multiplication in $\End_R (A)$ to be $g \cdot f
= g \circ f$.

\begin{defin}\label{grazumayadefin}
A graded $R$-algebra $A$ is called a {\em graded Azumaya
algebra}\index{graded Azumaya algebra} if the following two
conditions hold:
\begin{enumerate}
\item $A$ is graded faithfully projective;

\item The natural map $\psi_A : A \otimes_R A^{\op} \ra \End_R(A)$
is a graded isomorphism.
\end{enumerate}
\end{defin}
We note that a graded $R$-algebra which is an Azumaya algebra (in
the non-graded sense) is also a graded Azumaya algebra, since it is
faithfully projective as an $R$-module, and the natural homomorphism
$A \otimes_R A^{\op} \ra \End_R (A)$ is clearly graded. So a graded
central simple algebra over a graded field (as in
Theorem~\ref{gcsaazumayaalgebra}) is in fact a graded Azumaya
algebra.

The following proposition proves, in the graded setting, a partial
result of Morita equivalence (only in one direction), which we will
use in the next chapter (see Proposition~\ref{grckidfunctorprop}).
For the following proposition, we require the group $\Ga$ to be an
abelian group, so it will be written additively. We observe that if
$\Ga$ is an abelian group, with $R$ a graded ring and $(d) = (\de_1
, \ldots , \de_n) \in \Ga^n$, then $R^n(d)$ is a graded $M_n(R)(d)
\mhyphen R$-bimodule and $R^n(-d)$ is a graded $R \mhyphen
M_n(R)(d)$-bimodule. By \cite[p.~30]{grrings}, we can define grading
on the tensor product of two graded modules in a similar way to that
of graded algebras (as we defined on
page~\pageref{pagerefforgradingontensorproduct}).

\begin{prop}[Morita Equivalence in the graded setting]
\label{grmorita}\index{graded Morita equivalence} Let $\Ga$ be an
abelian group, $R$ be a graded ring and let $(d) = (\de_1 , \ldots ,
\de_n) \in \Ga^n$. Then the functors
\begin{align*}
\psi : \Pgr( M_n(R)(d) ) & \lra \Pgr( R) \\
P & \longmapsto R^n(-d)\otimes_{M_n(R)(d)}P \\
\textrm{ and \;\;\;\; } \va : \Pgr( R) & \lra \Pgr ( M_n(R)(d)) \\
Q & \longmapsto R^n(d)\otimes_R Q
\end{align*}
form equivalences of categories.
\end{prop}

\begin{proof}
There are graded $R$-module homomorphisms
\begin{align*}
\theta : \; R^n(-d)\otimes_{M_n(R)(d)} R^n(d)  & \lra R \\
(a_1 , \ldots , a_n ) \otimes (b_1, \ldots , b_n) & \longmapsto  a_1
b_1 + \cdots + a_n b_n ; \\
\textrm{ and \;\;\;\;}\sigma : \; R  & \lra
R^n(-d)\otimes_{M_n(R)(d)}
R^n(d)\\
a & \longmapsto (a , 0 , \ldots , 0) \otimes (1 , 0 , \ldots 0)
\end{align*}
with $\sigma \circ \theta = \id$ and $\theta \circ \sigma = \id$.
Further
\begin{align*}
\theta' : \; R^n(d) \otimes_{R} R^n(-d) & \lra M_n(R)(d)\\
\begin{pmatrix}
a_1 \\ \vdots \\ a_n
\end{pmatrix}
\otimes
\begin{pmatrix}
b_1 \\ \vdots \\ b_n
\end{pmatrix}
& \longmapsto
\begin{pmatrix}
a_{1} b_1 & \cdots & a_{1} b_n  \\
\vdots  &   & \vdots  \\
a_{n} b_1  & \cdots & a_{n} b_n
\end{pmatrix}
\end{align*}
\begin{align*}
\textrm{and \;\;\;\; } \sigma' : \; M_n(R)(d) & \lra R^n(d) \otimes_{R} R^n(-d) \\
(m_{i,j}) & \longmapsto
\begin{pmatrix}
m_{1,1} \\ m_{2,1} \\ \vdots \\ m_{n,1}
\end{pmatrix}
\otimes
\begin{pmatrix}
1\\ 0 \\ \vdots \\0
\end{pmatrix}
+ \cdots +
\begin{pmatrix}
m_{1,n} \\ \vdots \\ m_{n-1,n} \\ m_{n,n}
\end{pmatrix}
\otimes
\begin{pmatrix}
0\\\vdots \\0 \\ 1
\end{pmatrix}
\end{align*}
are graded $M_n(R)(d)$-module homomorphisms with $\sigma' \circ
\theta' = \id$ and $\theta' \circ \sigma' = \id$. So
$R^n(-d)\otimes_{M_n(R)(d)} R^n(d) \conggr R$ as graded $R \mhyphen
R$-bimodules and $R^n(d) \otimes_{R} R^n(-d) \conggr M_n(R)(d)$ as
graded $M_n(R)(d) \mhyphen M_n(R)(d)$-bimodules respectively.

Then for $P \in \Pgr( M_n(R)(d) )$, $R^n(d) \otimes_{R} R^n(-d)
\otimes_{M_n(R)(d)} P \conggr P$. Suppose $f : P \ra P'$ is a graded
$M_n(R)(d)$-module homomorphism. Then there is a commutative diagram
\begin{displaymath}
\xymatrix{R^n(d) \otimes_{R} R^n(-d) \otimes_{M_n(R)(d)} P
\ar[d]_{\id \otimes f} \ar[rr] && P \ar[d]^{f} \\
R^n(d) \otimes_{R} R^n(-d) \otimes_{M_n(R)(d)} P' \ar[rr] && P' }
\end{displaymath}
For $Q \in \Pgr (R)$, $R^n(-d) \otimes_{M_n(R)(d)} R^n(d) \otimes_R
Q \conggr Q$. If $g :Q \ra Q'$ is a graded $R$-module homomorphism,
then there is a commutative diagram
\begin{displaymath}
\xymatrix{R^n(-d) \otimes_{M_n(R)(d)} R^n(d) \otimes_R Q
\ar[d]_{\id \otimes g} \ar[rr] && Q \ar[d]^{g} \\
R^n(-d) \otimes_{M_n(R)(d)} R^n(d) \otimes_R Q' \ar[rr] && Q' }
\end{displaymath}
This shows that there are natural equivalences from $\psi \circ \va$
to the identity functor and from $\va \circ \psi$ to the identity.
Hence $\psi$ and $\va$ are mutually inverse equivalences of
categories.
\end{proof}



\chapter{Graded {\itshape K}-Theory of Azumaya Algebras} \label{chgradedktheoryofazumayaalgebras}

In Corollary~\ref{gcsacor}, we saw that the $K$-theory of a graded
central simple algebra, graded by an abelian group, is very close to
the $K$-theory of its centre, where this follows immediately from
the corresponding result in the non-graded setting (see
Theorem~\ref{azumayafreethm}). Note that for the $K$-theory of a
graded central simple algebra $A$, we are considering $K_i(A)  = K_i
(\pr (A))$, where $\pr (A)$ denotes the category of finitely
generated projective $A$-modules. But in the graded setting, there
is also the category of graded finitely generated projective modules
over a given graded ring, which we will consider in
Section~\ref{sectiongrdfunctors}.

In Section~\ref{sectiongrk0defin}, we consider the functor $K_0$ in
the graded setting. For a graded ring $R$, we set $K_0^{\gr} (R) =
K_0 (\Pgr (R))$, where $\mathcal P \mathrm{gr}(R)$ is the category
of graded finitely generated projective $R$-modules. We show that
this definition is equivalent to defining $K_0^\gr (R)$ to be the
group completion of $\Proj^{\, \gr} (R)$, where $\Proj^{\, \gr} (R)$
denotes the monoid of isomorphism classes of graded finitely
generated projective modules over $R$.

We include a number of results involving $K_0$ of strongly graded
rings in Section~\ref{sectiongradedk0}. Then in
Section~\ref{sectionanexample}, we consider a specific example of a
graded Azumaya algebra to show that the graded $K$-theory of this
graded Azumaya algebra is not the same as the graded $K$-theory of
its centre (see Example~\ref{eggrktheory}). So for this example, in
the setting of graded $K$-theory, Corollary~\ref{gcsacor} does not
hold. We also provide in Section~\ref{sectionanexample} an example
of a  graded Azumaya algebra such that its (non-graded) $K$-theory
does not coincide with its graded $K$-theory.

In Section~\ref{sectiongrdfunctors}, we study the graded $K$-theory
of graded Azumaya algebras. We introduce an abstract functor called
a graded $\mathcal D$-functor defined on the category of graded
Azumaya algebras graded free over a fixed commutative graded ring
$R$ (Definition~\ref{gradeddfunctor}). As in
Section~\ref{sectiondfunctors}, this allows us to show that, for a
graded Azumaya algebra $A$ graded free over $R$ and subject to
certain conditions, we have a relation similar to \eqref{keqn} in
the graded setting (see Theorem~\ref{grazumayafreethm} and
\cite{hm2}). A corollary of this is that the isomorphism holds for
any Azumaya algebra free over its centre (see
Corollary~\ref{grazumayafreecor} and Theorem~\ref{azumayafreethm}).


\section{Graded {\itshape K}$\mathbf{_0}$}\label{sectiongrk0defin}

In Section~\ref{sectionlowerkgroups}, we defined $K_0$ of a ring $R$
to be the group completion of the monoid $\Proj (R)$ of isomorphism
classes of finitely generated projective $R$-modules. In
\cite[Ch.~3]{rosenberg}, there is another definition of $K_0$, which
defines $K_0$ for categories, rather than just for rings. Then using
this definition of $K_0$ for categories, $K_0$ of a ring $R$ can
also be defined to be $K_0 (\pr (R))$, which is shown in
\cite[Thm.~3.1.7]{rosenberg} to be equivalent to the definition of
$K_0$ given in Section~\ref{sectionlowerkgroups}.

In this section, we let $\Ga$ be a multiplicative group and $R$ be a
$\Ga$-graded ring. We begin this section by stating the definition
of a category with exact sequences. We then show that $\Pgr (R)$ is
a category with exact sequences, and we recall from
\cite[Defin.~3.1.6]{rosenberg} the definition of $K_0$ of a category
(see Definition~\ref{k0ofcategory}).

\begin{defin}
A \emph{category with exact sequences} is a full additive
subcategory $\mathcal C$ of an abelian category $\mathcal A$, with
the following properties:
\begin{enumerate}
\item $\mathcal C$ is closed under extensions; that is, if
$$
0 \lra P_1 \lra P \lra P_2 \lra 0
$$
is an exact sequence in $\mathcal A$  and $P_1 , P_2 \in \Obj
\mathcal C$, then $P \in \Obj \mathcal C$.

\item $\mathcal C$ has a small skeleton; that is, $\mathcal C$ has a
full subcategory $\mathcal C_0$ such that $\Obj \mathcal C_0$ is a
set, and for which the inclusion $\mathcal C_0 \hookrightarrow
\mathcal C$ is an equivalence.
\end{enumerate}
The exact sequences in such a category are defined to be the exact
sequences in the ambient category $\mathcal A$ involving only
objects (and morphisms) all chosen from $\mathcal C$.
\end{defin}

Let $R$ be a graded ring. We will show that $\Pgr (R)$ is a category
with exact sequences. Firstly note that $\Pgr (R)$ is a full
additive subcategory of $R \hygrmod$, where $R \hygrmod$ is an
abelian category (by \cite[\S 2.2]{grrings}). Take
$$
0 \lra P_1 \lra P \lra P_2 \lra 0
$$
to be an exact sequence in $R \hygrmod$. If $P_1 ,P_2 \in \Obj \Pgr
(R)$, then $P_1 \oplus Q_1 \conggr R^n (\de_1, \ldots , \de_n)$ and
$P_2 \oplus Q_2 \conggr R^m (\de_{n+1} , \ldots , \de_{n+m})$, for
some $\de_i \in \Ga$. So by Theorem~\ref{grprojectivethm}, $P
\conggr P_1 \oplus P_2$ from the exact sequence and
$$
P \oplus (Q_1 \oplus Q_2) \conggr (P_1 \oplus Q_1) \oplus (P_2
\oplus Q_2) \conggr R^{n+m} (\de_1 , \ldots , \de_{n+m});
$$
that is, $P \in \Obj \Pgr (R)$.

For the second property, take $\mathcal C_0$ to be the set of graded
direct summands of $\{ R^n(d) : n \in \mathbb N , (d) \in \Ga^n \}$.
We will observe here that this is a set.
%
%
%
Graded direct summands are, in particular, direct summands, and we
know from the non-graded setting that the collection of direct
summands of $\{R^n : n \in \mathbb N\}$ is a set. Taking the union
of these sets over all $(d) \in \Ga^n$, we still have a set since
the union of sets indexed by a set is still a set.

For the second part of property (2), we have an equivalence of
categories from $\mathcal C_0$ to $\Pgr (R)$. This  follows since
for any $P \in \Pgr (R)$, we know from Theorem~\ref{grprojectivethm}
that $P$ is graded isomorphic to a graded direct summand of $\{R^n
(d): n \in \mathbb N, (d) \in \Ga^n\}$. This gives a functor
$\mathcal F : \Pgr (R) \ra \mathcal C_0$. Consider the inclusion
functor $\mathcal J :\mathcal C_0 \ra \Pgr (R)$. Then it is routine
to check that the functors $\mathcal F$ and $\mathcal J$ form
mutually inverse equivalences of categories.

\begin{defin}\label{k0ofcategory}
Let $\mathcal C$ be a category with exact sequences with small
skeleton $\mathcal C_0$. Then $K_0 (\mathcal C)$ is defined to be
the free abelian group based on $\Obj \mathcal C_0$, modulo the
following relations:
\begin{enumerate}
\item $[P] = [P']$ if there is an isomorphism $P \cong P'$ in $\mathcal
C$.

\item $[P] =[P_1] + [P_2]$ if there is a short exact sequence $0 \ra
P_1 \ra P \ra P_2 \ra 0$ in $\mathcal C$.
\end{enumerate}
Here $[P]$ denotes the element of $K_0 (\mathcal C)$ corresponding
to $P \in \Obj \mathcal C_0$. We note that relation (1) is a special
case of (2) with $P_1 =0$. Since every $P \in \Obj \mathcal C$ is
isomorphic to an object of $\mathcal C_0$, the notation $[P]$ makes
sense for every object of $\mathcal C$.
\end{defin}

We will now observe that $\Proj^\gr (R)$, the isomorphism classes of
graded finitely generated projective modules, forms a monoid.
Firstly it is a set, by the above argument, since it is the set of
direct summands of $\{R^n (d): n \in \mathbb N, (d) \in \Ga^n\}$,
modulo the equivalence relation of graded isomorphism. We will show
the direct sum on $\Proj^\gr (R)$ is well defined. If $[P] = [P']$
and $[Q] = [Q']$, then $P \conggr P'$ and $Q \conggr Q'$. So we have
$P \oplus Q \conggr P' \oplus Q'$; that is, $[P] + [Q ] = [P'] +
[Q']$. Clearly the binary operation is commutative and associative,
and the identity element in $\Proj^\gr (R)$ is the isomorphism class
of the zero module.

For a graded ring $R$, we define
$$
K_0^{\gr} (R) = K_0 (\Pgr (R)),
$$
where $K_0 (\Pgr (R))$ is defined to be $K_0$ of the category $\Pgr
(R)$. We show in Theorem~\ref{grk0thm} that this definition of
$K_0^{\gr}$ is equivalent to defining $K_0^\gr (R)$ to be the group
completion of $\Proj^{\gr} (R)$ (see \cite[Thm.~1.1.3]{rosenberg}
for the group completion construction). The following proof follows
in a similar way to the equivalent result in the non-graded setting
(see \cite[Thm.~3.1.7]{rosenberg}).

\begin{thm} \label{grk0thm}
Let $R$ be a graded ring and let $\Pgr (R)$ be the category of
graded finitely generated projective modules over $R$. Then the
group completion of $\Proj^{\gr} (R)$ may be identified naturally
with $K_0 (\Pgr (R))$, where $K_0 (\Pgr (R))$ is defined in
Definition~\ref{k0ofcategory}.
\end{thm}

\begin{proof}
By definition, $K_0 (\Pgr (R))$ and the group completion of
$\Proj^{\gr} (R)$ are both defined to be abelian groups with one
generator for each isomorphism class of graded finitely generated
projective modules over $R$.

Addition in the latter is defined as $[P] + [Q] = [P \oplus Q]$. In
$K_0 (\Pgr (R))$, $[P] + [Q]$ is defined to be $[N]$ for a graded
finitely generated projective $R$-module $N$ if there is an exact
sequence
$$
0 \lra P {\lra} N {\lra} Q \lra 0
$$
in $\Pgr (R)$. If $N = P \oplus Q$ then there is clearly an exact
sequence
$$
0 \lra P {\lra} P \oplus Q {\lra} Q \lra 0
$$
in $\Pgr (R)$. Thus $[P \oplus Q] = [N]$, and the addition
operations in the two groups coincide.

We also need to check that both groups satisfy the same relations.
The group completion construction is the free abelian group based on
$\Proj^{\gr} (R)$ subject to the relations  $[P] = [P']$ if $P
\conggr P'$ and $[P] + [Q] = [P \oplus Q]$. We have observed above
that the second relation holds in $K_0 (\Pgr (R))$, and it is clear
that the first relation also holds.

Then $K_0 (\Pgr (R))$ is subject to the relations (1) and (2) as in
Definition~\ref{k0ofcategory}. It remains to check that the second
of these relations holds in the group completion construction.
Suppose
$$
0 \lra P_1 \stackrel{f}{\lra} P \stackrel{g}{\lra} P_2 \lra 0
$$
is an exact sequence in $\Pgr (R)$. By Theorem~\ref{grprojectivethm}
this is split exact since $P_2$ is graded projective. So there is a
graded $R$-module homomorphism $h : P_2 \ra P$ with $g \circ h =
\id_{P_2}$. Then as in the proof of Theorem~\ref{grprojectivethm},
there is a graded $R$-module isomorphism
\begin{align*}
P_1 \oplus P_2 & \lra P \\
(p,q) & \lmps f (p) + h(q).
\end{align*}
So with the group completion construction, $[P] = [P_1 \oplus P_2] =
[P_1] +[P_2]$ which shows that it satisfies the second relation.
\end{proof}

We also observe that by \cite[p.~291]{rosenberg}, for $i=0$,
Quillen's $Q$-construction of $K_0$ for a category (as in
Section~\ref{sectionhigherktheory}) coincides with
Definition~\ref{k0ofcategory} of $K_0$ of a category. We finish this
section by calculating $K_0^{\gr}$ of a trivially graded field.

\begin{prop} \label{k0groftriviallygradedfield}
Let $F$ be a field, $\Ga$ be a group and consider $F$ as a trivially
$\Ga$-graded field. Then $K_0^\gr(F) \cong \bigoplus_{\ga \in \Ga}
\mathbb Z_{\ga}$ where $\mathbb Z_{\ga} = \mathbb Z$ for each $\ga
\in \Ga$.
\end{prop}

\begin{proof}
Let $M$ be a graded finitely generated projective $F$-module. By
Theorem~\ref{gradedfree}, $M$ is graded free
%
%
%
so $M \conggr {F(\de_1)}^{r_1} \oplus \cdots \oplus
{F(\de_k)}^{r_k}$, where $r_i \in \mathbb N$ and the $\de_i \in \Ga$
are distinct. To show that $M$ is written uniquely in this way,
suppose $M \conggr {F(\al_1)}^{r'_1} \oplus \cdots \oplus
{F(\al_l)}^{r'_l}$ for some $r'_i \in \mathbb N$ and some distinct
$\al_i \in \Ga$. Consider the set $\{\ga_1, \ldots , \ga_n\} =
\{\de_1 , \ldots , \de_k \}\cup \{\al_1 , \ldots , \al_l\}$. By
rearranging the terms and adding zeros where required, we have
\begin{align*}
{F(\de_1)}^{r_1} \oplus \cdots \oplus {F(\de_k)}^{r_k} =
{F(\ga_1)}^{s_1} \oplus \cdots \oplus {F(\ga_n)}^{s_n}
\end{align*}
where $s_i \in \mathbb N$ and the $\ga_i$ are distinct.
%
%
Similarly, ${F(\al_1)}^{r'_1} \oplus \cdots \oplus
{F(\al_l)}^{r'_l}$  can be written as ${F(\ga_1)}^{s'_1} \oplus
\cdots \oplus {F(\ga_n)}^{s'_n}$.

We note that as $F_e$-modules
$$
\big( F(\ga_1)^{s_1} \oplus \cdots \oplus F(\ga_n)^{s_n}
\big)_{\ga_1^{-1}} \cong \big( F(\ga_1)^{s'_1} \oplus \cdots \oplus
F(\ga_n)^{s'_n} \big)_{\ga_1^{-1}}.
$$
If $s_1 = 0$, then on the left hand side of the isomorphism we have
zero, since ${F(\ga_i)}_{\ga_1^{-1}} = 0$ for all $i \neq 1$. So we
must also have zero on the right hand side of the isomorphism.
Therefore as ${F(\ga_1)}_{\ga_1^{-1}} = F_e =F$, we must have
$s'_1=0$. If $s_1 \neq 0$, then
on the left hand side of the isomorphism, we have ${F^{}}^{s_1}$.
Since it is an $F$-module isomorphism, we have the same on the right
hand side, and thus $s_1 = s'_1$.

Repeat the same argument for each $\ga_i \in \{\ga_1, \ldots ,
\ga_n\}$. This shows that for each $i$, we have $s_i = s'_i$. Thus
$M$ can be written uniquely as ${F(\de_1)}^{r_1} \oplus \cdots
\oplus {F(\de_k)}^{r_k}$, where $r_i \in \mathbb N$ and the $\de_i
\in \Ga$ are distinct.
This gives an isomorphism from $\Proj^{\,\gr} (F)$ to
$\bigoplus_{\ga \in \Ga} \mathbb N_\ga$ where $\mathbb N_\ga =
\mathbb N$ for each $\ga \in \Ga$.
%
%
%
As the group completion of $\mathbb N$ is $\mathbb Z$, it follows
that $K_0^\gr (F)$ is isomorphic to $\bigoplus_{\ga \in \Ga} \mathbb
Z_{\ga}$ where $\mathbb Z_{\ga} = \mathbb Z$ for each $\ga \in \Ga$.
\end{proof}


\section{Graded {\itshape K}$\mathbf{_0}$ of strongly graded rings}
\label{sectiongradedk0}

Throughout this section, we let $\Ga$ be a multiplicative group and
$R$ be a $\Ga$-graded ring. For any $R_e$-module $N$ and any $\ga
\in \Ga$, we identify the $R_e$-module $R_{\ga} \otimes_{R_e} N$
with its image in $R \otimes_{R_e} N$. Then $R \otimes_{R_e} N$ is a
$\Ga$-graded $R$-module, with $R \otimes_{R_e} N = \bigoplus_{\ga
\in \Ga} R_{\ga} \otimes_{R_e} N$. Consider the restriction functor
\begin{align*}
\mathcal{G} : R  \hygrmod & \lra R_e \Mod\\
M & \lmps M_e \\
\psi & \lmps \psi|_{M_e},
\end{align*}
and the induction functor defined by
\begin{align*}
\I : R_e \Mod & \lra R \hygrmod \\
N & \lmps R \otimes_{R_e} N \\
\phi & \lmps \id_R \otimes \, \phi.
\end{align*}
%
Proposition~\ref{propbeforedadesthm} and Theorem~\ref{dadesthm} are
from Dade \cite[p.~245]{dade} (see also \cite[Thm.~3.1.1]{grrings}).

\begin{prop} \label{propbeforedadesthm}
Let $R$ be a $\Ga$-graded ring. With $\G$ and $\I$ defined as above,
there is a natural equivalence of the composite functor $\G \circ \I
: R_e \Mod \ra R_e \Mod$ with the identity functor on $R_e \Mod$.
\end{prop}

\begin{proof}
Let $N \in R_e \Mod$. Then $\G \circ \I (N) = (R \otimes_{R_e} N)_e
= R_e \otimes_{R_e} N$. We know that for the ring $R_e$, the map
$\al: R_e \otimes_{R_e} N \ra N$; $r \otimes n \mps rn$ is an
isomorphism. For $\phi : N \ra N'$ in $R_e \Mod$, the following
diagram is clearly commutative
\begin{displaymath}
\xymatrix{ R_e \otimes_{R_e} N \ar[rr]^{\;\;\;\;\;\;\;\; \al}
\ar[d]_{\id_{R_e} \! \otimes \, \phi} && N \ar[d]^{\phi} \\
R_e \otimes_{R_e} N' \ar[rr]_{\;\;\;\;\;\;\;\; \al} && N' }
\end{displaymath}
So $\al$ is a natural equivalence from $\G \circ \I$ to the identity
functor.
\end{proof}



\begin{thm}[Dade's Theorem] \label{dadesthm}
Let $R$ be a $\Ga$-graded ring. If $R$ is strongly graded, then the
functors $\G$ and $\I$ defined above form mutually inverse
equivalences of categories.
\end{thm}

\begin{proof}
Let $M$ be a graded $R$-module. Then $\I \circ \G (M) = R
\otimes_{R_e} M_e$. We will show that the natural map $\beta: R
\otimes_{R_e} M_e \ra M;$ $r \otimes m \mps rm$ is a graded
$R$-module isomorphism. It is an $R$-module homomorphism using
properties of tensor products, and is clearly graded. Since $R$ is
strongly graded, it follows that for all $\ga, \de \in \Ga$,
$$
M_{\ga \de} = R_e M_{\ga \de} = R_{\ga} R_{\ga^{-1}} M_{\ga \de}
\subseteq R_{\ga} M_{\de} \subseteq M_{\ga \de}
$$
so we have $R_{\ga} M_{\de} = M_{\ga \de}$. We note that $\beta
(R_{\ga} \otimes_{R_e} M_e) = R_{\ga} M_e = M_{\ga}$, so $\beta$ is
surjective.

Let $N= \ker (\beta)$, which is a graded $R$-submodule of $R
\otimes_{R_e} M_e$, so $N_e = N \cap (R_e \otimes_{R_e} M_e$). Now
$N_e = \ker (\al)$, where $\al :R_e \otimes_{R_e} M_e \ra M_e$ is
the canonical isomorphism, so $N_e = 0$. Since $N$ is a graded
$R$-module, as above we have  $N_\ga = R_\ga N_e = 0$ for all $\ga
\in \Ga$. It follows that $\beta$ is injective.

Let $\psi : M \ra M'$ in $R \hygrmod$. Then the following diagram
commutes
\begin{displaymath}
\xymatrix{ R \otimes_{R_e} M_e \ar[rr]^{\;\;\;\;\;\;\;\; \beta}
\ar[d]_{\id_{R} \! \otimes \, \psi|_{_{M_e}}} && M  \ar[d]^{\psi} \\
R \otimes_{R_e} M'_e \ar[rr]_{\;\;\;\;\;\;\;\; \beta} &&  M' }
\end{displaymath}
so $\beta$ is a natural equivalence from $\I \circ \G$ to the
identity functor, completing the proof.
\end{proof}

\begin{prop}\label{rgafgprojectiveoverr0}
Let $R$ be a $\Ga$-graded ring. If $R$ is strongly graded, then for
each $\ga \in \Ga$, $R_\ga$ is a finitely generated projective left
(or right) $R_e$-module.
\end{prop}

\begin{proof}
See \cite[Cor.~2.16.10]{dimensionsofringth}.
\end{proof}

We now show that the functors $\G$ and $\I$, when restricted to the
categories of finitely generated projective modules, still form an
equivalence of categories.

\begin{cor} \label{corafterdadesthm}
Let $R$ be a $\Ga$-graded ring. If $R$ is strongly graded, then the
functors
\begin{align*}
\mathcal{G} : \Pgr( R ) & \lra \pr( R_e) \text{ \;\; and \;\; } \I :
\pr( R_e ) \lra \Pgr( R )
\end{align*}
form mutually inverse equivalences of categories.
\end{cor}

\begin{proof}
If $A \in \pr (R_e)$, then $A \oplus B \cong R_e^n$ for some
$R_e$-module $B$. Then $\I (A \oplus B) \cong \I (R_e^n)$, so $\I
(A) \oplus \I (B) \cong R^n$. This shows $\I (A)$ is finitely
generated and projective as an $R$-module, and we know $\I (A) \in R
\hygrmod$, so $\I (A) \in \Pgr (R)$.

If $M \in \Pgr (R)$, then $M \oplus N \conggr R^n (d)$ for some
graded $R$-module $N$ and some $(d) = (\de_1, \ldots , \de_n) \in
\Ga^n$. Then $\G (M \oplus N) \cong  \G (R^n (d))$, so $\G (M)
\oplus \G (N) \cong R_{\de_1} \oplus \cdots \oplus R_{\de_n}$.
Since, by Proposition~\ref{rgafgprojectiveoverr0}, each $R_{\de_i}$
is a finitely generated projective module over $R_e$, we have that
$\G (M)$ is also finitely generated and projective as an
$R_e$-module. The result now follows from Theorem~\ref{dadesthm}.
\end{proof}

We note that $\I$ and $\G$ as in Corollary~\ref{corafterdadesthm}
are exact functors between exact categories. This follows since if
$0 \ra L \ra M \ra N \ra 0$ is an exact sequence in $\Pgr (R)$, then
as $N$ is graded projective, the exact sequence splits. So $M
\conggr L \oplus N$ and $\G (M) \cong \G(L \oplus N) \cong \G(L)
\oplus \G(N)$. Thus $0 \ra \G (L) \ra \G (M) \ra \G (N) \ra 0$ is a
split exact sequence in $\pr (R_e)$.

Suppose $0 \ra A \ra B \ra C \ra 0$ is an exact sequence in $\pr
(R_e)$. Then the exact sequence splits, so $B \cong A \oplus C$ and
$\I (B) \conggr \I(A \oplus C) \conggr \I(A) \oplus \I(C)$. Define
$R$-module homomorphisms
\begin{align*}
\pi : \I (A) \oplus \I (C) & \lra \I (C) &
\text{\!\!\!\!\!\!\!\!\!\!\!\!\!\!\!\!\!\!\!\! and \;\;\;\;\;\;\;}
\imath :\I
(A) & \lra \I (A) \oplus \I (C) \\
(a,c) & \lmps c & a & \lmps (a,0).
\end{align*}
From the non-graded setting \cite[Prop.~2.7]{magurn},
$$
\xymatrix{ 0 \ar[r] & \I (A) \ar[r]^{ \theta^{-1} \circ \imath} & \I
(B) \ar[r]^{\pi \circ \theta} & \I (C) \ar[r] & 0}
$$
is an exact sequence of $R$-modules, where $\theta : \I (B) \ra
\I(A) \oplus \I(C)$ is the $R$-module homomorphism as above. Then as
the maps $\theta , \imath , \pi$ are graded $R$-module
homomorphisms, the exact sequence is an exact sequence in $\Pgr
(R)$.



\begin{prop} \label{grk0ofstronglygrring}
Let $R$ be a $\Ga$-graded ring. If $R$ is strongly graded, then
$K_0^\gr (R) \cong K_0 (R_e)$.
\end{prop}

\begin{proof}
Since $R$ is strongly graded, we can apply
Corollary~\ref{corafterdadesthm}, which says that the category of
graded finitely generated projective modules over $R$ is equivalent
to the category of finitely generated projective modules over $R_e$.
By Theorem~\ref{kifunctorthm}, each $K_i$ is a functor from the
category of exact categories with exact functors to the category of
abelian groups. We observed above that $\I$ and $\G$ are exact
functors between exact categories, so this implies $K_0 \big(\Pgr
(R)\big) \cong K_0 \big(\pr (R_e)\big)$ as abelian groups. That is,
using the previous notation, $K_0^\gr (R) \cong K_0 (R_e)$.
\end{proof}



\section{Examples}\label{sectionanexample}

In this section, we consider a specific example of a graded Azumaya
algebra $A$\label{labelofreftoexample} (see
Example~\ref{eggrktheory}). We show that $K_0^\gr (A) \otimes
\mathbb Z [1/n]$ is not isomorphic to $K_0^{\gr} (Z(A)) \otimes
\mathbb Z [1/n]$ for this graded Azumaya algebra $A$. This leads to
the following question, which we will partially answer in
Section~\ref{sectiongrdfunctors} (see
Theorem~\ref{grazumayafreethm}).

\begin{ques}
Let $\Ga$ be an abelian group, $R$ be a commutative $\Ga$-graded
ring, and $A$ be a graded Azumaya algebra over its centre $R$ of
rank $n$. When do we have
$$
K_0^{\gr}(A) \otimes \mathbb Z[1/n] \cong K_0^{\gr}(R) \otimes
\mathbb Z[1/n]?
$$
\end{ques}
We now explain the example mentioned above.

\begin{example}\label{eggrktheory}
Consider the quaternion algebra $\mathbb H = \mathbb R \oplus
\mathbb R i \oplus \mathbb R j \oplus \mathbb R k$. In
Example~\ref{azumayaexample3}, we showed that $\H$ is an Azumaya
algebra over $\mathbb R$. Then from
Example~\ref{egofgrdivisionrings}(2), $\mathbb H$ is a $\mathbb Z_2
\times \mathbb Z_2$-graded division ring, so it is in fact a graded
Azumaya algebra, which is strongly $\mathbb Z_2 \times \mathbb
Z_2$-graded.

We can now use Proposition~\ref{grk0ofstronglygrring}. Here $\mathbb
H_0 = \mathbb R$, so $K_0^\gr (\mathbb H) \cong K_0 (\mathbb R)
\cong \mathbb Z$. The centre $Z(\mathbb H) = \mathbb R$ is a field
and is trivially graded by $\mathbb Z_2 \times \mathbb Z_2$. By
Proposition~\ref{k0groftriviallygradedfield}, $K_0^\gr \big( Z(
\mathbb H) \big) = K_0^\gr (\mathbb R) \cong \mathbb Z \oplus
\mathbb Z \oplus \mathbb Z\oplus \mathbb Z$. We can see that
$K_0^\gr (\H) \otimes \mathbb Z[1/2] \cong \mathbb Z \otimes \mathbb
Z [1/2]$, but $K_0^\gr \big( Z(\H) \big) \otimes \mathbb Z[1/2]
\cong ( \mathbb Z \oplus \mathbb Z \oplus \mathbb Z \oplus \mathbb Z
) \otimes \mathbb Z [1/2]$, so they are clearly not isomorphic.
\end{example}

We give here another example, which generalises the above example of
$\H$ as a $\Z_2 \times \Z_2$-graded ring.

\begin{example}
Let $F$ be a field, $\xi$ be a primitive $n$-th root of unity and
let $a$, $b \in F^*$. Let
$$
A = \bigoplus_{i=0}^{n-1} \bigoplus_{j=0}^{n-1} F x^i y^j
$$
be the $F$-algebra generated by the elements $x$ and $y$, which are
subject to the relations $x^n = a$, $y^n = b$ and $xy = \xi yx$. By
\cite[Thm.~11.1]{draxl}, $A$ is an $n^2$-dimensional central simple
algebra over $F$.

Further, we will show that $A$ forms a graded division ring. Clearly
$A$ can be written as a direct sum
$$
A= \bigoplus_{(i,j) \in \Z_n \+ \Z_n} A_{(i,j)}, \text{ \;\; where
 } A_{(i,j)} = F x^i y^j
$$
and each $A_{(i,j)}$ is an additive subgroup of $A$. Using the fact
that $\xi^{-kj} x^k y^j = y^j x^k$ for each $j,k$, with $0 \leq j,k
\leq n-1$, we can show that $A_{(i,j)} A_{(k,l)} \subseteq
A_{([i+k], [j+l])}$, for $i,j,k,l \in \Z_n$. A non-zero homogeneous
element $f x^i y^j \in A_{(i,j)}$ has an inverse $$f^{-1} a^{-1}
b^{-1} \xi^{-ij} x^{n-i} y^{n-j},$$ proving $A$ is a graded division
ring. Clearly the support of $A$ is $\Z_n \times \Z_n$, so $A$ is
strongly $\Z_n \times \Z_n$-graded. As for
Example~\ref{eggrktheory}, $K_0^\gr (A) \cong K_0(A_0) = K_0 (F)
\cong \mathbb Z$. The centre $F$ is trivially graded by $\mathbb Z_n
\times \mathbb Z_n$, so $K_0^\gr \big( Z(A) \big) \cong
\bigoplus_{i=1}^{n^2} \mathbb Z_i$ where $\Z_i = \Z$.

\end{example}

%
%

\begin{remark}
We saw in Example~\ref{egofgrdivisionrings}(2) that $\mathbb H$ can
also be considered as a $\mathbb Z_2$-graded division ring. So
$\mathbb H$ is also strongly $\mathbb Z_2$-graded, and $K_0^\gr
(\mathbb H) \cong K_0 (\mathbb H_0) = K_0 (\mathbb C) \cong \mathbb
Z$. Then $Z(\mathbb H)= \mathbb R$, which we can consider as a
trivially $\mathbb Z_2$-graded field, so by
Proposition~\ref{k0groftriviallygradedfield}, $K_0^\gr (\mathbb R) =
\mathbb Z \oplus \mathbb Z$. We note that for both grade groups,
$\Z_2$ and $\Z_2 \times \Z_2$, we have $K_0^\gr (\H) \cong \Z$, but
the $K_0^\gr (\R)$ are different. So the graded $K$-theory of a
graded ring depends not only on the ring, but also on its grade
group.
\end{remark}

We note that in the above example, the graded $K$-theory of $\mathbb
H$ is isomorphic to the usual $K$-theory of $\mathbb H$. This
follows since $\mathbb H$ is a division ring, so by
Example~\ref{k0examples}(1) we have $K_0(\mathbb H) \cong \mathbb
Z$, and we observed above that $K_0^\gr ( \mathbb H) \cong \mathbb
Z$. But it is not always the case that the graded $K$-theory and
usual $K$-theory coincide. For $\R$, we know that $K_0 (\R) \cong
\Z$. But in the above examples, when $\R$ was trivially graded by
$\Z_2$ (resp. $\Z_2 \times \Z_2$),  we had $K_0^\gr (\R) \cong \Z \+
\Z$ (resp. $K_0^\gr (\R) \cong \Z \+ \Z \+ \Z \+ \Z$). Below we will
give an example of a graded ring which is not trivially graded, and
its graded $K$-theory is not isomorphic to its usual $K$-theory.


\begin{example}\label{eggrktheorynotsameasktheory}
Let $K$ be a field and let $R= K[x^2 , x^{-2}]$. Then $R$ is a
$\mathbb Z$-graded field, where $R$ can be written as $R=
\bigoplus_{n \in \mathbb Z} R_n$, with $R_n = Kx^n$ if $n$ is even
and $R_n = 0$ if $n$ is odd.
Consider the shifted graded matrix ring $A = M_3 (R)(0,1,1)$, which
has support $\mathbb Z$. Then we will show that $A$ is a graded
central simple algebra over $R$, so by Theorem
\ref{gcsaazumayaalgebra}, $A$ is a graded Azumaya algebra over $R$.

It is clear that the centre of $A$ is $R$, and $A$ is finite
dimensional over $R$. We note that if $A$ has a non-zero homogeneous
two-sided ideal $J$, then $J$ is generated by homogeneous elements
(see Proposition~\ref{homogeneousidealhomogeneouselementsprop}).
Using the elementary matrices, we can show that $J=A$ (see
\cite[Ex.~III.2.9]{hungerford}), so $A$ is graded simple.

Further, we will show that $A$ is a strongly $\mathbb Z$-graded
ring. Using Proposition~\ref{crossedproductstronglygradedprop}, it
is sufficient to show that $I_3 \in A_{n} A_{-n}$ for all $n \in
\mathbb Z$. To show this, we note that
\begin{align*}
&I_3 =
\begin{pmatrix}
0 & 1 & 0 \\
0 & 0 & 0 \\
0 & 0 & 0
\end{pmatrix} \!\!\!
\begin{pmatrix}
0 & 0 & 0 \\
1 & 0 & 0 \\
0 & 0 & 0
\end{pmatrix}\!
+ \! \begin{pmatrix}
0 & 0 & 0 \\
x^2\! & 0 & 0 \\
0 & 0 & 0
\end{pmatrix} \!\!\!
\begin{pmatrix}
0 &\! x^{-2}\!\! & 0 \\
0 & 0 & 0 \\
0 & 0 & 0
\end{pmatrix} \!
+ \! \begin{pmatrix}
0 & 0 & 0 \\
0 & 0 & 0 \\
x^2 \! & 0 & 0
\end{pmatrix} \!\!\! \begin{pmatrix}
0 & 0 & \! x^{-2} \\
0 & 0 & 0\, \\
0 & 0 & 0\,
\end{pmatrix} \\
&\text{and } I_3 \! = \! 
\begin{pmatrix}
0 & \! x^{-2} \!\! & 0 \\
0 & 0 & 0 \\
0 & 0 & 0
\end{pmatrix} \!\!\!
\begin{pmatrix}
0 & 0 & 0 \\
x^2\! & 0 & 0 \\
0 & 0 & 0
\end{pmatrix} \!
+ \! \begin{pmatrix}
0 & 0 & 0 \\
1 & 0 & 0 \\
0 & 0 & 0
\end{pmatrix} \!\!\!
\begin{pmatrix}
0 & 1 & 0 \\
0 & 0 & 0 \\
0 & 0 & 0
\end{pmatrix} \!
+ \! \begin{pmatrix}
0 & 0 & 0 \\
0 & 0 & 0 \\
1 & 0 & 0
\end{pmatrix} \!\!\! \begin{pmatrix}
0 & 0 & 1 \\
0 & 0 & 0 \\
0 & 0 & 0
\end{pmatrix}
\end{align*}
So $I_3 \in A_1 A_{-1}$ and $I_3 \in A_{-1} A_1$, and the required
result follows by induction, showing $A$ is strongly graded.

As in the previous examples, using
Proposition~\ref{grk0ofstronglygrring} we have $K_0^\gr (A) \cong
K_0 (A_0)$. Here $R_0 = K$, so there is a ring isomorphism
$$
A_0 = \begin{pmatrix} R_0 & 0 & 0 \\
0 & R_0 & R_0  \\
0 & R_0 & R_0
\end{pmatrix}
\cong K \times M_2(K).
$$
Then
$$
K_0^\gr (A) \cong K_0 (A_0) \cong K_0 (K) \oplus K_0
\big(M_2(K)\big) \cong \mathbb Z \oplus \mathbb Z,
$$
since $K_0$ respects Cartesian products and Morita equivalence. Note
that for the usual $K$-theory of $A$, $K_0 \big( M_3 (R)(0,1,1)
\big) \cong K_0 (R) = K_0 \big( K[x^2 , x^{-2}] \big)$. Using the
Fundamental Theorem of Algebraic $K$-theory
\cite[Thm.~3.3.3]{rosenberg} (see also \cite[p.~484]{magurn}), $K_0
\big( K[x , x^{-1}] \big) \cong K_0 (K) \cong \mathbb Z$, and it
follows that $K_0 (A) \cong \mathbb Z$, since $K[x^2 , x^{-2}] \cong
K[x , x^{-1}]$ as rings. So the $K$-theory of $A$ is isomorphic to
one copy of $\mathbb Z$, which is not the same as the graded
$K$-theory of $A$.
\end{example}




\section{Graded $\mathcal D$-functors} \label{sectiongrdfunctors}

Throughout this section, we will assume that $\Ga$ is an abelian
group, $R$ is a fixed commutative $\Ga$-graded ring and all graded
rings, graded modules and graded algebras are also $\Ga$-graded. As
mentioned in Remark~\ref{remarkongrmatrixrings}, in this section we
will define multiplication in $\End_R (A)$ to be $g \cdot f = g
\circ f$. Let $\mathcal Ab$ be the category of abelian groups and
let $\text{Az}_\gr(R)$ denote the category of graded Azumaya
algebras graded free over $R$ with graded $R$-algebra homomorphisms.
We recall from Section~\ref{sectiongradedmodules} that
$$
\Ga^*_{M_{k}(R)} = \big \{ (d) \in \Ga^k : \GL_{k}(R)[d] \neq
\emptyset \big \},
$$
where, for $(d)=(\de_1, \ldots ,\de_k) \in \Ga^k$, $\GL_{k} (R)[d]$
consists of invertible $k \times k$ matrices with the $ij$-entry in
$R_{-\de_i}$ (see page~\pageref{pagerefmatrixwithshifting}).

\begin{defin} \label{gradeddfunctor}
An abstract functor $\mathcal F : \text{Az}_\gr(R) \ra \mathcal Ab$
is defined to be a \emph{graded $\mathcal D$-functor}\index{graded
d@graded $\mathcal D$-functor} if it satisfies the three properties
below:
\begin{description}
\item[(1)]  $\mathcal F(R)$ is the trivial group.

\item[(2)] For any graded Azumaya algebra $A$ graded free over $R$
and for any $(d) = (\de_1 , \ldots , \de_k) \in \Ga^*_{M_k (R)}$,
there is a homomorphism
$$
\rho: \mathcal F \big(M_k (A) (d)\big) \lra \mathcal F (A)
$$
such that the composition $\mathcal F (A) \ra \mathcal F \big(M_k
(A) (d)\big) \ra \mathcal F (A)$ is $\eta_{k}$, where
$\eta_k(x)=x^k$.

\item[(3)] With $\rho$ as in property (2), then $\ker(\rho)$ is $k$-torsion.
\end{description}
\end{defin}
Note that these properties are well-defined since both $R$ and
$M_k(A)(d)$ are graded Azumaya algebras  graded free over $R$. The
proof of the theorem below follows in a similar way to that of
Theorem~\ref{torsionthm}.


\begin{thm} \label{grtorsionthm}
Let $A$ be a graded Azumaya algebra which is graded free over its
centre $R$ of dimension $n$, such that $A$ has a homogeneous basis
with degrees $(\de_1, \ldots , \de_n)$ in $\Ga^*_{M_n(R)}$. Then
$\mathcal F (A)$ is $n^2$-torsion, where $\mathcal F$ is a graded
$\mathcal D$-functor.
\end{thm}

\begin{proof}
Let $\{a_1 , \ldots , a_n\}$ be a homogeneous basis for $A$ over
$R$, and let $(d)= (\deg (a_1) , \ldots ,\deg( a_n)) \in \Ga^*_{M_n
(R)}$. Since $R$ is a graded Azumaya algebra over itself, by (2) in
the definition of a graded $\mathcal D$-functor, there is a
homomorphism $\rho:\mathcal F \big( M_n(R)(d) \big) \ra \mathcal
F(R)$. But $\mathcal F(R)$ is trivial by property (1) and therefore
the kernel of $\rho$ is $\mathcal F \big( M_n(R)(d) \big)$ which is,
by (3), $n$-torsion. Further, the graded $R$-algebra isomorphism $A
\otimes_R A^{\text{op}} \conggr \End_R(A)$ from the definition of a
graded Azumaya algebra, combined with the graded isomorphism
$\End_R(A) \conggr M_n(R)(d)$, induces an isomorphism $\mathcal F(A
\otimes_R A^{\op}) \cong \mathcal F(M_n(R)(d))$. So $\mathcal F(A
\otimes_R A^{\op})$ is also $n$-torsion.

In the category $\text{Az}_\gr(R)$, the two graded $R$-algebra
homomorphisms $i:A \ra A \otimes_R A^{\text{op}}$ and $r: A^{\op}
\ra \End_R (A^{\op}) \ra M_n(R)(d)$ induce group homomorphisms $i :
\mathcal F(A) \ra \mathcal F(A \otimes_R A^{\op})$ and $r : \mathcal
F(A \otimes_R A^{\op}) \ra \mathcal F (A \otimes_R M_n(R)(d))$,
where $\mathcal F(A \otimes_R M_n(R)(d)) \cong \mathcal F
(M_n(A)(d))$. Consider the following diagram
\begin{equation*}
\begin{split}
\xymatrix{
\mathcal F(A) \; \ar[d]_i \ar@/^/[ddrr]^{\eta_n}&&  \\
\mathcal F(A \otimes_R A^{\text{op}}) \; \ar[d]_r &&  \\
\mathcal F(M_n(A)(d))\ar[rr]^-{\rho} && \mathcal F(A)}
\end{split}
\end{equation*}
which is commutative by property~(2). It follows that $\mathcal
F(A)$ is $n^2$-torsion.
\end{proof}

For a graded ring $A$, let $\mathcal P \mathrm{gr}(A)$ be the
category of graded finitely generated projective $A$-modules. Then
as for $K_0^\gr$ in Section~\ref{sectiongrk0defin}, we define
$$
K_i^{\gr} (A) = K_i (\Pgr (A))
$$
for $i \geq 0$, where $K_i (\Pgr (A))$ is the Quillen $K$-group of
the exact category $\Pgr (A)$ (see
Section~\ref{sectionhigherktheory}).

Let $A$ be a graded ring with graded centre $R$. Then the graded
ring homomorphism $R \rightarrow A$ induces an exact functor $A
\otimes_R - :\mathcal P \mathrm{gr}(R) \ra \Pgr (A)$, which in turn
induces a group homomorphism $K_i^{\gr} (R) \ra K_i^{\gr} (A)$. Then
we have an exact sequence
\begin{equation}\label{grexactseq}
1 \lra \ZK[i]^{\gr}(A) \lra K_i^{\gr}(R) \lra K_i^{\gr}(A) \lra
\CK[i]^{\gr}(A) \lra 1
\end{equation}
where $\ZK[i]^{\gr}(A)$ and $\CK[i]^{\gr}(A)$ are the kernel and
cokernel of the map $K_i^{\gr}(R) \rightarrow K_i^{\gr}(A)$
respectively. We will show that $\CK[i]^\gr$ can be regarded as the
following functor
\begin{align*}
\CK[i]^{\gr}: \text{Az}_{\gr}(R) &\lra  \mathcal Ab \\
A &\longmapsto \CK[i]^{\gr}(A).
\end{align*}

For a graded Azumaya algebra $A$ graded free over $R$, clearly
$\CK[i]^{\gr} (A) = \coker \big(K_i^{\gr} (R) \lra K_i^{\gr}
(A)\big)$ is an abelian group. Consider graded Azumaya algebras $A ,
A'$ graded free over $R$ and a graded $R$-algebra homomorphism $f :
A \ra A'$. Then there is an induced exact functor $A' \otimes_A - :
\Pgr (A) \ra \Pgr (A')$, which induces a group homomorphism $f_* :
K_i^{\gr}  (A) \ra K_i^{\gr}  (A')$. We have an exact functor
$$
A' \otimes_A (A \otimes_R -): \Pgr(R) \lra \Pgr (A) \lra \Pgr (A').
$$
As the map $f$ restricted to $R$ is the identity map, the induced
functor $\Pgr (R) \ra \Pgr (R)$ is also the identity. So it induces
the identity map $K_i^{\gr} (R) \ra K_i^\gr (R)$. We have an exact
functor
$$
A' \otimes_R - : \Pgr (R) \lra \Pgr (R) \lra \Pgr (A').
$$
Since these two functors from $\Pgr (R)$ to $\Pgr (A')$ are
isomorphic, they induce the same map on the level of the $K$-groups
by Theorem~\ref{kifunctorthm}.


Since $K_i$ is a functor from the category of exact categories to
the abelian groups, the following diagram is commutative
\begin{equation*}
\begin{split}
\xymatrix{ K_i^\gr(R)  \ar[r] \ar[d]_{\id} & K_i^\gr(A)  \ar[r]
\ar[d]^{f_*} & \CK[i]^\gr(A) \ar[d]^{\CK[i]^\gr(f)} \\
K_i^\gr (R)  \ar[r] & K_i^\gr (A')  \ar[r]  & \CK[i]^\gr (A') }
\end{split}
\end{equation*}
Then it is routine to check that $\CK[i]^\gr$ forms the required
functor. Similarly $\ZK[i]^{\gr}$ can be regarded as the following
functor
\begin{align*}
\ZK[i]^{\gr}: \text{Az}_{\gr}(R) &\lra  \mathcal Ab \\
A &\longmapsto \ZK[i]^{\gr}(A).
\end{align*}

\begin{prop} \label{grckidfunctorprop}
With $\CK[i]^{\gr}$ defined as above, $\CK[i]^{\gr}$ is a graded
$\mathcal D$-functor.
\end{prop}

\begin{proof}
Property (1) is clear, since $R$ is commutative so $K_i^{\gr} (Z(R))
\ra K_i^{\gr}(R)$ is the identity map. For property~(2), let $A$ be
a graded Azumaya algebra graded free over $R$ and let $(d) = (\de_1
, \ldots , \de_k) \in \Ga^*_{M_k (R)}$. Then there are functors:
\begin{align}\label{grfor}
\phi : \Pgr(A) & \longrightarrow \Pgr (M_k(A)(d)) \\
 X  & \longmapsto M_k(A)(d) \otimes_A X \notag
\end{align}
and
\begin{align}\label{grfor1}
\psi : \Pgr (M_k(A)(d))   & \longrightarrow \Pgr (A) \\
 Y  & \longmapsto A^k(-d) \otimes_{M_k(A)(d)} Y. \notag
\end{align}
The functor $\phi$ induces a homomorphism from $K_i^{\gr}(A)$ to
$K_i^{\gr}(M_k(A)(d))$. By the graded version of the Morita Theorems
\label{moritapageref} (see Proposition~\ref{grmorita}), the functor
$\psi$ establishes a natural equivalence of categories, so it
induces an isomorphism from $K_i^{\gr}(M_k(A)(d))$ to
$K_i^{\gr}(A)$.

For $X \in \Pgr(A)$, $\psi \circ \phi (X) \conggr X^k(-d)$. We will
use a similar argument to that of Proposition~\ref{rndconggrrna} to
show that $X^k(-d) \conggr X^k$. Since $(d) \in \Ga^*_{M_k (R)}$,
there is $r = (r_{ij}) \in \GL_{k}(R)[d]$. Then there is a graded
$A$-module homomorphism
\begin{align*}
\mal_r : \;\; X^k & \lra  X^k (-d) \\
(x_1 , \ldots , x_k ) & \lmps r (x_1 , \ldots , x_k ).
\end{align*}
Since $r$ is invertible there is an inverse matrix $t \in
\GL_{k}(R)$, and there is an $A$-module homomorphism $\mal_t : X^k
(-d) \ra X^k$ which is an inverse of $\mal_r$. So $\mal_r$ is a
graded $A$-module isomorphism. By the remarks after
Theorem~\ref{kirespectsdirectsums}, $K_i$ are functors which respect
direct sums, so this induces a multiplication by $k$ on the level of
the $K$-groups.

The exact functors \eqref{grfor} and \eqref{grfor1} induce the
following commutative diagram:
\begin{equation*}
\begin{split}
\xymatrix{ K_i^{\gr}(R)  \ar[r] \ar[d]^{=} & K_i^{\gr}(A)  \ar[r]
\ar[d]^\phi
& \CK[i]^{\gr}(A) \ar[r] \ar[d]&1  \\
K_i^{\gr}(R)  \ar[r] \ar[d]^{\eta_k} & K_i^{\gr}(M_k(A)(d))  \ar[r]
\ar[d]^\psi_\cong & \CK[i]^{\gr}(M_k(A)(d)) \ar[r] \ar[d]^\rho& 1 \\
K_i^{\gr}(R)  \ar[r]  & K_i^{\gr}(A)  \ar[r] & \CK[i]^{\gr}(A)
\ar[r]& 1 }
\end{split}
\end{equation*}
where composition of the columns are $\eta_k$, proving property (2).
A diagram chase verifies that property (3) also holds.
\end{proof}

A similar proof shows that $\ZK[i]^\gr$ is also a graded $\mathcal
D$-functor. We now show that for a graded Azumaya algebra which is
graded free over its centre, its graded $K$-theory is essentially
the same as the graded $K$-theory of its centre.

\begin{thm}\label{grazumayafreethm}
Let $A$ be a graded Azumaya algebra which is graded free over its
centre $R$ of dimension $n$, such that $A$ has a homogeneous basis
with degrees $(\de_1, \ldots , \de_n)$ in $\Ga^*_{M_n(R)}$. Then for
any $i \geq 0$,
$$
K_i^{\gr}(A) \otimes \mathbb Z[1/n] \cong K_i^{\gr}(R) \otimes
\mathbb Z[1/n].
$$
\end{thm}

\begin{proof}
Proposition~\ref{grckidfunctorprop} shows that $\CK[i]^{\gr}$ (and
in the same manner $\ZK[i]^{\gr}$) is a graded $\mathcal D$-functor,
and thus by Theorem~\ref{grtorsionthm} $\CK[i]^{\gr}(A)$ and
$\ZK[i]^{\gr}(A)$ are $n^2$-torsion abelian groups. Tensoring the
exact sequence (\ref{grexactseq}) by $\mathbb Z[1/n]$, since
$\CK[i]^{\gr}(A) \otimes \mathbb Z[1/n]$ and $\ZK[i]^{\gr}(A)
\otimes \mathbb Z[1/n]$ vanish, the result follows.
\end{proof}

It remains as a question when this result holds for a graded Azumaya
algebra of constant rank.

\begin{ques}
Let $\Ga$ be an abelian group, $R$ be a commutative $\Ga$-graded
ring, and $A$ be a graded Azumaya algebra over its centre $R$ of
rank $n$. When is it true that for any $i \geq 0$,
$$
K_i^{\gr}(A) \otimes \mathbb Z[1/n] \cong K_i^{\gr}(R) \otimes
\mathbb Z[1/n]?
$$
\end{ques}

\begin{remark}
Using Example~\ref{eggrktheory}, we remark here that the graded
Azumaya algebra $\mathbb H$ does not satisfy the conditions of
Theorem~\ref{grazumayafreethm}. Suppose $\mathbb H$ does satisfy
these conditions; that is, suppose there is a homogeneous basis
$\{a_1, \ldots, a_4\}$ for $\mathbb  H$ over $\mathbb R$, such that
the elements of the basis have degrees $(\de_1 , \ldots , \de_4)$ in
$\Ga^*_{M_4( R)}$. So there exists a matrix $r \in \GL_4 (\R)[d]$.
Since $\Supp (\R) = 0$, then as each row of $r$ must contain a
non-zero element, this implies $\de_i =0$ for each $i$. But this
would imply that the support of $\mathbb H$ is also $0$, which
clearly is a contradiction. So such a homogeneous basis for $\mathbb
H$ does not exist.
\end{remark}

\begin{cor} \label{grazumayafreecor}
Let $A$ be an Azumaya algebra free over its centre $R$ of dimension
$n$. Then for any $i \geq 0$,
$$
K_i(A) \otimes \mathbb Z[1/n] \cong K_i(R) \otimes \mathbb Z[1/n].
$$
\end{cor}

\begin{proof}
By taking $\Ga$ to be the trivial group, this follows immediately
from Theorem~\ref{grazumayafreethm}.
\end{proof}



\chapter{Additive Commutators}\label{chapteradditivecommutators}


In 1905 Wedderburn proved that a finite division ring is a field.
This is an example of a commutativity theorem; that is, it is a
theorem which states certain conditions under which a given ring is
commutative. Since then, Wedderburn's result has motivated many,
more general commutativity theorems (see \cite[\S 13]{lam1}). Both
additive and multiplicative commutators play an important role in
these theorems. In this chapter we consider additive commutators in
the setting of graded division algebras.

We begin Section \ref{sectionadditivecommutators} with two results
involving the support of a graded division ring. We then show that
some commutativity theorems involving additive commutators hold in
the graded setting. We give a counter-example to show that one such
commutativity theorem for multiplicative commutators does not hold.



In Section~\ref{sectiongradedsplitting}, we show that in the setting
of graded division algebras, the reduced trace exists and it is a
graded map. In Section~\ref{sectionsomenongradedresults}, we recall
some results from the non-graded setting which will be used in the
final section. We end the chapter by considering, in
Section~\ref{sectionquotientdivisionrings}, the quotient division
ring $QD$ of a graded division ring $D$. We show how the submodule
generated by the additive commutators in $QD$ relates to that of $D$
(see Corollary~\ref{d/d,dcongqd/qd,qd}).


\section{Homogeneous additive commutators} \label{sectionadditivecommutators}

Throughout this chapter, let $\Ga$ be an abelian group unless
otherwise stated. We recall from Section~\ref{sectiongradedrings}
that the support of a graded ring $R = \bigoplus_{ \ga \in \Ga}
R_{\ga}$ is defined to be the set
$$
\Supp (R) = \: \Ga_{R} = \big \{ \ga \in \Ga : R_{\ga} \neq \{0 \}
\big \}.
$$
It follows that a graded ring $R$ is zero if and only if $\Supp (R)
= \emptyset$.

Let $D$ be a $\Ga$-graded division ring with centre $F$. Since $\Ga$
is an abelian group, the centre of $D$ is a graded subring of $D$
(see Section~\ref{sectiongradedrings}). A \emph{homogeneous additive
commutator}\index{homogeneous additive commutator}\index{additive
commutator} of $D$ is defined to be an element of the form $ab - ba$
where $a,b \in D^{h}$. Throughout this chapter, we will use the
notation $[a,b] = ab-ba$ and $[D, D]$ is the graded submodule of $D$
generated by all homogeneous additive commutators of $D$.

We note that in this chapter we will consider the group $\Ga$ to be
an abelian group, unless stated otherwise. This is the natural
setting to consider homogeneous additive commutators. If $\Ga$ is
not abelian, then a given homogeneous additive commutator may not be
a homogeneous element in $D$.

A \emph{graded division algebra}\index{graded division algebra} $D$
is defined to be a graded division ring with centre $F$ such that
$[D:F] < \infty$. Note that since $F$ is a graded field, $D$ has a
finite homogeneous basis over $F$. A graded division algebra $D$
over its centre $F$ is said to be
\emph{unramified}\index{unramified} if $\Ga_D = \Ga_F$ and
\emph{totally ramified}\index{totally ramified} if $D_0 = F_0$. The
following lemma considers the support of $[D, D]$; the proof of part
(2) is due to Hazrat.



\begin{lemma}
Let $D = \bigoplus_{\ga \in \Ga} D_{\ga}$ be a graded division
algebra over its centre $F$.
\begin{enumerate}
\item If $D$ is totally ramified, then $\emptyset
\neq \Supp ([D,D]) \subsetneqq \Ga_D$.

\item If $D$ is not totally ramified, then $\Supp (D) = \Supp
([D,D])$.
\end{enumerate}
\end{lemma}

\begin{proof}
(1): Clearly $\emptyset \neq \Supp ([D,D]) \subseteq \Ga_D$. Since
$D_0 = F_0 = Z(D) \cap D_0$ we have $D_0 \subseteq Z(D)$. Suppose $0
\in \Supp([D,D])$. Then there is an element $\sum_i (x_i y_i - y_i
x_i) \in [D,D]$, with $\deg(x_i) + \deg(y_i) = 0$ for all $i$. If
$x_i y_i - y_i x_i = 0$ for all $i$, then clearly the sum is also
zero. Thus there are non-zero homogeneous elements $x \in D_{\ga}$,
$y \in D_{\de}$ with $0 \neq xy -yx \in D_0$ and $\ga + \de =0$.

Then $(xy-yx)y^{-1} \neq 0$, as $y^{-1} \in D_{-\de} \mi 0$ and
$xy-yx \in D_{0} \mi 0$, so their product is a non-zero homogeneous
element of degree $-\de$. Since
$$
(xy-yx)y^{-1} \;\; = \;\:\, xy y^{-1} - y x y^{-1} \;\; = \;\:\,
y^{-1} y x - y x y^{-1},
$$
we have $y^{-1} (yx) \neq (yx) y^{-1}$; that is $yx \notin Z(D)$.
Since $yx \in D_0$, this contradicts the fact that $D_0 = F_0$, so
$0 \notin \Supp([D,D])$.

\vspace{3pt}

(2): It is clear that $\Supp ([D,D]) \subseteq \Ga_D$. For the
reverse containment, for $\ga \in \Ga_D$ we will show that there is
an $x \in D_{\ga}$ which does not commute with some $y \in D_{\de}$
for some $\de \in \Ga_D$. Suppose not, then $D_{\ga} \subseteq
Z(D)$, so $D_{\ga} = F_{\ga}$. Let $x \in D_{\ga}$, $d \in D_0$, $y
\in D_{\de}$ be arbitrary non-zero elements. Then
\begin{align*}
x(d y) \;\, = \;\, (d y) x  \;\, =\,\;d( y x)\;\, =\;\, d(x y) \;\,
= \;\, (d x) y \;\, = \;\, y(d x)  \;\, =\,\;(y d) x\;\, =\;\,x (y
d).
\end{align*}
So for all $d \in D_0$, $y \in D_{\de}$ we have $x(d y) =x(y d)$.
Since $x$ is a non-zero homogeneous element, it is invertible. This
implies $d y = y d$, so $D_0 = F_0$ contradicting the fact that $D$
is not totally ramified. Then there is an $x \in D_{\ga}$ which does
not commute with $y \in D_{\de}$, so $x y y^{-1} - y^{-1} x y \neq
0$ proving $\ga \in \Supp ([D,D])$.
\end{proof}

\begin{example}
%
Let $\mathbb{H}$ be the real quaternion algebra. We saw in
Example~\ref{egofgrdivisionrings}(2) that $\mathbb H$ forms a $\Z_2
\times \Z_2$-graded division ring, where its centre is $\R$. Then
$\Supp (\mathbb{H}) = \Z_2 \times \Z_2$ and $\Supp (\R) = (0,0)$. We
can show that $\Supp ([\mathbb{H},\mathbb{H}]) = \{(1,0) ,(0,1),
(1,1)\}$.
\end{example}

Let $D$ be a graded division ring and let $F$ be a graded subfield
of $D$ which is contained in the centre of $D$. We know that $F_0 =
F \cap D_0$ is a field and $D_0$ is a division ring. The group of
invertible homogeneous elements of $D$, denoted by $D^{h*}$, is
equal to $D^h \mi 0$. Considering $D$ as a graded $F$-module, since
$F$ is a graded field,  there is a uniquely defined dimension
$[D:F]$ by  Theorem~\ref{gradedfree}. Note that $\Ga_F \subseteq
\Ga_D$, so $\Ga_F$ is a normal subgroup of $\Ga_D$.

The proposition below has been proven by Hwang, Wadsworth
\cite[Prop.~2.2]{hwalg} for  two graded fields $R \subseteq S$ with
a torsion-free abelian grade group.

\begin{prop}
Let $D$ be a graded division ring  and let $F$ be a graded subfield
of $D$ which is contained in the centre of $D$. Then
$$[D:F] = [D_0 : F_0] | \Ga_D : \Ga_F|.$$
\end{prop}

\begin{proof}
Let $\{ x_i \}_{i \in I}$ be a basis for $D_0$ over $F_0$. Consider
the cosets of $\Ga_D$ over $\Ga_F$ and take a transversal $\{ \de_j
\}_{j \in J}$ for these cosets, where $\de_j \in \Ga_D$. Take $\{y_j
\}_{j \in J} \subseteq D^{h*}$ such that $\deg (y_j ) = \de_j$ for
each $j$. We will show that $\{ x_i y_j \}$ is a basis for $D$ over
$F$.

Consider the map
\begin{align*}
\psi : D^{h*} & \lra \Ga_D / \Ga_F\\
d \; &\lmps \deg(d) + \Ga_F.
\end{align*}
This is a group homomorphism with kernel $D_0 F^{h*}$, since for any
$d \in \ker(\psi)$ there is some $f \in F^{h*}$ with $d f^{-1} \in
D_0$. Let $d \in D$ be arbitrary. Then $d = \sum_{\ga \in \Ga}
d_{\ga}$ where $d_{\ga} \in D_{\ga}$ and $\psi (d_\ga) = \ga + \Ga_F
= \de_j + \Ga_F$ for some $\de_j$ in the transversal of $\Ga_D$ over
$\Ga_F$. Then there is some $y_j$ with $\deg(y_j) = \de_j$ and
$d_\ga y_j^{-1} \in \ker (\psi)$. So $d_{\ga} y_j^{-1} = \sum_k a_k
g_k$ for $g_k \in F^{h*}$ and $a_k =\sum_i r_i^{(k)} x_i$ with
$r_i^{(k)} \in F_0$. It follows that $d$ can be written as an
$F$-linear combination of the elements of $\{x_i y_j \}$.

To show linear independence, suppose $\sum_{i=1}^n r_i x_i y_i = 0$
for $r_i \in F$. Since the homogeneous components of $D$ are
disjoint, we can take a homogeneous component of this sum, say
$\sum_{k=1}^m r_k x_k y_k$ where $\deg( r_k x_k y_k) = \al$. Then
$\deg(r_k) + \deg(y_k) = \al$ for all $k$, so all of the $y_k$ are
the same. This implies that $\sum_k r_k x_k =0$, where all of the
$r_k$ have the same degree. If $r_k=0$ for all $k$ then we are done.
Otherwise, for some $r_l \neq 0$, we have $\sum_k (r_l^{-1} r_k )
x_k = 0$. Since $\{x_i\}$ forms a basis for $D_0$ over $F_0$, this
implies $r_k = 0$ for all $k$.
\end{proof}

We include below a number of results involving homogeneous additive
commutators. The proofs follow in exactly the same way as the proofs
of the equivalent non-graded results (see \cite[\S 13]{lam}).

\begin{lemma}\label{gradditivecommutatorscentral}
Let $D$ be a graded division ring. If all of the homogeneous
additive commutators of $D$ are central, then $D$ is a graded field.
\end{lemma}

\begin{proof}
Let $y \in D^h$ be an arbitrary homogeneous element of $D$. Then by
assumption, $y$ commutes with all homogeneous additive commutators,
and it follows that $y$ commutes with all (non-homogeneous) additive
commutators of $D$. We will show $y \in Z(D)$.

Suppose $y \notin Z(D)$. Then there exists $x \in D^h$ such that
$[x,y] \neq 0$. We have $[x,xy] = x[x,y]$ with $[x,xy]$ and $[x,y]$
non-zero. Since $y$ commutes with $[x,xy]$ and $[x,y]$, it follows
that $(yx-xy)[x,y]=0$. As we assumed $[x,y] \neq 0$, it is a
non-zero homogeneous element of $D$, so is invertible. So $yx-xy
=0$; that is, $[x,y]=0$ contradicting our choice of $x$. It follows
that $D$ is a commutative graded division ring, as required.
\end{proof}

We include an alternative proof to
Lemma~\ref{gradditivecommutatorscentral} in
Section~\ref{sectionquotientdivisionrings}. This alternative proof
uses the relation between a graded division ring and its quotient
division ring combined with the non-graded result.

\begin{thm}
Let $D$ be a graded division ring with centre $F$. Then the smallest
graded division subring over $F$ generated by homogeneous additive
commutators is $D$.
\end{thm}

\begin{proof}
Let $E = \bigoplus_{\ga \in \Ga} E_{\ga}$ be the smallest graded
$F$-division subring containing all homogeneous additive
commutators. Then clearly we have $F \subseteq E \subseteq D$. To
show $D=E$, it is sufficient to show that $D^{h} \subseteq E^{h}$.
Let $x \in D^{h} \mi F^{h}$. Then there exists $y \in D^{h}$ such
that $xy \neq yx$. We have $[x, xy]=x[x,y]$ with $[x, xy], [x,y] \in
E^{h}$, since they are non-zero homogeneous additive commutators.
Then $x = [x, xy]\cdot [x,y]^{-1} \in E^{h}$, as required.
\end{proof}

\begin{prop}\label{additivecbhthm}
Let $K \subseteq D$ be graded division rings, with $[D,K] \subseteq
K$. If $\chr K \neq 2$, then $K \subseteq Z(D)$.

\end{prop}

\begin{proof}
First note that the condition $[D,K] \subseteq K$ is equivalent to
$[D^{h},K^{h}] \subseteq K^{h}$. Let $a \in D^{h} \mi K^{h}$ and $c
\in K^{h}$. We will show $ac=ca$. We have $[a, [a,c]] + [a^{2},c] =
2a[a,c] \in K$. If $[a,c] \neq 0$, then, since char$K \neq 2$, this
implies $a \in K$, contradicting our choice of $a$. Thus $[a,c] =0$.

Now let $b,c \in K^{h}$. Consider $a \in D^{h} \mi K^{h}$. Then $a,
ab \in D^{h} \mi K^{h}$ and we have shown that $[a,c], [ab,c]=0$.
Then $[b,c]= a^{-1} \cdot [ab,c] =0$. It follows that $K \subseteq
Z(D)$.
\end{proof}

The above results show that, in some aspects, the behaviour of
homogeneous additive commutators in graded division rings is similar
to that of additive commutators in division rings. However, this
analogy seems to fail for multiplicative commutators of graded
division rings. For example, in the setting of division rings the
Cartan-Brauer-Hua theorem is the multiplicative version of
Proposition~\ref{additivecbhthm} and its proof involves
multiplicative commutators. In the non-graded setting, the
Cartan-Brauer-Hua theorem states:
\begin{quote}
Let $D$ and $K$ be division rings, with $K \subseteq D$. Suppose
that $K^{\ast}$ is a normal subgroup of $D^{\ast}$ and $K \neq D$.
Then $K \subseteq Z(D)$.
\end{quote}
This theorem does not hold in the setting of graded division rings,
as is shown by the following counterexample.

\begin{example}
Let $D$ be a division ring and let $R= D[x, x^{-1}]$. Then we saw in
Example~\ref{egofgrdivisionrings}(1) that $R$ is a graded division
ring. Let $K = R_0$. Then $K^{\ast}$ is a normal subgroup of
$R^{\ast}$ and $K \neq R$. Choose $ a ,b \in D$ with $ab \neq ba$.
Then $ax^0 \in K$ and $bx^n \in R$, and we have $(ax^0 )(bx^n) =
abx^n \neq bax^n = (bx^n)(ax^0)$. So $ax^0 \notin Z(R)$, proving
that $K \nsubseteq Z(R)$.
\end{example}

Consider a homogeneous multiplicative commutator $x y x^{-1} y^{-1}$
in a graded division ring $D$. Since $x, y \in D^h$ and
$\deg(x^{-1}) = - \deg (x)$, we note that $\deg( x y x^{-1} y^{-1})
= 0$; that is, $x y x^{-1} y^{-1}$ is in the zero-homogeneous
component of $D$. This suggests that there are ``too few''
multiplicative commutators to affect the structure of the division
ring.


\section{Graded splitting fields} \label{sectiongradedsplitting}


For a graded division algebra $D$ over its centre $F$, we will show
that $D$ is split by any graded maximal subfield $L$ of $D$. This
allows us to construct the reduced characteristic polynomial of $D$
over $F$.

\begin{lemma} [Graded Schur's Lemma]
Let $R$ be a graded ring and let $M$ be a graded $R$-module. If $M$
is graded simple, then $\END_R(M)$ is a graded division ring.
\end{lemma}

\begin{proof}
We know from page~\pageref{endisagrring} that $\END_R (M)$ is a
graded ring. If $f$ is a nonzero homogeneous endomorphism, then
$\ker(f) \neq M$ and $\im(f) \neq 0$. Since $\ker(f)$ and $\im(f)$
are both graded submodules of $M$, which is graded simple, it
follows that $\ker(f) = 0$ and $\im(f) = M$. Thus $f$ is a graded
isomorphism, and hence is invertible.
\end{proof}

In the following theorem, we rewrite Rieffel's proof of Wedderburn's
Theorem in the graded setting (see \cite[Prop.~1.3(a)]{hwcor}). See
\cite[Thm.~2.10.10]{grrings} for a more general version of the
theorem.

\begin{thm} \label{wedderburnthm}
Let $A$ be a graded central simple algebra over a graded field $R$.
Then there is a graded division algebra $E$ over $R$ such that $A
\conggr M_n(E)(d)$ for some $(d) = (\de_1, \ldots, \de_n) \in
\Ga^n$.
\end{thm}

\begin{proof}
Take a minimal  nonzero homogeneous right ideal $L$ of $A$ (which
exists since $[A:R] < \infty$). Then $L$ is a graded simple right
$A$-module, as it has no nonzero proper graded right $A$-submodules.
Let $E= \End_A(L)$, where $\End_A (L) = \END_A (L)$ since $L$ is
finitely generated as a graded $A$-module. Then $E$ is a graded
division ring, by the graded Schur's Lemma, and $[E:R] \leq
[\End_R(L):R] < \infty$. Then $L$ is a graded free left $E$-module,
with a homogeneous base, say $b_1, \ldots, b_n$, of $L$ over $E$.
Then by Theorem~\ref{endgrisomatrixring}, $\End_E(L) \conggr
M_n(E)(d)$, where $(d) = (\deg(b_1), \ldots, \deg(b_n))$. Consider
the map
\begin{displaymath}
\begin{array}{rcrcl}
i:A & \lra & \End_E(L) && \nono \\
a & \longmapsto & i(a) : L &\ra& L \nono \\
&&x &\mapsto & xa \nono
\end{array}
\end{displaymath}
We will show that $i$ is a graded $R$-algebra isomorphism. Firstly
note that $i(a)$ is an element of $\End_E(L)$, since for all $e \in
E$, $x \in L$ we have
$$
(i(a))(ex) = (ex)a = (e(x))a = e(xa) =e \big((i(a))(x)\big).
$$
It can be easily shown that $i$ is a graded $R$-algebra
homomorphism. Since $A$ is graded simple, it follows that  $i$ is
injective. Note that $i(A)$ contains the identity element of
$\End_E(L)$. Then to prove surjectivity, it suffices to show that
$i(A)$ is a homogeneous right ideal of $\End_E (L)$.

For $y \in L$, let $\mal_y : L \ra L; x \mapsto yx$ so that $\mal_y
\in \End_A(L)$. Then for $f \in \End_E(L)^h$, $x \in L$ we have
$$
f(yx) = f(\mal_y (x) ) = \mal_y (f(x)) = y(f(x)).
$$
It follows that $( i(x) \cdot f)(y) = (f \circ i(x))(y) =
\big(i(f(x))\big)(y)$, which implies that $i(L)$ is a right ideal of
$\End_E(L)$. Also, $i(L) = \bigoplus_{\ga \in \Ga} \big(i(L) \cap
(\End_E(L))_{\ga}\big)$ giving that the ideal is homogeneous. The
two-sided homogeneous ideal of $A$ generated by $L$ is $ALA = AL$,
and since $A$ is graded simple, $AL = A$. Thus $i(A) = i(AL) =
i(A)i(L)$ is a homogeneous right ideal of $\End_E(L)$, as required.
\end{proof}

Let $D$ be a graded division algebra over its centre $F$ and let $L$
be any graded subfield of $D$ containing $F$. We define a grading on
$L[x]$ as follows. Let $\theta \in \Ga_D$ and let
\begin{align*}
L[x]^{\theta}  = \bigoplus_{\ga \in \Ga} L[x]_{\ga}, \;\;
\mathrm{where} \;\; L[x]_{\ga}  = \left\{ \sum a_{i} x^{i} : a_{i}
\in L^{h}, \deg (a_{i}) +i \theta = \ga \right\} .
\end{align*}
Then $L[x]^{\theta}$ forms a graded ring, and $x \in L[x]^{\theta}$
is homogeneous of degree $\theta$. Let $Z_D (L) = \{d \in D : dl =ld
\textrm{ for all } l \in L\}$ denote the centraliser of $L$ in $D$.
Then $Z_D(L)$ is a graded subring of $D$  and it is a graded
$L$-algebra. For any $c \in Z_D (L)^h$ of degree $\theta$, let
\begin{gather*}
L[ \hspace{.2 ex} c \hspace{.2 ex} ] = \big\{ f(c) : f (x) \in
L[x]^{\theta} \big\}.
\end{gather*}
Then $L[ \hspace{.1 ex} c \hspace{.1 ex} ]$ forms a graded ring with
$L \subseteq L[ \hspace{.1 ex} c \hspace{.1 ex} ]$, and we note that
$L[ \hspace{.1 ex} c \hspace{.1 ex} ]$ is commutative since $c \in
Z_D (L)^h$.
The map $L[ \hspace{.1 ex} x \hspace{.1 ex} ]^{\theta} \ra L[
\hspace{.1 ex} c \hspace{.1 ex} ];$ $f(x) \mapsto f(c)$ is a graded
ring homomorphism.

Further we will show that $L[ \hspace{.1 ex} c \hspace{.1 ex} ]$ is
in fact  a graded field. Let $a \in L[ \hspace{.1 ex} c \hspace{.1
ex} ]_{\ga}$ be a non-zero element, and consider the $\ga$-shifted
$L$-module $L[\hspace{.1 ex} c \hspace{.1 ex}](\ga)$. The map
$\mal_a : L[ \hspace{.1 ex} c \hspace{.1 ex}] \ra L[ \hspace{.1 ex}
c \hspace{.1 ex}] (\ga)$; $l \mps al$ is a graded $L$-module
homomorphism, which is injective since $a$ is invertible in $D$.
Then $\dim_L (\im (\mal_a)) = \dim _L (L[ \hspace{.1 ex} c
\hspace{.1 ex} ])$. We will show that $\dim _L (L[ \hspace{.1 ex} c
\hspace{.1 ex} ]) < \infty$.

Since $c \in D$ and $[D:F] < \infty$, we have that $c$ is algebraic
over $F$, and thus is algebraic over $L$. So it has a minimal
polynomial $h(x) = l_0 + l_1 x + \cdots + l_k x^k$, and the set $\{
1 , c, c^2 , \cdots , c^{k-1} \}$ generates $L[c]$ over $L$. If they
are not all linearly independent, we can write $1$ as an $L$-linear
combination of the others and remove it from the set, leaving a set
which still generates $L[c]$ over $L$. Repeating this process will
give a linearly independent spanning set for $L[c]$ over $L$.

So $\dim _L (L[ \hspace{.1 ex} c \hspace{.1 ex} ]) < \infty$, and
$\mal_a$ is surjective by dimension count. Then there is a graded
$L$-module homomorphism $\psi$ which is the inverse of $\mal_a$, and
$\psi (1_L)$ is the inverse of $a$. Since $L[ \hspace{.1 ex} c
\hspace{.1 ex} ]$ is also commutative, it is a graded field.

The following result is in \cite[p.~40]{draxl} in the non-graded
setting.

\begin{thm} \label{zdlthm}
Let $D$ be a graded division algebra over a graded field $F$, and
let $L$ be a graded subfield of $D$. Then $L$ is a graded maximal
subfield if and only if $Z_D(L)=L$.
\end{thm}

\begin{proof}
If $Z_D (L) = L$, then for any graded subfield $L'$ of $D$
containing $L$ we have $L' \subseteq Z_D (L)$. So $L$ is a graded
maximal subfield. Conversely, assume $L$ is graded maximal. Then for
any homogeneous $c \in Z_D(L)$, $L[\hspace{.1 ex} c \hspace{.1 ex}]$
forms a graded field. By the maximality of $L$, we must have $L[
\hspace{.1 ex} c \hspace{.1 ex}] = L$ and so $c \in L$. Then it
follows that $Z_D(L) = L$.
\end{proof}

\begin{cor} \label{grsplitcor}
Let $D$ be a graded division algebra with centre $F$ and let $L$ be
a graded maximal subfield of $D$. Then
$$
D \otimes_F L \conggr M_n(L)(d)
$$
for some $n \in \mathbb N$ and some $d= (\de_1, \ldots, \de_n) \in
\Ga^n$.
\end{cor}

\begin{proof}
As graded $F$-modules, we have $D \otimes_F L \conggr L \otimes_F
D$. Since $D$ is a graded division algebra over the graded field
$F$, and $L$ is graded simple,
Theorems~\ref{tensorgradedsimple},\ref{tensorgradedcentral} give
that $L \otimes_F D$ is graded central simple over $Z(L) = L$. We
have that $D$ is a graded simple right module over $L \otimes_F D$,
with the right action $x(l \otimes d) = lxd$ for $d, x \in D$, $l
\in L$. Then by Theorem \ref{wedderburnthm}, $E:= \End_{L \otimes_F
D}(D)$ is a graded division algebra over $L$, such that $L \otimes_F
D \conggr M_n(E)(d)$ for some $d= (\de_1, \ldots, \de_n) \in \Ga^n$.
From Theorem \ref{zdlthm}, $Z_D(L) = L$, so it remains to show
$Z_D(L) \conggr E$. Define
\begin{displaymath}
\begin{array}{rcr}
\psi : Z_D(L) & \ra & \End_{L \otimes_F D}(D) \;\;\;\;\; \nono \\
d & \mapsto & \psi(d) :D \; \ra \; D \; \nono \\
&& x \; \mapsto \; d x \nono
\end{array}
\end{displaymath}
It can be easily shown that $\psi$ is a graded $L$-algebra
homomorphism, which is injective since $L$ is graded simple as a
graded ring. Let $f \in {\big(\End_{L \otimes_F D}(D)\big)}_{\ga}$
be a homogeneous map. Then $f(1) \in D_{\ga}$ and since $f(1) \ell =
\ell f(1)$ for all $\ell \in L^h$, we have $f(1) \in Z_D(L)$. For $x
\in D$,
$$
f(x) = f(1 \cdot (1\otimes x)) = f(1) (1 \otimes x)= f(1) x=
\big(\psi(f(1))\big)(x).
$$
So $f = \psi(f(1))$, proving that $\psi$ is surjective. It follows
that $M_n(E)(d) \conggr M_n(L)(d)$, completing the proof.
\end{proof}

Suppose $D$ is a graded division algebra over a graded field $F$. We
will show that there exists a graded maximal subfield of $D$. For
any $a \in D^h \mi F^h$, we have shown above that $F[a]$ is a graded
field, with $F \subseteq F[a] \subseteq D$. Consider
$$
X = \{ L : L \text{ is a graded subfield of } D \text{ with } F
\subsetneqq L \}.
$$
This is a non-empty set since $F[a] \in X$ and it is partially
ordered with inclusion. Every chain $L_1 \subseteq L_2 \subseteq
\ldots$ in $X$ has an upper bound $\bigcup L_i \in X$. By Zorn's
Lemma, $X$ has a maximal element, so there is a graded maximal
subfield of $D$.

Then Corollary~\ref{grsplitcor} shows that a graded maximal subfield
$L$ of $D$ splits $D$; that is, $j: D \otimes L \conggr M_n (L)
(d)$. As in the non-graded setting, for an element $d \in D$ we
define the reduced characteristic polynomial as
\begin{align*}
\chr_{D/F} (d,x) & = \det \big(x I_n - j (d \otimes 1) \big) \\
& =  x^{n} - \Trd_{D}(d) x^{n-1} + \cdots + (-1)^{n} \Nrd_{D}(d),
\end{align*}
where $\Trd_{D}(d) = \mathrm{trace} \big(j(d \otimes 1)\big)$ is the
reduced trace of $d$ and $\Nrd_{D} (d) = \det \big(j(d \otimes
1)\big)$ is the reduced norm.

Since $L$ is a graded module over a graded field $F$, by Proposition
\ref{gradedfree}, it is graded free and therefore free over $F$. It
follows that $L$ is faithfully flat over $F$, where $F$ is a ring
and $D$ and $L$ are $F$-algebras, so we can apply
\cite[Lemma~III.1.2.1]{knus}. This shows that the reduced
characteristic polynomial lies in $F[x]$ and it is independent of
the choice of $j$ and $L$. We note that the reduced norm and reduced
trace satisfy the following properties.

\begin{cor} \label{nrdandtrdproperties}
Reduced norm and trace satisfy the following rules:
\begin{align*}
\Nrd_D (ab)&= \Nrd_D (a) \Nrd_D (b) & \Trd_D (a+b)& =  \Trd_D (a)+ \Trd_D (b)\\
\Nrd_D( r a) &= r^{n} \Nrd_D(a) & \Trd_D( r a)&= r \Trd_D(a)\\
&& \Trd_D (ab)&= \Trd_D(ba)
\end{align*}
for all $a,b \in D$ and $r \in F$.
%
%
\end{cor}


\begin{proof}
This follows immediately from the properties of determinant and
trace in a matrix.
\end{proof}

It follows from the above corollary that $\Trd_D : D \ra F$ is an
$F$-module homomorphism. Since all three maps $D  \ra D \otimes_F L
\ra  M_n (L)(d) \stackrel{\mathrm{trace}}{\lra} F$ are graded maps,
we have that $\Trd_D$ is a graded $F$-module homomorphism.

Note that a graded division algebra $D$ with centre $F$ is an
Azumaya algebra by Theorem~\ref{gcsaazumayaalgebra}. Since the
dimension of an Azumaya algebra is a square, it follows that $[D:F]$
is also a square number.

\begin{prop} \label{trdsurjective}
Let $D$ be a graded division algebra over its centre $F$. Then
$\Trd_{D} : D \ra F$ is surjective.

\end{prop}

\begin{proof}

Suppose $\Trd_D$ is not surjective. Since $\im (\Trd_D )$ is a
graded module over the graded field $F$ with $\dim_F (\im (\Trd_D))
\leq \dim_F (F) = 1$, then $\dim (\im (\Trd_D)) = 0$ and so $\Trd_D$
is the zero map. Let $\{x_1 , \ldots , x_{n^2}\}$ be a homogeneous
basis for $D$ over $F$ and let $L$ be a graded maximal subfield of
$D$. Then by Corollary \ref{grsplitcor}, $f : D \otimes_F L \conggr
M_n(L)(d)$ for some $(d) \in \Ga^n$ and it is known that $\{x_1
\otimes 1 , \ldots , x_{n^2} \otimes 1 \}$ forms a homogeneous basis
for $D \otimes_F L$ over $L$. Since $f$ is a graded isomorphism, $\{
f(x_i \otimes 1) : 1 \leq i \leq n^2\}$ forms a homogeneous basis of
$M_n(L)(d)$ over $L$. By definition $\Trd_{D} (d_i) = \mathrm{tr} (f
(d_i \otimes 1))$, which equals zero since $\Trd_D$ is the zero map.
That is, the trace is the zero function on $M_n(L)(d)$, which is
clearly a contradiction.
\end{proof}



\section{Some results in the non-graded setting} \label{sectionsomenongradedresults}

We recall here some results from the non-graded setting, which will
be used in the next section.

\begin{lemma} \label{commutatorlemma}
Let $R$ be a division ring. If all of the additive commutators of
$R$ are central, then $R$ is a field.
\end{lemma}

\begin{proof}
See \cite[Cor.~13.5]{lam}.
\end{proof}

\begin{thm} \label{trda=na+danongrversion}
Let $D$ be a division algebra over its centre $F$ of index $n$. Then
for $a \in D$, $\Trd_{D} (a) = na +d_a$ where $d_a \in [D,D]$.
\end{thm}

\begin{proof}
Let $a \in D$ with minimal polynomial $f (x) \in F[x]$ of degree
$m$. Then by \cite[p.~124, Ex.~9.1]{rei}, we have
$$
f(x)^{n/m} = x^n - \Trd_{D} (a) x^{n-1} + \cdots + (-1)^n \Nrd_{D}
(a).
$$
where the right hand side of this equality is the reduced
characteristic polynomial of $a$.  Wedderburn's Factorisation
Theorem \cite[Thm.~16.9]{lam1} says $f(x) = (x- d_1 a d_1^{-1})
\cdots (x- d_m a d_m^{-1})$ for $d_1 , \ldots , d_m \in D$.
Combining these, we have
\begin{align*}
\Trd_{D} (a) &= \frac{n}{m} \big( d_1 a d_1^{-1} + \cdots + d_m a
d_m^{-1}\big)\\
& = \frac{n}{m} \big( ma +( d_1 a d_1^{-1} - a d_1^{-1} d_1)  +
\cdots + ( d_m a d_m^{-1}-  a d_m^{-1} d_m) \big)\\
&= na + d_a \;\;\;\; \text{ where } d_a \in [D,D],
\end{align*} as required.
\end{proof}

Let $R$ be a commutative Noetherian ring. The dimension of the
maximal ideal space of $R$ is defined to be the supremum on the
lengths of properly descending chains of irreducible closed sets.

\begin{thm}
Let $R$ be a commutative Noetherian ring and let $A$ be an Azumaya
algebra over $R$. Then every element of $A$ of reduced trace zero is
a sum of at most $2d +2$ additive commutators, where $d$ is the
dimension of the maximal ideal space of $R$.
\end{thm}

\begin{proof}
See \cite[Thm.~5.3.1]{rossetthesis}.
\end{proof}

\begin{cor}\label{kertrd=dd}
Let $D$ be a graded division algebra over its centre $F$, which is
Noetherian as a ring. Then $\ker (\Trd_{D}) = [D,D]$.
\end{cor}

\begin{proof}
For any $xy- yx \in [D,D]$, we have $\Trd_D (xy-yx) = 0$ by
Corollary~\ref{nrdandtrdproperties}. The reverse containment follows
immediately from the above theorem, since by
Theorem~\ref{gcsaazumayaalgebra}, $D$ is an Azumaya algebra over
$F$.
\end{proof}

\begin{remark}\label{remarkonkertrd}
Let $D$ be a graded division algebra over its centre $F$, which is
Noetherian as a ring. Since $\ker (\Trd_D ) = [D,D]$ by
Corollary~\ref{kertrd=dd} and $\Trd_D$ is surjective by
Proposition~\ref{trdsurjective}, the First Isomorphism Theorem says
that $D/ [D,D] \conggr F$ as graded $F$-modules. So $\dim_F (D/
[D,D]) = \dim_F F= 1$. By Proposition~\ref{grdimensionprop}, $\dim_F
([D,D]) + 1 = \dim_F (D) < \infty$.
\end{remark}

We recall here the definitions of a totally ordered group and a
torsion-free group. Let $(\Ga, +)$ be a group. A partial order is a
binary relation $\leq$ on $\Ga$ which is reflexive, antisymmetric
and transitive. The order relation is translation invariant if for
all $a,b,c \in \Ga$, $a \leq b$ implies $a + c \leq b + c$ and $c +
a \leq c + b$. A partially ordered group is a group $\Ga$ equipped
with a partial order $\leq$ which is translation invariant. If $\Ga$
has a partial order, then two distinct elements $a, b \in \Ga$ are
said to be comparable if $a \leq b$ or $b \leq a$. If $\Ga$ is a
partially ordered group in which every two elements of $\Ga$ are
comparable, then $\Ga$ is called a \emph{totally ordered
group}\index{totally ordered group}\index{group!totally ordered}.

For a group $\Ga$, an element $a$ of $\Ga$ is called a torsion
element there is a positive integer $n$ such that $a^n = e$. If the
only torsion element is the identity element, then the group $\Ga$
is called \emph{torsion-free}\index{torsion-free
group}\index{group!torsion-free}. By \cite{levi}, an abelian group
can be equipped with a total order if and only if it is
torsion-free.



\section{Quotient division rings}
\label{sectionquotientdivisionrings}

Throughout this section, $\Ga$ is a torsion-free abelian group and
all graded objects are $\Ga$-graded. In this setting, graded
division rings have no zero divisors,  and similarly for graded
fields. This follows since we can choose a total order for $\Ga$. So
for a graded division ring $D =\bigoplus_{\ga \in \Ga} D_{\ga}$ and
two non-zero elements $a, b \in D$, we can write
$$
a = a_\ga + \textrm{ terms of higher degree \;\; and \;\; } b =
b_\de + \textrm{ terms of higher degree}.
$$
Then $ab = a_\ga b_\de +$ terms of higher degree, so we have that
$ab$ is non-zero. Thus the group of units of $D$ is $D^* = D^h \mi
0$. Similarly, a graded field $F$ is an integral domain with group
of units $F^* = F^h \mi 0$. This allows us to construct $QF = (F \mi
0)^{-1} F$, the \emph{quotient field}\index{quotient field} of $F$,
which is clearly a field and an $F$-module. For a graded division
algebra $D$ with centre $F$, we define the \emph{quotient division
ring}\index{quotient division ring} of $D$ to be $QD = QF \otimes_F
D$. We observe some properties of $QD$ in the proposition below,
including in part (5) that it is a division ring.

\begin{prop}
Let $D$ be a graded division algebra with centre $F$, and let $QF$
and $QD$ be as defined above. Then the following properties hold:
\begin{enumerate}
\item $QD$ is an algebra over $QF$;

\item $D \ra QD$; $d \mps 1 \otimes d$ is injective;

\item $QD$ has no zero divisors;

\item $[QD :QF] = [D :F]$;

\item $QD$ is a division ring;

\item $QD \cong (F \mi 0)^{-1} D$; $(f_1/ f_2) \otimes d \mps (f_1
d)/f_2$;

\item The elements of $QD$ can be written as $d/f$, $d \in D, f \in F$.
\end{enumerate}
\end{prop}

\begin{proof} (1): This follows easily.

\vspace*{3pt}

(2): To show $D \ra QD$ is injective, we first show that $F \ra QF$
is injective. Consider $F \ra QF$; $f \mps f/1$. If $f/1 =0$, there
is $s \in F \mi 0$ with $sf =0$. Since $F$ is an integral domain, it
follows that $f=0$ and the map is injective. As $D$ is graded free
over $F$, and thus flat over $F$, the required result follows.

\vspace*{3pt}

(3): Let $x \in  QF \otimes D$, say $x = \sum (f_i / f'_i) \otimes
d_i$. Then
\begin{align*}
x  & = (f_1 / f'_1) \otimes d_1 + \cdots + (f_k / f'_k) \otimes d_k
\\
&= (1/ f) \otimes f_1 f'_2 \cdots f'_k d_1 + \cdots + (1/ f) \otimes
f_k f'_1 \cdots f'_{k-1} d_k \\
&= (1/f) \otimes d,
\end{align*}
where $f =f'_1 f'_2 \cdots f'_k$ and $d \in D$. So for any arbitrary
element $x$ of $QF \otimes D$, we can write $x = (1/f) \otimes d$,
for $f \in F, d \in D$. Now let $x = (1/f) \otimes d$ and $y  =
(1/f') \otimes d'$ be arbitrary elements of $QF \otimes D$ with $xy
=0$. Then $1/ (f f') \otimes d d' =0$, so
$$
(f f' \otimes 1)\left( \frac{1}{f f'} \otimes d d' \right) =0.
$$
Thus $1 \otimes d d'=0$ and since $D \ra QF \otimes D$ is injective,
we have $dd'=0$. As $D$ has no zero divisors, it follows that $d=0$
or $d'=0$; that is, $x=0$ or $y=0$ as required.

\vspace*{3pt}

(4): Since $D$ is free over $F$ and $QD = QF \otimes_F D$, we have
$$
[D \otimes_F QF : F \otimes_F QF] =[D: F];
$$
that is, $[QD :QF] = [D :F]$.

\vspace*{3pt}

(5): We will show that any domain which is finite dimensional over a
field is a division ring. Since $QF$ is a field, using (3) and (4),
it follows that $QD$ is a division ring.

Suppose $A$ is a domain, $F$ is a field, $Z(A)=F$ and $[A:F] =n <
\infty$. Let $a \in A \mi 0$ and consider $1 , a, a^2, \ldots ,
a^n$. Then there are $r_i \in F$, not all zero, with $r_0 + r_1 a +
\cdots + r_n a^n = 0$. If $r_j$ is the first non-zero element of
$\{r_0 , \ldots , r_n\}$, then $r_j +r_{j+1} a + \cdots + r_n
a^{n-j} = 0$. As $a \neq 0$ and $F$ is a field, we have $ a(r_{j+1}
+ \cdots + r_n a^{n-j-1}) ( -r_j)^{-1}= 1$, so $a$ is invertible.

\vspace*{3pt}

(6): Follows immediately from \cite[Prop.~6.55]{magurn}. Note that
they are isomorphic as $QF$-algebras.

\vspace*{3pt}

(7): We know from part~(3) that if $x \in QD$, then we can write $x
= (1/f) \otimes d$. From the isomorphism in~(6), $QF \otimes D \cong
(F \mi 0)^{-1} D$; $(1/f) \otimes d \mps d/f$. So we can consider
the elements of $QD$ as $d/f$ for $d \in D, f \in F \mi 0$.
\end{proof}

Note that we have $QF \otimes_F D \cong D \otimes_F QF$ as
$QF$-algebras, so we will use the terms interchangeably. We include
here an alternative proof of
Lemma~\ref{gradditivecommutatorscentral}, due to Hazrat, which
follows from the non-graded result by using the quotient division
ring.

\begin{proof}[Alternative proof of Lemma~\ref{gradditivecommutatorscentral}.]
%
Let $y \in D^h$ be an element which commutes with homogeneous
additive commutators of $D$. Then it follows that $y$ commutes with
all (non-homogeneous) commutators of $D$. Consider $[x_1 , x_2]$
where $x_1 , x_2 \in QD$, with $x_1 = d_1 / f_1 $ and $x_2 = d_2 /
f_2$ for $d_1 , d_2 \in D, f_1 , f_2 \in F$. Then $y[x_1 , x_2 ] =
y([d_1 , d_2 ] / f_1 f_2) = y[d_1 , d_2 ] / f_1 f_2 = [d_1 , d_2 ]y
/ f_1 f_2 = ([d_1 , d_2 ] / f_1 f_2)y = [x_1 , x_2 ]y$. So $y$
commutes with all commutators of $QD$, a division ring. Using Lemma
\ref{commutatorlemma}, $y \in QF$.

We will show that $D^h \cap QF \subseteq F^h$. If $x \in D^h \cap
QF$, then $x = d/1 = f'/f$ for $d \in D^h$, $f,f' \in F$ with $f
\neq 0$. So there exists $t \in F \mi 0$ with $tfd = tf'$, so that
$d=f^{-1} f'$. Thus $x = (f^{-1} f') /1 \in F$ with $\deg (d) =
\deg( f^{-1} f')$; that is, $x \in F^h$. This proves that $y \in
F^h$. It follows that all elements of $D$ are commutative,
completing the proof.
\end{proof}

\begin{prop}\label{redcharpolygrversion}
Let $D$ be a graded division algebra over its centre $F$. Then for
$a \in D^h$, the reduced characteristic polynomial of $a$ with
respect to $D$ over $F$ coincides with the reduced characteristic
polynomial of $a \otimes 1$ with respect to $QD$ over $QF$.
\end{prop}

\begin{proof}
Let $L$ be a splitting field of $QD$ over $QF$, which exists as $QD$
is a finite dimensional division algebra over $QF$. So $i : QD
\otimes_{QF} L \cong M_n (L)$, and therefore $L$ is a splitting
field of $D$ over $F$, since
$$
j : D \otimes_F L \cong D \otimes_{QF} QF \otimes_F L \cong M_n (L).
$$
Then for $a \in D$,
\begin{align*}
\chr_{QD/QF} (a \otimes 1) & = \det \big( x I_n - i( (a \otimes
1_{QF}) \otimes 1_L)\big) \\
& = \det \big( x I_n - j( a \otimes 1_L)\big) \\
& = \chr_{D/F} (a).
\end{align*}
So in particular, we have $\Trd_{QD} (a \otimes 1) = \Trd_{D} (a)$
and $\Nrd_{QD} (a \otimes 1) = \Nrd_{D} (a)$.
\end{proof}

\begin{prop}\label{dd=qdqdcapd}
Let $D$ be a graded division algebra over its centre $F$, which is
Noetherian as a ring.  Then $[D,D] = [QD, QD] \cap D$.
\end{prop}

\begin{proof}
We observed in Remark~\ref{remarkonkertrd} that $\dim_F ([D,D]) + 1
= \dim_F (D) < \infty$. For any $\sum_i x_i y_i -y_i x_i \in [D,D]$,
as $D \ra QD$ is injective,
$$
1 \otimes \sum_i x_i y_i -y_i x_i = \sum_i (1 \otimes x_i)(1 \otimes
y_i) - (1 \otimes y_i)(1 \otimes x_i) \in [QD,QD] \cap D.
$$
So we have $[D,D] \subseteq [QD,QD] \cap D \subseteq D$. Here $D
\nsubseteq [QD, QD]$, so $[QD,QD] \cap D \neq D$. Thus $[D,D] =
[QD,QD] \cap D$.
\end{proof}

\begin{cor} \label{trda=na+dagradedversion}
Let $D$ be a graded division algebra with centre $F$ of index $n$,
where $F$ is Noetherian as a ring. Then for each $a \in D$,
$\Trd_{D} (a) = na +d_a$ for some $d_a \in [D,D]$.

\end{cor}

\begin{proof}
Let $a \in D$, where $\Trd_{QD} (a \otimes 1) = \Trd_{D} (a)$ by
Proposition~\ref{redcharpolygrversion}. Since $QD$ is a division
ring, by Theorem~\ref{trda=na+danongrversion}, $\Trd_{D} (a) = n (a
\otimes 1) +c$ where $c \in [QD,QD]$. We know $\Trd_{D} (a) \in F$.
Since $D \ra QD$ is injective, we can consider $n (a \otimes 1)$ as
an element of $D$. So $c  = \Trd_{D} (a) - na \in [QD, QD] \cap D =
[D,D]$ by Proposition~\ref{dd=qdqdcapd}, as required.
\end{proof}

The proof of the above corollary shows that using the quotient
division ring and the result in the non-graded setting, the graded
result follows immediately. We note that this proof only holds for
division algebras with a torsion-free grade group.

\begin{cor}\label{d/d,dcongqd/qd,qd}
Let $D$ be a graded division algebra over its centre $F$, which is
Noetherian as a ring. Then
$$
\frac{D}{[D,D]} \otimes_F QF \cong \frac{QD}{[QD,QD]}.
$$
\end{cor}

\begin{proof}
From the above results we have $D /[D,D] \conggr F$ as graded
$F$-modules. So
$$
\frac{D}{[D,D]} \otimes_F QF \cong QF \cong \frac{QD}{[QD,QD]},
$$
where the second isomorphism comes from the non-graded versions of
the above results.
\end{proof}

\begin{remark}
Let $D$ be a graded division algebra over its centre $F$. By
definition $\SK[1] (D) = D^{(1)} / D'$, where $D^{(1)} = \ker
(\Nrd_D)$ and $D'$ is the commutator subgroup of the multiplicative
group $D^*$.  We remark that in \cite[Thm.~5.7]{hazwadsk1}, they
have shown that $\SK[1] (D) \cong \SK[1] (QD)$.
Corollary~\ref{d/d,dcongqd/qd,qd} is similar to this result, where
in the above corollary we are considering additive commutators
instead of multiplicative ones.
\end{remark}


\clearpage \addcontentsline{toc}{chapter}{Index} \printindex

\end{document}